\documentclass[10pt,oneside,a4paper]{article}
\usepackage{longtable}

\usepackage{titlesec}
\usepackage{authblk}
\titleformat{\subsection}[runin]{\normalfont\bfseries}{\thesubsection.}{.5em}{}[.~ ]
\titlespacing{\subsection}{0pt}{1.5ex plus .1ex minus .2ex}{0pt}
\usepackage{graphicx}
\usepackage{subcaption}
\usepackage{wrapfig}
\usepackage{amsfonts,amssymb,amsmath,amsthm,mathtools}
\usepackage{enumerate}
\usepackage{url}
\usepackage[mathcal,mathscr]{euscript}
\usepackage[dvipsnames]{xcolor}
\usepackage{hyperref}
\hypersetup{colorlinks=true,linkcolor=Magenta,citecolor=PineGreen}
\usepackage{float}
\usepackage{tikz-cd}

\theoremstyle{plain}
\newtheorem{thm}[equation]{Theorem}

\newtheorem{lemma}[equation]{Lemma}

\newtheorem{prop}[equation]{Proposition}

\theoremstyle{definition}

\newtheorem{defi}[equation]{Definition}

\newtheorem{rmk}[equation]{Remark}
\newcommand{\dd}{\mathrm{d}}
\newcommand{\B}{\mathrm{B}}
\newcommand{\G}{\mathrm{\mathrm{G}}}

\newcommand{\tr}{\mathrm{\mathrm{tr}}}
\newcommand{\T}{\mathrm{\mathrm{T}}}
\newcommand{\ta}{\mathrm{\mathrm{ta}}}

\newcommand{\Arg}{\mathrm{\mathrm{Arg}}}

\newcommand{\SU}{\mathrm{\mathrm{SU}}(2,1)}
\newcommand{\PU}{\mathrm{\mathrm{PU}}(2,1)}
\newcommand{\BV}{{\mathrm B}(V)}
\newcommand{\SV}{{\mathrm S}(V)}
\newcommand{\EV}{{\mathrm E}(V)}
\newcommand{\real}{\mathrm{\mathrm{Re}}}
\newcommand{\imag}{\mathrm{\mathrm{Im}}}
\newcommand{\PP}{\mathbb{P}}
\newcommand{\HH}{\mathbb{H}}
\newcommand{\CC}{\mathbb{C}}
\newcommand{\ZZ}{\mathbb{Z}}
\newcommand{\SP}{\mathbb{S}}
\newcommand{\RR}{\mathbb{R}}
\newcommand{\QQ}{\mathbb{Q}}
\newcommand{\DD}{\mathbb{D}}

\newcommand{\SL}{\mathrm{\mathrm{SL}}}
\newcommand{\Lin}{\mathrm{Lin}}

\title{Quotients of the holomorphic\/ $2$-ball and the turnover}
\author[$\dagger$]{Hugo C.~Bot\'os (corresponding author)\footnote{Supported by S\~ao Paulo Research Foundation (FAPESP).}} 
\affil[$\dagger$]{\small{Departamento de Matem\'atica Aplicada, IME, Universidade de S\~ao Paulo, S\~ao Paulo, Brasil\authorcr hugocbotos@usp.br}}
\author[$\dagger\dagger$]{Carlos H.~Grossi}
\affil[$\dagger\dagger$]{\small{Departamento de Matem\'atica, ICMC, Universidade de S\~ao Paulo, S\~ao Carlos, Brasil\authorcr grossi@icmc.usp.br}}
\date{}

\begin{document}

\maketitle

\begin{abstract}
We construct two-dimensional families of complex hyperbolic structures on disc orbibundles over the sphere with three cone points. This contrasts with the previously known examples of the same type, which are locally rigid. In particular, we obtain examples of complex hyperbolic structures on trivial and cotangent disc bundles over closed Riemann surfaces.
\end{abstract}

\smallskip


\smallskip


\section{Introduction}

In this paper, we deal with complex hyperbolic Kleinian groups in complex dimension $2$, that is, discrete holomorphic isometry groups of the complex hyperbolic plane $\HH^2_\CC$. There are few known examples of such groups and a comprehensive survey can be found in \cite{kap2}. 

The complex hyperbolic plane is the holomorphic ball $$\HH_\CC^2\coloneq \{[z_0:z_1:z_2] \in \PP_{\CC}^2: -|z_0|^2+|z_1|^2+|z_2|^2<0\}$$ in the complex projective plane $\PP_\CC^2$. For the Hermitian form $\langle z,w \rangle\coloneq -z_0\overline{w_0}+z_1\overline{w_1}+z_2\overline{w_2}$ on $\CC^3$ we obtain the group $$\SU\coloneq \{A \in \SL(3,\CC): \langle Az,Aw \rangle=\langle z,w \rangle\}$$ of unitary transformations in $\SL(3,\CC)$ and its projectivization $$\PU = \SU/\{1,\exp(2\pi i/3),\exp(-2\pi i/3)\}.$$ The group $\PU$ is the group of biholomorphisms of the holomorphic ball $\HH_\CC^2$. The complex hyperbolic plane is a Kähler manifold, its symplectic form is denoted by $\omega$, and the discrete subgroups of $\PU$ are called complex hyperbolic Kleinian groups. For details, see Section \ref{section preliminaries}.

We are interested in constructing disc orbibundles over closed hyperbolic $2$-orbifolds with the total space being complex hyperbolic. 
\begin{rmk} \label{remark disc orbibundle} Consider a co-compact Fuchsian group $\Gamma$ and an open disc $\DD$. If $\Gamma$ acts on $\HH_\CC^1 \times \DD$, where $\HH_\CC^1$ is the Poincaré disc and for each $g \in \Gamma$ we have $g(z,f) = (gz, a(z,g)f)$ and $(z,f) \mapsto a(z,g)f$ smooth, then the map $\zeta: (\HH_\CC^1 \times \DD)/\Gamma \mapsto \HH_\CC^1/\Gamma$ is a good disc orbibundle, which we call orbigoodle (see Definition \ref{definition orbibundle}). We say that such disc orbibundle is complex hyperbolic if there exists a complex hyperbolic Kleinian group $\Gamma'$ such that $(\HH_\CC^1 \times \DD)/\Gamma$ is diffeomorphic to $\HH_\CC^2/\Gamma'$, implying that $\Gamma$ and $\Gamma'$ are isomorphic. Thus, for each complex hyperbolic disc orbibundle constructed, we obtain a discrete faithful representation in $\hom(\Gamma, \PU)$. Considering complex hyperbolic disc orbibundles up to holomorphic isomorphisms, we obtain points on the $\PU$-character variety $\hom(\Gamma, \PU)/\PU$ for $\Gamma$, the space of representations modulo $\PU$-conjugations. 

A disc orbibundle $(\HH_\CC^1 \times \DD)/\Gamma \to \HH_\CC^1/\Gamma$ has two natural topological invariants. If we think of $\HH_\CC^1/\Gamma$ as a section $S$ embedded in $(\HH_\CC^1 \times \DD)/\Gamma$, then we have the tangent orbibundle $TS \to S$ and the normal orbibundle $NS \to S$. The Euler number of the vector orbibundle is the integral of its Euler class. We obtain from $TS$ the Euler characteristic $\chi$ of $S\simeq \HH_\CC^1/\Gamma$ and from $NS$ the Euler number $e$ of the disc orbibundle. The computation of the Euler number in this work is done via an adaptation of the Poincaré-Hopf theorem (see Definition \ref{Def euler number}) developed in work \cite{bot}. Whenever the total space admits a complex hyperbolic structure, the Toledo invariant $\tau\coloneq  \frac{4}{2\pi i} \int_S \omega$ is a third discrete invariant (see Definition \ref{def toledo}).
\end{rmk}

To the authors' knowledge, previously known examples of complex hyperbolic disc orbibundles over closed hyperbolic $2$-orbifolds are found in \cite{GKL}, \cite{discbundles}, and \cite{agu}.

The complex hyperbolic Kleinian groups we construct here resemble those in \cite{discbundles} as they arise from discrete faithful representations of the turnover group
$$G(n_1,n_2,n_3)\coloneq \langle g_1,g_2,g_3\mid g_1^{n_1}=g_2^{n_2}=g_3^{n_3}=1{\text{ \rm and }}g_3g_2g_1=1\rangle$$
in the group $\PU$ of biholomorphisms of $\HH_\CC^2$. 
The turnover group $G(n_1,n_2,n_3)$ is the fundamental group for $\SP^2(n_1,n_2,n_3)$, the $2$-sphere with three conic points with angles $\frac{2\pi}{n_1}$, $\frac{2\pi}{n_2}$, $\frac{2\pi}{n_3}$, and it is a Fuchsian group when the Euler characteristic $\chi(\SP^2(n_1,n_2,n_3))=-1+\frac{1}{n_1}+\frac{1}{n_2}+\frac{1}{n_3}$ is negative (see Section \ref{turnover definition}).
These discrete faithful representations lead to orbibundles over hyperbolic spheres with three cone points or, up to finite cover, to disc bundles over closed Riemann surfaces, because every co-compact Fuchsian group admits a finite index torsion-free subgroup due to Selberg's Lemma.

As mentioned, to construct orbibundles we study representations $G(n_1,n_2,n_3) \to \PU$. In particular, the isometries $\rho(g_1),\rho(g_2),\rho(g_3)$ are elliptic. 
\begin{rmk}An isometry $A \in \PU \setminus \{\mathrm{id}\}$ can be seen as an element of $\SU$ up to choice of representative. We say that $A$ is elliptic if it has a negative eigenvector $c$, i.e., $\langle c,c \rangle<0$, which corresponds to a fixed point in $\HH_\CC^2$. If an elliptic isometry has three distinct eigenvalues, then we say it is regular elliptic, otherwise, we say it is special elliptic. We discuss the geometry of such isometries in Section \ref{subsection: holomorphic isometries}. A finite order element of $\PU$ is always elliptic. Thus, $\rho(g_1),\rho(g_2),\rho(g_3)$ are elliptic. 
\end{rmk}

The examples in \cite{discbundles} come from representations with $n_1=n_3=n$ and $n_2=2$ such that $\rho(g_1),\rho(g_3)$ are regular elliptic isometries and $\rho(g_2)$ is a reflection in a complex geodesic (see Section \ref{subsection: holomorphic isometries} for the corresponding definitions). In this work, we drop this restriction over $n_1,n_2,n_3$ and analyze examples where $\rho(g_1),\rho(g_2),\rho(g_3)$ are regular elliptic and examples where two of them are regular and the third one is special. 
We exclude those where at least two of the $\rho(g_j)$'s are not regular, since discrete faithful representations of this type are $\CC$-Fuschsian, meaning that they have a stable complex geodesic, see Proposition \ref{prop case 2 isometries are special}.

\smallskip

The generic representations where the $\rho(g_j)$'s are all regular elliptic are the most interesting ones because they correspond to two-dimensional regions on the $\PU$-character variety $\mathcal R(n_1,n_2,n_3)\coloneq\hom(G(n_1,n_2,n_3),\PU)/\PU$. More precisely, accordingly to Definition \ref{generic_representation} a representation $\rho:G(n_1,n_2,n_3) \to \PU$ is generic if there exists $j\ne k$ such that no eigenvector of $\rho(g_j)$ is orthogonal to an eigenvector of $\rho(g_k)$. In Proposition \ref{prop main result} we show that the open subset of $\mathcal R(n_1,n_2,n_3)$ of all generic representations $\rho$ such that $\rho(g_1), \rho(g_2), \rho(g_3)$ are regular elliptic isometries of order $n_1,n_2,n_3$ is two-dimensional if it is non-empty. We also provide in Remark \ref{remark alg} explicit parametrization for this open region.

This allows us to find two-dimensional families of pairwise non-isometric complex hyperbolic structures over the same disc orbibundle. In contrast, when two of the isometries are regular elliptic and the third one is special, the representation is an isolated point in the character variety, meaning that the example is locally rigid, not allowing small deformations to its complex hyperbolic structure. 
The representations $G(n,2,n)\to\PU$, found in \cite{discbundles}, are included in this second class of representations. The examples we have found are outlined in Section \ref{section computational results}.

We highlight two interesting families of examples of complex hyperbolic disc orbibundles where $\rho$ is generic and  all $\rho(g_i)'s$ are regular elliptic.

The first satisfies $e=0$, where $e$ stands for the Euler number of the disc orbibundle. Therefore, up to pullback, it gives rise to trivial disc bundles over closed Riemann surfaces. Determining whether or not a trivial bundle over a Riemann surface admits a complex hyperbolic structure was a long-standing problem; see, for instance, \cite[Open Question 8.1]{eli}, \cite[p. 583]{gol2}, and \cite[p. 14]{sch}. It has been first solved in \cite{agu} using a discrete faithful representation in the isometry group of $\HH_\CC^2$ of a group generated by two reflections in points and a reflection in an $\mathbb R$-plane. We provide explicit computations for a non-rigid example satisfying $e=0$ in Section \ref{section:Explicit example with trivial Euler number}.

The second family satisfies $e/\chi=-1$; here, $\chi$ denotes the Euler characteristic of the sphere with three cone points. Up to pullback, we obtain complex hyperbolic structures on cotangent bundles of Riemann surfaces. To the best of our knowledge, the fact that the cotangent bundle of a Riemann surface has a complex hyperbolic structure was previously unknown.

As we mentioned above, besides the Euler number $e$ of the disc orbibundle and the Euler characteristic $\chi$ of the $2$-orbifold, there is a third discrete invariant attached to each of our constructed complex hyperbolic disc orbibundles, the Toledo invariant (see, for instance, \cite[Definition 35]{bot}, \cite{krebs}, \cite{tol}). As in \cite{discbundles}, the formula $2(e+\chi)=3\tau$ holds in all examples we have constructed. 

\begin{rmk} This formula expresses a necessary condition for the existence of a holomorphic section of an orbibundle complex hyperbolic orbibundle $M \to S$. In fact, up to finite cover, we may assume that $S$ is a Riemann surface holomorphically embedded in $M$ as a section and from the exact sequence of complex vector bundles $0\to TS \to TM|_{S} \to NS \to 0$
we obtain the following relation between the first Chern numbers $$e+\chi =c_1(NS)+c_1(TS)= c_1(TM|_S).$$ 
From \cite[Section 2.2]{GKL} we have $3\tau = 2c_1(TM|_S)$. Thus, $2(e+\chi) = 3\tau$ in the presence of holomorphic sections. \end{rmk}

For the complex hyperbolic disc orbibundles in \cite{discbundles}, such a section does indeed exist \cite{kap2}; however, the proof relies on the local rigidity of representations $\rho: G(n,2,n)\to\PU$ and, therefore, does not extend to the non-rigid examples constructed here. Moreover, all the disc orbibundles we found endorse the complex hyperbolic variant of the Gromov-Lawson-Thurston conjecture (see \cite{discbundles}, \cite{GLT}) which states that an oriented disc bundle over a closed Riemann surface admits a complex hyperbolic structure if and only if $|e/\chi|\leq1$. Indeed, $-1 \leq e/\chi\leq1/2$ in all the examples we constructed. 

As in \cite{discbundles}, the fundamental domains we deal with
are bounded by a quadrangle of bisectors, i.e., of segments of hypersurfaces that are equidistant from a pair of points. Nevertheless, we found it necessary to develop some new tools to calculate the Euler number because, in the general case, there is no explicit way to obtain a surface group (whose existence is guaranteed by the Selberg Lemma) as a finite index subgroup of the turnover. Among these tools, we have the deformation Lemma
\ref{deformationlemma}, a central piece in calculating the Euler number.

At some point, we believed that all representations of the turnover in $\PU$ with regular $\rho(g_j)$'s were discrete. This naive point of view turned out to be false (see the reasoning above Figure \ref{334goldman}) but it seemed to be supported by the following observation. To prove discreteness, we essentially need to verify a list of inequalities involving some geometric invariants related to the fundamental domain. However, even when these inequalities are invalid (and, furthermore, even when we know the corresponding representation is not discrete) we can still apply the formulas that calculate the invariants $\chi,e,\tau$. Surprisingly, $2(e+\chi)=3\tau$ still holds. All this favors the study of complex hyperbolic geometry underlying quotients of $\HH_\CC^2$ which are more singular than orbifolds and has been a central motivation for the diffeological approach started in~\cite{bot}.

\section{Preliminaries}\label{section preliminaries}

\subsection{Complex hyperbolic generalities}
Let $V$ be a three-dimensional complex vector space endowed with a Hermitian form $\langle-,-\rangle:V\times V\to\CC$ of signature $-++$. Let
$$\BV\coloneq \big\{p\in\PP_\CC(V)\mid\langle p,p\rangle<0\big\},\quad
\SV\coloneq \big\{p\in\PP_\CC(V)\mid\langle p,p\rangle=0\big\},$$
$$\EV\coloneq \big\{p\in\PP_\CC(V)\mid\langle p,p\rangle>0\big\}$$
stand respectively for the subspaces of the complex projective plane $\PP_\CC(V)$ consisting of {\it negative,} {\it isotropic,} and {\it positive\/} points. We use the same letter to denote both a point $p\in\PP_\CC(V)$ and a representative of it in $V\setminus\{0\}$. This is harmless as long as we are referring to formulas that are independent of the choice of representatives.

The tangent space $\T_p\PP_\CC(V)$ to a nonisotropic point $p\in\PP_\CC(V)$ can be naturally identified with the space $\Lin(\CC p,p^\perp)$ of $\CC$-linear maps from the complex line $\CC p$ to its orthogonal complement with respect to the Hermitian form. The {\it complex hyperbolic plane\/} $\HH_\CC^2$ is the holomorphic $2$-ball $\BV$ of negative points equipped with the positive-definite {\it Hermitian metric}
\begin{equation}\label{metric}\langle t_1,t_2\rangle\coloneq -\frac{\big\langle t_1(p),t_2(p)\big\rangle}{\langle p,p\rangle},\quad t_1,t_2\in\T_p\BV.
\end{equation}
The ideal boundary of the complex hyperbolic plane in $\PP_\CC(V)$ is the $3$-sphere $\SV$ called the {\it absolute\/} and denoted by $\partial\HH_\CC^2$. We write $\overline\HH_\CC^2\coloneq \HH_\CC^2\cup\partial\HH_\CC^2$.

The real part of the Hermitian metric \eqref{metric} is a Riemannian metric in $\HH_\CC^2$ whose distance function is given by $d(p,q)=\mathrm{arccosh}\sqrt{\ta(p,q)}$, where $$\ta(p,q)\coloneq \frac{\langle p,q\rangle\langle q,p\rangle}{\langle p,p\rangle\langle q,q\rangle}$$
is the {\it tance\/} between $p,q\in\HH_\CC^2$. The imaginary part of \eqref{metric} is the {\it K\"ahler\/ 2-form\/} $\omega$. For each $c\in\HH_\CC^2$,
\begin{equation}\label{potential}P_c(t)\coloneq -\frac12\imag\frac{\big\langle t(p),c\big\rangle}{\langle p,c\rangle},\quad t\in\T_p\HH_\CC^2,
\end{equation}
is a potential for $\omega$, that is, $dP_c=\omega$. Potentials $P_{c_1},P_{c_2}$ based at possibly distinct points $c_1,c_2\in\HH_\CC^2$ are related by
\begin{equation}\label{basepointchange}
P_{c_1}=P_{c_2}+df_{c_1,c_2},\;\;\textrm{where}\;\;
f_{c_1,c_2}(p)\coloneq \frac12\Arg\frac{\langle c_1,p\rangle\langle p,c_2\rangle}{\langle c_1,c_2\rangle}\;\;\textrm{for every}\;\;p \in \HH_\CC^2
\end{equation}
(due to the signature of the Hermitian form, $\langle c_1,c_2\rangle\ne0$ for all
$c_1,c_2\in\HH_\CC^2$). The above explicit relation between potentials with distinct basepoints lies at the core of the calculation of the Toledo invariant of the discrete faithful $\PU$-representations that we construct (see Proposition \ref{toledomod2}).

\subsection{Totally geodesic subspaces}
The geodesics of the Riemannian metric are given by the nonempty intersections with $\HH_\CC^2$ of projectivizations $\PP_\CC(W)=\PP_\RR(W)$ of two-dimensional real subspaces $W$ of $V$ such that the Hermitian form restricted to $W$ is real and does not vanish. A geodesic $\PP_\CC(W)\cap\HH_\CC^2$ has two distinct {\it vertices\/} $\PP_\CC(W)\cap\partial\HH_\CC^2=\{v_1,v_2\}$, $v_1\ne v_2$. The unique geodesic determined by a pair of distinct points $c_1,c_2\in\overline\HH_\CC^2$ will be denoted by $G(c_1,c_2)$ and the segment of geodesic connecting $c_1,c_2$, by $G[c_1,c_2]$. Note that, explicitly, $G(c_1,c_2)=\PP_\CC\big(\RR c_1+\RR\langle c_1,c_2\rangle c_2\big)$.

There are two types of totally geodesic (real) surfaces in $\HH_\CC^2$: the complex geodesics and the $\RR$-planes. The complex geodesics are the nonempty intersections of projective lines with $\HH_\CC^2$; they are nothing but copies of a Poincar\'e disc (of constant curvature $-4$) inside $\HH_\CC^2$. The $\RR$-planes are the nonempty intersections of $\HH_\CC^2$ with projectivizations $\PP_\CC(W)=\PP_\RR(W)$ of three-dimensional real subspaces $W$ of $V$ such that the Hermitian form restricted to $W$ is real of signature $-++$. They correspond to copies of a Beltrami-Klein disc (of constant curvature $-1$) inside $\HH_\CC^2$.

We will sometimes consider that geodesics, complex geodesics, and $\RR$-planes are extended to the absolute $\partial\HH_\CC^2$.

Let $U$ be a two-dimensional complex subspace of $V$ such that the signature of the Hermitian form restricted to $U$ is $-+$. The positive point $\PP_\CC(U^\perp)\in\EV$ is the {\it polar\/} point of the complex geodesic $\PP_\CC(U)\cap\HH_\CC^2$. So, $\EV$ is the space of all complex geodesics in $\HH_\CC^2$. Note that the geodesic $\PP_\CC(W)\cap\HH_\CC^2$ is contained in a unique complex geodesic given by $\PP_\CC(W+iW)\cap\HH_\CC^2$.

A pair of complex geodesics is called {\it ultraparallel,} {\it asymptotic,} or {\it concurrent\/} when the complex geodesics do not intersect in $\overline\HH_\CC^2$, have a single common point in $\partial\HH_\CC^2$, or have a single common point in $\HH_\CC^2$. We write $C_1||C_2$ for ultraparallel complex geodesics $C_1,C_2$.

\begin{rmk}\label{simplecgproperties}
{\bf1.}~Let $L_1,L_2$ be complex geodesics with polar points $p_1,p_2$. Then $L_1,L_2$ are respectively ultraparallel, asymptotic, concurrent iff $\ta(p_1,p_2)>1$, $\ta(p_1,p_2)=1$, $\ta(p_1,p_2)<1$.

\smallskip

\noindent
{\bf2.}~Let $L=\PP_\CC(U)$ be a projective line such that the Hermitian form on $U$ is nondegenrate. Given $p\in L$, there exists a unique $q\in L$ such that $\langle p,q\rangle=0$.

\smallskip

\noindent
{\bf3.}~The tance between a complex geodesic $L$ and a point $p\in\HH_\CC^2$ is given by
$$\ta(L,p)\coloneq \min\big\{\ta(x,p)\mid x\in L\big\}=1-\ta(p,q),$$
where $q$ is the polar point of $L$.
\end{rmk}

\subsection{Bisectors}\label{subsection bisectors}
There are no totally geodesic hypersurfaces in $\HH_\CC^2$. In our construction of fundamental polyhedra, we use hypersurfaces known as {\it bisectors.} A bisector can be characterized as the equidistant locus from two distinct points in $\HH_\CC^2$. Alternatively, it is also determined by a (real) geodesic in $\HH_\CC^2$ and this is the viewpoint that we adopt and briefly describe in what follows.

Let $G=\PP_\CC(W)$ be a geodesic in $\HH_\CC^2$, let $L$ be its complex geodesic, that is, $L=\PP_\CC(W+iW)$, and let $p$ be the polar point of $L$. The bisector $B$ with {\it real spine\/} $G$ and {\it complex spine\/} $L$ is given by
$$B\coloneq \PP_\CC(W+\CC p)\cap\HH_\CC^2.$$
As before, we will sometimes consider bisectors as being extended to $\overline\HH_\CC^2$. The point $p$ is called {\it focus} of the bisector $B$.

The bisector $B$ with real spine $G$ is foliated by complex geodesics (see Figure \ref{bisector1}),
$$B=\bigsqcup_{x\in G}L_x,\;\;\textrm{where}\;\; L_x\coloneq \PP_\CC(\CC x+\CC p)\cap\HH_\CC^2.$$
For each $x\in G$, the complex geodesic $L_x$ is the unique complex geodesic through $x$ orthogonal to the complex spine $L$ in the sense of the Hermitian metric \eqref{metric}.  The complex geodesic $L_x$ is called the {\it slice\/} of $B$ through $x$. Each point in $B$ belongs to a unique slice of $B$.

\begin{figure}[H]
		\centering
		\includegraphics[scale = .7]{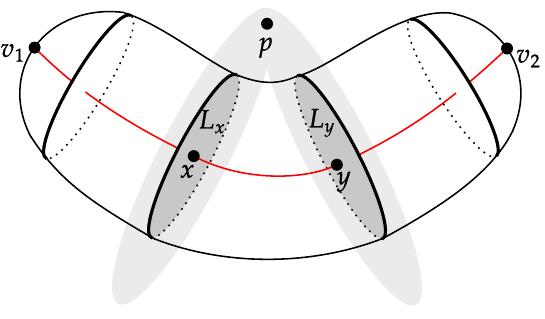}
		\captionof{figure}{Bisector foliated by complex geodesics.}
		\label{bisector1}
\end{figure}

The bisector $B$ with real spine $G=\PP_\CC(W)$ also admits the {\it meridional decomposition\/} 
$$B=\bigcup_{\varepsilon\in\SP^1}\PP_\CC\big(W+\RR\varepsilon p\big)\cap\HH_\CC^2,$$
where $p\in V\setminus\{0\}$ is a fixed representative of the polar point $p$ of the complex spine $L$ and $\varepsilon\in\SP^1$ is a unit complex number (see Figure \ref{bisector2}). Given $\varepsilon\in\SP^1$, the $\RR$-plane $\PP_\CC(W+\RR\varepsilon p)\cap\HH_\CC^2$ is called a {\it meridian\/} of the bisector. Every meridian of $B$ contains the real spine $G$. Each point $p\in B\setminus G$ is contained in a unique meridian $M$ of $B$ and determines a {\it meridional curve\/} which is the curve in $M$ through $p$ equidistant from $G$ (in other words, a {\it hypercycle\/} in the Beltrami-Klein disc $M$). We also define a meridional curve when $p\in B$ is isotropic. In this case, the intersection $M\cap\partial\HH_\CC^2$ is a circle divided by the vertices of $G$ into two semicircles; we take the one containing $p$.

\begin{figure}[H]
	\centering
	\includegraphics[scale = 1]{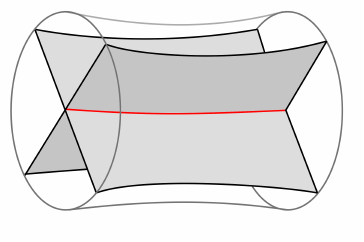}
	\captionof{figure}{Meridional decomposition.}
	\label{bisector2}
\end{figure}

A pair of ultraparallel complex geodesics $L_1,L_2$ determines a unique bisector $B(L_1,L_2)$ whose real spine is the unique geodesic $G$ that is simultaneously orthogonal to $L_1$ and $L_2$. Explicitly, this geodesic can be constructed as follows. The projective lines containing $L_1,L_2$ intersect at a positive point $p\in\EV$. The complex geodesic $\PP_\CC(p^\perp)$ intersects $L_i$ at $c_i$, $i=1,2$, and $G=G(c_1,c_2)$. The {\it segment of bisector\/} $B[L_1,L_2]$ is defined by
$$B[L_1,L_2]\coloneq \bigsqcup_{x\in G[c_1,c_2]}L_x,$$
where $L_x$ stands for the slice of $B[L_1,L_2]$
through $x$. The slice of $B[L_1,L_2]$ through the middle point of $G[c_1,c_2]$ is called the {\it middle slice\/} of $B[L_1,L_2]$.

\subsection{Holomorphic isometries}\label{subsection: holomorphic isometries}

The group of holomorphic isometries of $\HH_\CC^2$ is the projective unitary group $\PU$. The special unitary group $\SU$ is a triple cover of $\PU$ (lifts differ by a cube root of unity) and we refer to elements in $\SU$ also as isometries.

In our construction of discrete groups, we essentially use elliptic isometries. An isometry $I\in\SU$ is said to be {\it elliptic\/} when it has a negative fixed point $c\in\HH_\CC^2$. In this case, the projective line $\PP_\CC(c^\perp)$ is $I$-stable. So, the isometry has a fixed point $p\in\PP_\CC(c^\perp)$. The point $q\in\PP_\CC(c^\perp)$ which is orthogonal to $p$ (see Remark \ref{simplecgproperties}) must also be fixed by $I$. In other words, there is an orthogonal basis in $V$ formed by eigenvectors of $I$. Let $\varepsilon_1,\varepsilon_2,\varepsilon_3\in\mathbb C$ with $\varepsilon_1\varepsilon_2\varepsilon_3=1$ be the eigenvalues corresponding respectively to $c,p,q$. Since none of $c,p,q$ is isotropic, we have $|\varepsilon_i|=1$ for $i=1,2,3$. It is straightforward to see that $I$ is given by the rule
\begin{equation}\label{general formula for elliptic isometry}I:x\mapsto(\varepsilon_1-\varepsilon_3)\frac{\langle x,c\rangle}{\langle c,c\rangle}c+(\varepsilon_2-\varepsilon_3)\frac{\langle x,p\rangle}{\langle p,p\rangle}p+\varepsilon_3x.
\end{equation}

The isometry $I$ is called {\it regular\/} elliptic if its eigenvalues are pairwise distinct and {\it special\/} elliptic otherwise. We may describe the geometry of regular and special elliptic isometries as follows.

\smallskip

\noindent
{\bf The regular elliptic case.} The points $c,p,q$ are the only fixed points of $I$. We call $c$ the {\it center\/} of the isometry. The
complex geodesics $\PP_\CC(p^\perp)$ and $\PP_\CC(q^\perp)$ intersect orthogonally at $c$ and both are $I$-stable. There are no other
$I$-stable complex geodesics. Moreover, $I$ acts on $\PP_\CC(p^\perp)\cap \HH_\CC^2$ as the rotation about $c$ by the angle
$\Arg(\varepsilon_1^{-1}\varepsilon_3)$ and on $\PP_\CC(q^\perp)\cap \HH_\CC^2$ as the
rotation about $c$ by the angle
$\Arg(\varepsilon_1^{-1}\varepsilon_2)$.

\smallskip

\noindent
{\bf The special elliptic case.} We can assume that not all eigenvalues
of $I$ are equal (for otherwise, $I$ acts identically on $\PP_\CC(V)$).
Hence, exactly one of the projective lines $\PP_\CC(c^\perp)$,
$\PP_\CC(p^\perp)$, or $\PP_\CC(q^\perp)$ is pointwise fixed by $I$. If the pointwise fixed line is $\PP_\CC(c^\perp)$, that is, if $\varepsilon_2=\varepsilon_3$, then every complex geodesic that passes through $c$ is $I$-stable, the isometry acts on such complex geodesic as the rotation about $c$ by the angle $\Arg(\varepsilon_1^{-1}\varepsilon_2)$, and there are no other $I$-stable complex geodesics. In this case, we call $c$ the {\it center\/} of $I$ as well. When $\PP_\CC(p^\perp)$ is pointwise fixed ($\varepsilon_1=\varepsilon_3$), every complex geodesic intersecting $\PP_\CC(p^\perp)$ orthogonally in a negative point is stable under $I$, the isometry acts on such complex geodesics as the rotation about the intersection point by the angle $\Arg(\varepsilon_1^{-1}\varepsilon_2)$, and there are no other $I$-stable complex geodesics. In other words, $I$ is a rotation with the {\it axis\/} $\PP_\CC(p^\perp)\cap \HH_\CC^2$. The same is true for a rotation with the axis $\PP_\CC(q^\perp) \cap \HH_\CC^2$. An important particular case of rotation about an axis is the {\it reflection\/} in a complex geodesic $L=\PP_\CC(p^\perp)\cap\HH_\CC^2$ given by the involution $x\mapsto-x+2\frac{\langle x, p\rangle}{\langle p,p\rangle}p$ (taking $\varepsilon_1=\varepsilon_3=-\varepsilon_2=-1$ in expression \eqref{general formula for elliptic isometry}). See Figure \ref{ellipticisometries}.

\begin{figure}[H]
	\centering
	\begin{minipage}{.3\textwidth}
		\centering
	\includegraphics[scale = .8]{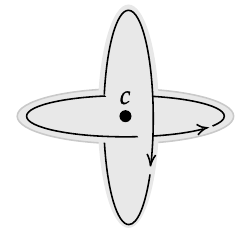}
	\caption*{(a)}
	\end{minipage}%
	\begin{minipage}{.3\textwidth}
		\centering
		\includegraphics[scale = .8]{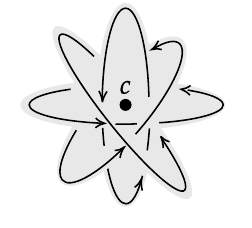}
		\caption*{(b)}
	\end{minipage}%
	\begin{minipage}{.3\textwidth}
		\centering
		\includegraphics[scale = .8]{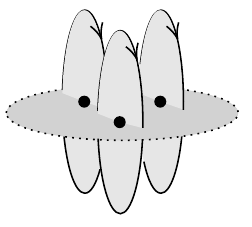}
		\caption*{(c)}
	\end{minipage}
\caption{\textbf{(a)} Rotations about $c$ on two orthogonal complex geodesics by distinct angles, \textbf{(b)} Rotation about point, and \textbf{(c)} Rotation about complex line.}
\label{ellipticisometries}
\end{figure}

\subsection{The turnover and its \texorpdfstring{$\PU$}{PU(2,1)}-character variety}

\label{section turnover and character variety}
\subsubsection{The turnover}\label{turnover definition}
The group
$$G(n_1,n_2,n_3)\coloneq \langle g_1,g_2,g_3\mid
g_1^{n_1}=g_2^{n_2}=g_3^{n_3}=1{\text{ \rm and }}g_3g_2g_1=1\rangle$$
is called the (hyperbolic) {\it turnover\/}, where $n_1$, $n_2$, $n_3$ are positive integers satisfying $\sum\limits_{j=1}^3\frac1{n_j}<1$.

We typically write simply $G$ in place of $G(n_1,n_2,n_3)$. It is well-known that $G$ has a discrete cocompact action on the real hyperbolic plane $\mathbb H^2_\mathbb R$ (see Figure \ref{fundamental domain and orbifold}). Indeed, take a geodesic triangle $\Delta\subset\mathbb H^2_\mathbb R$ with interior angles $\pi/n_1,\pi/n_2,\pi/n_3$ and let $H(n_1,n_2,n_3)$ denote the {\it triangle\/} group generated by the reflections $r_1,r_2,r_3$ in the sides of $\Delta$. The turnover $G$ appears as the index $2$ subgroup in $H$ generated by the rotations $g_1\coloneq r_1r_2$, $g_2\coloneq r_3r_1$, $g_3\coloneq r_2r_3$. By the Poincar\'e Polyhedron Theorem, the quadrilateral $P\coloneq \Delta\cup r_2\Delta$ with the vertices, sides, and side-pairings indicated in Picture \ref{fundamental domain and orbifold}, is a fundamental domain for the action of $G$ on $\HH_\RR^2$.

The orbifold $\HH_\RR^2/G$ is the $2$-sphere $\SP^2(n_1,n_2,n_3)$ with $3$ cone points of angles $2\pi/n_1,2\pi/n_2,2\pi/n_3$ and orbifold Euler characteristic (see \cite{sco})
\[\chi=-1+\frac{1}{n_1}+\frac{1}{n_2}+\frac{1}{n_3}.\]

\begin{figure}[H]
	\centering
	\begin{minipage}{.5\textwidth}
		\centering
		\includegraphics[scale = .8]{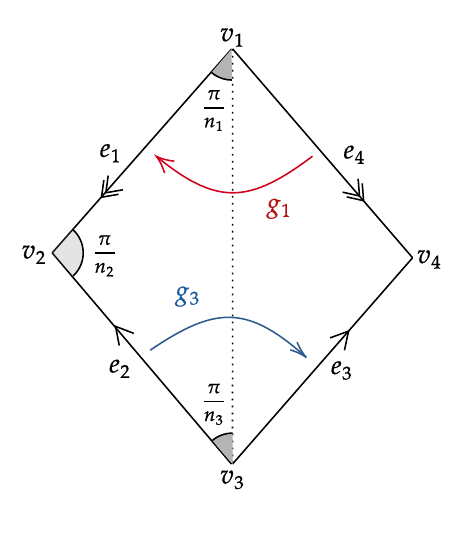}
		\caption*{(a)}
	\end{minipage}%
	\begin{minipage}{.5\textwidth}
		\vspace{1.53cm}
		\centering
		\includegraphics[scale =0.9]{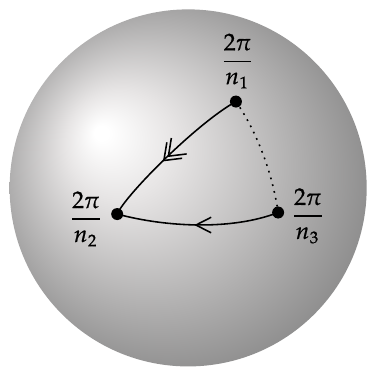}
		\caption*{(b)}
	\end{minipage}
	\caption{\textbf{(a)} Fundamental domain for $G$ and \textbf{(b)} Orbifold $\SP^2(n_1,n_2,n_3)$}
	\label{fundamental domain and orbifold}
\end{figure}

\subsubsection{Character variety}\label{character_variety} The  $\PU$-character variety for a co-compact Fuchsian group $\Gamma$ is the space of representations $\Gamma \to \PU$ modulo $\PU$-conjugation
$$\hom(\Gamma,\PU)/\PU.$$
Here, the action of $\PU$ on $\hom(\Gamma,\PU)$ given by $(g,\rho) \mapsto g\rho(\cdot)g^{-1}$, where $\rho:\Gamma \to \PU$ is a representation and $g \in \PU$.

We denote by $\mathcal R(n_1,n_2,n_3)$ the $\PU$-character variety for the turnover group $G(n_1,n_2,n_3)$. Sometimes we write $\mathcal R$ instead of $\mathcal R(n_1,n_2,n_3)$ if the parameters are clear from the context.

\newpage 

\section{Generic representations}
\label{section generic representations}
As commented in Remark \ref{remark disc orbibundle}, we are interested in constructing disc orbibundles with complex hyperbolic total spaces. Especially non-rigid ones, meaning that the holomorphic structure of the disc orbibundle can be deformed without altering the topology of the orbibundles. We can view a complex hyperbolic disc orbibundle modulo holomorphic isomorphism as a point in the $\PU$-character variety. Thus, to find a non-rigid complex hyperbolic disc orbibundle we must investigate representations of the $\PU$-character variety that are not isolated points.  

From here onwards we only consider representations $\rho:G(n_1,n_2,n_3) \to \PU$ such that $\rho(g_j)$ have order $n_j$ for $j=1,2,3$. The isometries $\rho(g_j)$ are elliptic because a non-identical isometry of finite order is elliptic. We will denote $\rho(g_j)$ by $I_j$.



Consider a representation $\rho: G\to\PU$ that stabilizes a projective line in $\mathbb P_\CC(V)$, that is, it possesses a fixed point in $\mathbb P_\CC(V)$. If $\rho$ is discrete, then the quotient of the stable projective line by $G$ is $\SP^2(n_1,n_2,n_3)$, which has negative Euler characteristic by definition. Thus, the representation $\rho$ is $\CC$-Fuchsian, meaning the stable projective line is a complex geodesic (a Poincaré disc). It is easy to construct complex hyperbolic disc orbibundles from $\CC$-Fuchsian examples, and their Toledo invariant and Euler number are known to satisfy $|\tau/\chi| = 1$ and $e/\chi = 1/2$ (see \cite{discbundles}).

Therefore, we exclude representations for which two of the isometries $I_1, I_2, I_3$ are special elliptic, because representations of this type always have stable projective lines by the Proposition~\ref{prop case 2 isometries are special}.

\medskip
\begin{prop} \label{prop case 2 isometries are special} If at least two of the\/ $I_j$'s are special elliptic isometries, then\/ $\rho$ stabilizes a projective line.
\end{prop}

\medskip

{\bf Proof.} A special elliptic isometry has a pointwise fixed projective line (see Section \ref{section preliminaries}). Hence, we can find a point simultaneously fixed by two special elliptic isometries among $I_j$, $j=1,2,3$. This point must also be fixed by the remaining isometry due to the relation $I_3I_2I_1=1$.
\hfill$_\blacksquare$

\medskip

\begin{rmk}\label{PUSU}
Whenever we deal with elliptic isometries
$I_1, I_2, I_3$ we will assume that either:

\begin{itemize}
	\item $I_1,I_2,I_3\in\PU$ with $I_j^{n_j}=I_3I_2I_1=1$ and
	$\sum\limits_{j=1}^3\frac1{n_j}<1$;
	\item $I_1,I_2,I_3\in\SU$ with $I_j^{n_j}=\delta_j$ and
	$I_3I_2I_1=1$, where $\delta_j\in\mathbb C$ is a cubic root of unity and
	$\sum\limits_{j=1}^3\frac1{n_j}<1$.
\end{itemize}
Whether we take the isometries in $\PU$ or in $\SU$ will be explicitly
indicated or should be clear from the context.
\end{rmk}

From now on, we assume that at least two of the isometries $I_1, I_2, I_3$ are regular.
So, let $I_1$ and $I_3$ be regular elliptic.  If that is not the case, we use that $I_3I_2I_1= I_2I_1I_3^{-1}=I_1^{-1}I_3I_2=1$. If $I_1$ is special elliptic we can define $I_3'\coloneq I_2$, $I_2'\coloneq I_1$ and $I_1'\coloneq I_3^{-1}$ and if $I_3$ is special elliptic we can define $I_3'\coloneq I_1^{-1}$, $I_2'\coloneq I_3$ and $I_1'\coloneq I_2$.

We can choose representatives for the isometries $I_1, I_2, I_3$ in $\SU$ with respective eigenvalues $\alpha_j,\beta_j,\gamma_j^{-1}$, $j=1,2,3$, satisfying $I_3I_2I_1=1$. The eigenvalues denoted with index $1$ correspond to a negative eigenvector. Since $I_1,I_3$ are regular elliptic, we have $\alpha_i \neq \alpha_j$ and $\gamma_i \neq \gamma_j$ for $i\neq j$. For $I_2$, there are three possibilities. It can be regular elliptic ($\beta_i\ne\beta_j$ for $i\ne j$), a rotation about a point in $\HH_\CC^2$ ($\beta_2 = \beta_3$), or a rotation around a complex geodesic ($\beta_1 = \beta_2$ or $\beta_1 =\beta_3$).

So, in order to find representations of the turnover $G(n_1,n_2,n_3)$ in $\PU$ we fix \[(\alpha_1,\alpha_2,\alpha_3),(\beta_1,\beta_2,\beta_3),(\gamma_1,\gamma_2,\gamma_3) \in \SP^1\times \SP^1 \times \SP^1\]  of order $3n_1,3n_2,3n_3$, respectively,  satisfying
\begin{gather*}
\alpha_1\alpha_2\alpha_3 = \beta_1\beta_2\beta_3=\gamma_1\gamma_2\gamma_3=1,\\
\alpha_i  \neq \alpha_j,\quad \gamma_i  \neq \gamma_j,
\quad\alpha_i^{n_1}  = \alpha_j^{n_1},\quad \beta_i^{n_2}  = \beta_j^{n_2},\quad \gamma_i^{n_3} = \gamma_j^{n_3} \quad for \quad  i\neq j,
\end{gather*}
and  a regular elliptic isometry $I_1$ with eigenvalues $\alpha_j$. We look for all $I_2$ such that \begin{equation}\label{trace equation}
\tr(I_2I_1)=\sum\limits_{i=1}^3\gamma_i.
\end{equation}
It follows from \cite[p.~204, Theorem 6.2.4]{goldmanbook} that this {\it trace equation\/} holds iff $I_3\coloneq (I_2I_1)^{-1}$ is a regular elliptic isometry with eigenvalues $\gamma_j^{-1}$. This strategy allows us to prove the Proposition~\ref{prop main result}.

\begin{rmk}This result by Goldman is analogous to the one in the Poincaré half-plane, where the trace of an isometry determines if an isometry is elliptic. More precisely, Goldman's result states that an isometry $I$ is regular elliptic if, and only if, it has the trace of a regular elliptic~isometry.
\end{rmk}
\smallskip

\begin{defi}\label{generic_representation}
Let $\rho$ be a representation where $\rho (g_1), \rho (g_2), \rho (g_3)$ are elliptic isometries. We call the representation {\it generic\/} if there exists $i\ne j$ such that the fixed points of $\rho(g_i)$ and $\rho(g_j)$ are pairwise non-orthogonal.
\end{defi}

\begin{prop} \label{prop main result} Let $G\coloneq G(n_1,n_2,n_3)$ be a turnover group and $\mathcal R$ its corresponding $\PU$-character variety.
\begin{itemize}
    \item The subset of $\mathcal R$ formed by representations $\rho: G \to \PU$ such that two of the isometries $\rho(g_1),\rho(g_2), \rho(g_3)$ are regular elliptic and the third one is special elliptic is discrete, meaning that its representations are rigid.
    
    \item The open subset of $\mathcal R$ of all
    generic representations $\rho$ such that $\rho(g_1), \rho(g_2), \rho(g_3)$ are regular elliptic isometries of order $n_1,n_2,n_3$ is two-dimensional if it is non-empty.

    More precisely, we prove the following: consider $$\Lambda: = \{\rho \in \mathcal R(n_1,n_2,n_3): 
 \rho(g_1) \text{ and }\rho(g_2) \text{ have pairwise non-orthogonal eigenvectors}\}.$$ 
 Let $L_1, L_2$ be the complex geodesics stable under the action of $\rho(g_1)$ and $u$ be the negative fixed point of $\rho(g_2)$. We have the two positive parameters $s\coloneq \ta(u,L_1)-1$ and $t\coloneq\ta(u,L_2)-1$.

 The map $\pi:\Lambda \mapsto \RR_{>0} \times \RR_{>0}$, $\rho \mapsto (s,t)$, is a two-to-one local homeomorphism. Given $\rho$, the other representation having the same $(s,t)$ is $\rho'$ defined by $$g_1 \mapsto \rho(g_1)^{-1}, \quad g_2 \mapsto \rho(g_2)^{-1}, \quad g_3 \mapsto \rho(g_2)^{-1}\rho(g_3)^{-1}\rho(g_2).$$ 

 \end{itemize}
\end{prop}
Section \ref{section proof of main result} is devoted to proving the Proposition \ref{prop main result}.
 \begin{rmk}\label{remark alg} We can write down $\rho, \rho' \in \Lambda $ explicitly as functions of $(s,t)$.
 
 Define the constants 
    $\alpha_{ij}\coloneq  \alpha_i - \alpha_j$, $\beta_{ij}\coloneq  \beta_i - \beta_j$,
    $$k\coloneq \frac1{\beta_{23}}\left(\sum\limits_{i=1}^3\gamma_i-\alpha_1(\beta_1+
    \beta_2-\beta_3)-\beta_3(\alpha_2+\alpha_3)\right), \quad M=\left[\smallmatrix\real\alpha_{21}&\real\alpha_{31}\\
\imag\alpha_{21}&\imag\alpha_{31}\endsmallmatrix\right],$$
    where $(\alpha_1,\alpha_2,\alpha_3)$, $(\beta_1,\beta_2,\beta_3)$, $(\gamma_1,\gamma_2,\gamma_3)$ are the eigenvalues of $\rho(g_1)$, $\rho(g_2)$, $\rho(g_3)$, respectively. 
    
     Consider the following functions of $(s,t)$:
     \begin{gather*}
     u_1 \coloneq  \sqrt{1+s+t}, \quad u_2\coloneq  \sqrt{s}, \quad u_3\coloneq \sqrt{t},\\
     |v_2|^2 \coloneq \frac{s\imag\frac{\alpha_{31}\overline\alpha_{21}\overline
	\beta_{13}}{\overline\beta_{23}}+t|\alpha_{31}|^2\imag\frac{\overline
	\beta_{13}}{\overline\beta_{23}}+\imag(\alpha_{31}\overline k)}{\det M},\\
    |v_3|^2 \coloneq  \frac{s|\alpha_{21}|^2\imag\frac{\beta_{13}}{\beta_{23}}+
t\imag\frac{\overline\alpha_{21}\alpha_{31}\beta_{13}}{\beta_{23}}+
\imag(\overline\alpha_{21}k)}{\det(M)},\\
v_1^2\coloneq |v_2|^2+|v_3|^2-1
\end{gather*}
All these functions are positive-valued. From, these functions we obtain the parameters
$v_1,v_2,v_3$ by the formulas
$$v_1 = \sqrt{|v_2|^2+|v_3|^2-1}$$
$$v_2=\frac1{2v_1\sqrt{s(1+s+t)}}\big(-t|v_3|^2+(1+s+t)v_1^2+
s|v_2|^2\pm
i\sqrt\Delta\big),$$
$$v_3=\frac1{2v_1\sqrt{t(1+s+t)}}\big(-s|v_2|^2+(1+s+t)v_1^2+
t|v_3|^2\mp
i\sqrt\Delta\big),$$
where
$$
\Delta\coloneq 4v_1^2|v_2|^2s(1+s+t)-\big(-t|v_3|^2+(1+s+t)v_1^2+
s|v_2|^2\big)^2\ge0.
$$

The representation $\rho$ modulo conjugation is given by $I_j=\rho(g_j)$, where
\begin{equation*}
I_1\coloneq \left[\smallmatrix\alpha_1&0&0\\0&\alpha_2&0\\0&0&\alpha_3
\endsmallmatrix\right],\quad
I_2\coloneq \left[\smallmatrix-v_1^2\beta_{23}+u_1^2\beta_{13}+\beta_3&v_1
\overline v_2\beta_{23}-u_2u_1\beta_{13}&v_1\overline
v_3\beta_{23}-u_3u_1\beta_{13}\\-v_1v_2\beta_{23}+u_1u_2\beta_{13}&
|v_2|^2\beta_{23}-u_2^2\beta_{13}+\beta_3&v_2\overline
v_3\beta_{23}-u_3u_2\beta_{13}\\-v_1v_3\beta_{23}+u_1u_3\beta_{13}&
\overline v_2v_3\beta_{23}-u_2u_3\beta_{13}&|v_3|^2\beta_{23}-u_3^2
\beta_{13}+\beta_3\endsmallmatrix\right], \quad I_3\coloneq (I_2I_1)^{-1}.
\end{equation*}

The vectors $u=(u_1,u_2,u_3)$ and $v=(v_1,v_2,v_3)$ are the eigenvectors of $I_2$ with respect to the eigenvalues $\beta_1$ and $\beta_2$. 

From these explicit formulas, it is straightforward the construction of representations with prescribed eigenvalues. We discuss our computational results in Section \ref{section computational results}.
\end{rmk}
\noindent

\subsection{Proof of Proposition \ref{prop main result}} \label{section proof of main result}

We investigate isometries $I_1,I_2,I_3 \in \PU$ such that $I_3I_2I_1 = 1$ and $I_j$ has order $n_j$.  As stated in proposition \ref{prop main result}, we are interested in representations where at least two of the three isometries are regular elliptic isometries. We may assume that $I_1, I_3$ are regular elliptic isometries.
\smallskip

Let $I_1,I_2,I_3\in\SU$ denote elliptic isometries in given conjugacy
classes: $\alpha_i$, $\beta_i$, and $\gamma_i^{-1}$, $i=1,2,3$, stand
respectively for the eigenvalues of $I_1$, $I_2$, and $I_3$. The first eigenvalue of each $I_i$ has a negative eigenvector.

Since we are considering representations up to conjugation, we can take $I_1$ to be diagonal. 
$$I_1=\left[\smallmatrix\alpha_1&0&0\\0&\alpha_2&0\\0&0&\alpha_3
\endsmallmatrix\right]$$

To determine an elliptic isometry $I_2$ with eigenvalues $\beta_1,\beta_2, \beta_3$ such that $I_3\coloneq (I_2I_1)^{-1}$ is a regular elliptic isometry with eigenvalues $\gamma_1^{-1},\gamma_2^{-1},\gamma_3^{-1}$ we have to solve the trace equation \eqref{trace equation}.
As previously said, by \cite[p.~204, Theorem 6.2.4]{goldmanbook}, the trace equation holds if and only if $I_2I_1$ is a
regular elliptic isometry with eigenvalues
$\gamma_1,\gamma_2,\gamma_3$.

We consider first the case where $I_2$ is regular elliptic and later the case where $I_2$ special elliptic, which is separated into two subcases: $I_2$ is rotation about a point in $\HH_\CC^2$ or $I_2$ is rotation about a complex geodesic in $\HH_\CC^2$.

\medskip

\subsubsection{Regular case: $I_2$ is regular elliptic.} 
Let
$u,v \in V$ denote eigenvectors  of $I_2$ corresponding to the eigenvalues $\beta_1$ and $\beta_2$, with $\langle u,u \rangle = -1$ and $\langle v,v \rangle = 1$. 
We fix an orthonormal basis $\mathcal B$ for $V$ of signature $-++$ consisting of eigenvectors of $I_1$. The corresponding eigenvalues are
$\alpha_1,\alpha_2,\alpha_3$. In this basis, we write
$$u=\left[\smallmatrix u_1\\u_2\\u_3\endsmallmatrix\right],\quad
v=\left[\smallmatrix v_1\\v_2\\v_3\endsmallmatrix\right].$$
Since we deal with generic representations we can assume that the eigenvalues of $I_1, I_2$ are pairwise non-orthogonal. In particular, the coordinates of $u,v$ in the basis $\mathcal B$ are non-zero.

Up to multiplying each vector of the basis $\mathcal B$ by unit complex numbers, we can assume that
$$u_1,u_2,u_3>0,\quad -u_1^2+u_2^2+u_3^2=-1$$ and up to multiplying $v$ by a unit complex number, we can assume, $v_1>0$.

Note that
\begin{equation}\label{condition u e v} -u_1^2+u_2^2+u_3^2=-1,\quad-v_1^2+|v_2|^2+|v_3|^2=1,
\quad-u_1v_1+u_2v_2+u_3v_3=0.
\end{equation}

Let us write down the trace equation
\eqref{trace equation} in the basis $\mathcal B$. We
define
\begin{equation}\label{definition s e t}\sqrt s\coloneq u_2,\quad\sqrt
t\coloneq u_3,\quad\beta_{ij}\coloneq \beta_i-\beta_j,\quad\alpha_{ij}\coloneq 
\alpha_i-\alpha_j,
\end{equation}
for $i,j=1,2,3$. Hence, $u_1=\sqrt{1+s+t}$ and $\alpha_{ij},\beta_{ij}\ne0$
if $i\ne j$. Using \eqref{general formula for elliptic isometry}, we write $I_1,I_2$ in the basis
$\mathcal B$:

\begin{equation}\label{matrices I1, I2}
I_1=\left[\smallmatrix\alpha_1&0&0\\0&\alpha_2&0\\0&0&\alpha_3
\endsmallmatrix\right],\qquad
I_2=\left[\smallmatrix-v_1^2\beta_{23}+u_1^2\beta_{13}+\beta_3&v_1
\overline v_2\beta_{23}-u_2u_1\beta_{13}&v_1\overline
v_3\beta_{23}-u_3u_1\beta_{13}\\-v_1v_2\beta_{23}+u_1u_2\beta_{13}&
|v_2|^2\beta_{23}-u_2^2\beta_{13}+\beta_3&v_2\overline
v_3\beta_{23}-u_3u_2\beta_{13}\\-v_1v_3\beta_{23}+u_1u_3\beta_{13}&
\overline v_2v_3\beta_{23}-u_2u_3\beta_{13}&|v_3|^2\beta_{23}-u_3^2
\beta_{13}+\beta_3\endsmallmatrix\right].
\end{equation}
The trace equation \eqref{trace equation} takes the form
\[\alpha_1(-v_1^2\beta_{23}+u_1^2\beta_{13}+\beta_3)+
\alpha_2\big(|v_2|^2\beta_{23}-u_2^2\beta_{13}+\beta_3\big)+
\alpha_3\big(|v_3|^2\beta_{23}-u_3^2\beta_{13}+\beta_3\big)=
\sum_{i=1}^3\gamma_i\]
which is equivalent to
\begin{equation} \label{dirty trace identity}|v_2|^2\alpha_{21}+|v_3|^2\alpha_{31}=\frac{\beta_{13}}{\beta_{23}}
\big(\alpha_{21}s+\alpha_{31}t\big)+k
\end{equation}
in view of the first two equalities in \eqref{condition u e v} and in \eqref{definition s e t}, where
\[k\coloneq \frac1{\beta_{23}}\left(\sum\limits_{i=1}^3\gamma_i-\alpha_1(\beta_1+
\beta_2-\beta_3)-\beta_3(\alpha_2+\alpha_3)\right).\]

\medskip

We rewrite equation \eqref{dirty trace identity} so that $|v_2|^2$ and $|v_3|^2$ are
explicitly given in terms of $s$ and $t$.

\medskip
\begin{lemma}\label{lemma 1 for I2 regular}
The determinant of
$M\coloneq \left[\smallmatrix\real\alpha_{21}&\real\alpha_{31}\\
\imag\alpha_{21}&\imag\alpha_{31}\endsmallmatrix\right]$
does not vanish. The trace equation is equivalent to the equations
\begin{gather*}
|v_2|^2\det M=s\imag\frac{\alpha_{31}\overline\alpha_{21}\overline
	\beta_{13}}{\overline\beta_{23}}+t|\alpha_{31}|^2\imag\frac{\overline
	\beta_{13}}{\overline\beta_{23}}+\imag(\alpha_{31}\overline k),\\
|v_3|^2\det M=s|\alpha_{21}|^2\imag\frac{\beta_{13}}{\beta_{23}}+
t\imag\frac{\overline\alpha_{21}\alpha_{31}\beta_{13}}{\beta_{23}}+
\imag(\overline\alpha_{21}k).
\end{gather*}
The coefficient of\/ $t$ in the first equation and that of\/ $s$ in the
the second equation does not vanish.
\end{lemma}

\medskip

{\bf Proof.} Note that $|\det M|$ is twice the area of the triangle with vertices $\alpha_1,\alpha_2,\alpha_3$ in the unit circle. Since $\alpha_1,\alpha_2,\alpha_3$ are pairwise distinct, this triangle has a non-vanishing area. Similarly, $\imag\displaystyle\frac{\beta_{13}}{\beta_{23}}\ne0$ because $\displaystyle\frac{\beta_{13}}{\beta_{23}}$ determines an internal angle of the triangle with vertices $\beta_1,\beta_2,\beta_3$.

The trace equation \eqref{trace equation} is equivalent to
$$|v_2|^2\real\alpha_{21}+|v_3|^2\real\alpha_{31}=\real z,\qquad
|v_2|^2\imag\alpha_{21}+|v_3|^2\imag\alpha_{31}=\imag z,$$
where
$z\coloneq \displaystyle\frac{\beta_{13}}{\beta_{23}}\big(\alpha_{21}s+
\alpha_{31}t\big)+k$.
Hence,
$$|v_2|^2\det M=\imag\alpha_{31}\real z-\real\alpha_{31}\imag
z=\imag(\alpha_{31}\overline z),$$
$$|v_3|^2\det M=\real\alpha_{21}\imag z-\imag\alpha_{21}\real
z=-\imag(\alpha_{21}\overline z).\eqno{_\blacksquare}$$

\medskip

Recall that the coordinates of $u,v$ are all non-zero and, consequently, $|v_2|^2>0$, $|v_3|^2>0$, and $|v_2|^2+|v_3|^2-1>0$, where this last allow us to compute $v_1 = \sqrt{|v_2|^2+|v_3|^2-1}$. Thus, following Lemma \ref{lemma 1 for I2 regular}, we have the following inequalities:

\begin{equation}\tag{\bf C1} \label{conditions C1}\begin{matrix}\displaystyle|v_2|^2=\frac1{\det
M}\big(s\imag\displaystyle\frac{\alpha_{31}\overline\alpha_{21}\overline
\beta_{13}}{\overline\beta_{23}}+t|\alpha_{31}|^2\imag\frac{\overline
\beta_{13}}{\overline\beta_{23}}+\imag(\alpha_{31}\overline k)\big)>0,

\\\\\displaystyle |v_3|^2=\frac1{\det M}\big(s|\alpha_{21}|^2\imag\displaystyle\frac{\beta_{13}}{\beta_{23}}+t\imag\frac{\overline\alpha_{21}\alpha_{31}\beta_{13}}{\beta_{23}}+\imag(\overline\alpha_{21}k)\big)>0,

\\\\\displaystyle v_1^2=-1+|v_2|^2+|v_3|^2>0.
\end{matrix}
\end{equation}

Conversely, if Condition \ref{conditions C1} holds for a pair
$(s,t)\in\mathbb R_{>0}\times\mathbb R_{>0}$ of positive real numbers, then
the equations in Lemma \ref{lemma 1 for I2 regular} and the second equation in \eqref{condition u e v} provide
the positive real numbers $v_1,|v_2|,|v_3|$.

In the lemma below, we state a condition, referred to as Condition \ref{condition C2},
that characterizes the possibility of expressing $v_2$ and $v_3$ in
terms of $s,t,v_1,|v_2|,|v_3|$.

\medskip
\begin{lemma}\label{lemma 2 for I2 regular}
 We have
$$v_2=\frac1{2v_1\sqrt{s(1+s+t)}}\big(-t|v_3|^2+(1+s+t)v_1^2+
s|v_2|^2\pm
i\sqrt\Delta\big),$$
$$v_3=\frac1{2v_1\sqrt{t(1+s+t)}}\big(-s|v_2|^2+(1+s+t)v_1^2+
t|v_3|^2\mp
i\sqrt\Delta\big),$$
where
\begin{equation}\label{condition C2} \tag{\bf C2}
\Delta\coloneq 4v_1^2|v_2|^2s(1+s+t)-\big(-t|v_3|^2+(1+s+t)v_1^2+
s|v_2|^2\big)^2\ge0.
\end{equation}

Reciprocally, let $s,t,v_1,|v_2|,|v_3|$ be given positive real
numbers such that $-v_1^2+|v_2|^2+|v_3|^2=1$ and $\Delta\ge0$.
Then\/ $v_2,v_3$ are well defined in terms of $s,t,v_1,|v_2|,|v_3|$
as above and satisfy $-u_1v_1+u_2v_2+u_3v_3=0$.
\end{lemma}

\medskip

{\bf Proof.} The third equality in \eqref{condition u e v} implies that
$\real\, v_3=\displaystyle\frac{u_1v_1-u_2\real\,v_2}{u_3}$ and
$\imag\, v_3=-\displaystyle\frac{u_2\imag\, v_2}{u_3}$. So,
$$|v_3|^2=\displaystyle\frac{(u_1v_1-u_2\real\, v_2)^2+u_2^2(\imag\,
v_2)^2}{u_3^2}=\frac{u_1^2v_1^2-2v_1u_1u_2\real\,
v_2+u_2^2|v_2|^2}{u_3^2},$$
that is,
$\real\, v_2=\displaystyle\frac
{-u_3^2|v_3|^2+u_1^2v_1^2+u_2^2|v_2|^2}{2v_1u_1u_2}.$
It follows that
$$\imag\, v_2=\frac{\sigma_1}{2v_1u_1u_2}\sqrt
{4v_1^2|v_2|^2u_1^2u_2^2-\big(-u_3^2|v_3|^2+u_1^2v_1^2+
u_2^2|v_2|^2\big)^2},$$
where $\sigma_1\in\{-1,1\}$. By symmetry,
$\real\, v_3=\displaystyle\frac
{-u_2^2|v_2|^2+u_1^2v_1^2+u_3^2|v_3|^2}{2v_1u_1u_3}$
and
$$\imag\, v_3=\frac{\sigma_2}{2v_1u_1u_3}\sqrt
{4v_1^2|v_3|^2u_1^2u_3^2-\big(-u_2^2|v_2|^2+u_1^2v_1^2+
u_3^2|v_3|^2\big)^2},$$
where $\sigma_2\in\{-1,1\}$. Taking $r\coloneq u_2^2|v_2|^2-u_3^2|v_3|^2$ in
the tautological equality
$$4v_1^2u_1^2r-(u_1^2v_1^2+r)^2+(u_1^2v_1^2-r)^2=0,$$
we obtain
$$4v_1^2|v_2|^2u_1^2u_2^2-\big(-u_3^2|v_3|^2+u_1^2v_1^2+
u_2^2|v_2|^2\big)^2=4v_1^2|v_3|^2u_1^2u_3^2-\big(-u_2^2|v_2|^2+
u_1^2v_1^2+u_3^2|v_3|^2\big)^2.$$
It follows from $u_2\imag v_2+u_3\imag v_3=0$ that $\sigma_2=-\sigma_1$.

A straightforward computation implies the converse.\hfill$_\blacksquare$

\medskip

Summarizing: Lemmas \ref{lemma 1 for I2 regular} and \ref{lemma 2 for I2 regular} imply that Conditions \ref{conditions C1} and \ref{condition C2} are valid for an isometry $I_2$ satisfying the trace equation.
Reciprocally, given $(s,t)\in\mathbb R_{>0}\times\mathbb R_{>0}$ such that
\ref{conditions C1} holds, we take the point $u$ with coordinates $u_1\coloneq \sqrt{1+s+t}$,
$u_2\coloneq \sqrt s$, and $u_3\coloneq \sqrt t$. Clearly, $\langle u,u\rangle=-1$.
The equations in Lemma \ref{lemma 1 for I2 regular} as well as the second equation in \eqref{condition u e v}
provide the positive numbers $v_1,|v_2|,|v_3|$. Suppose that \ref{condition C2}
holds. Choosing a sign in the formulae for $v_2$ and $v_3$ in Lemma
\ref{lemma 2 for I2 regular}, we get the point $v$ with coordinates $v_1,v_2,v_3$ such that
$\langle v,v\rangle=1$. By Lemma \ref{lemma 2 for I2 regular}, $\langle u,v\rangle=0$. We have
just constructed an isometry $I_2$ with the fixed points $u,v$ (and the
third fixed point uniquely determined by $u,v$) satisfying the trace
equation. The coordinates $s,t$ are geometrical invariants of the
representation $\rho: G\to\PU$, $\rho:g_i\mapsto I_i$ ($G$ is the
turnover group defined in Section \ref{turnover definition}). Indeed, $\ta(u,L_1)=1+s$ and
$\ta(u,L_2)=1+t$, where $L_1,L_2$ stand for the $I_1$-stable complex
geodesics. In other words, we parameterized the generic part of the
representation space in question. Let us briefly discuss the role of
the sign in the formulae for $v_2,v_3$.

The isometries $I_2$ and $I_2'$ determined by the different choices of
sign in the formulae for $v_2,v_3$ in Lemma \ref{lemma 2 for I2 regular} are related as
follows. Let $u,v,w$ and $u,v',w'$ stand respectively for the fixed
points of $I_2$ and $I_2'$ ($w,w'$ are the points in $\mathbb P(u^\perp)$
orthogonal respectively to $v,v'$). In the basis $\mathcal B$, the
reflection $R$ in the $\mathbb R$-plane $\mathbb P(W)$ ($W\subset V$ is spanned
over $\mathbb R$ by $\mathcal B$) corresponds to the complex conjugation of
coordinates. Obviously, $u\in\mathbb P(W)$, $Ru=u$, and $Rv=v'$. This
implies that
$\langle Rw,u\rangle=\overline{\langle
w,Ru\rangle}=\overline{\langle w,u\rangle}=0$,
i.e., $Rw\in\mathbb P(u^\perp)$. Analogously, $\langle Rw,v'\rangle=0$. We
obtain $Rq=q'$. In other words, the fixed points of $I_2'$ are those of
$I_2$ reflected in $\mathbb P(W)$. Since the eigenvalues of $I^R\coloneq RIR^{-1}$ are
complex conjugate to those of $I$, we obtain $I_1^R=I_1^{-1}$ and
$I_2^R={I'_2}^{-1}$. So, the group generated by $I_1,I'_2,I'_3$
comes from the one generated by $I_1^{-1},I_2^{-1},I_2^{-1}I_3^{-1}I_2$.

\bigskip

\subsubsection{ Special case: $I_2$ is a rotation about a point in $\HH_\CC^2$.} 
Let $u\in\HH_\CC^2$ denote the center of $I_2$ with corresponding eigenvalue $\beta_1$. We fix a basis $\mathcal B$ in $V$ of signature $-++$ consisting of eigenvectors of $I_1$ with corresponding eigenvalues $\alpha_1,\alpha_2,\alpha_3$. In this basis, we write
$u=\left[\smallmatrix u_1\\u_2\\u_3\endsmallmatrix\right]$. We can
assume that $u_1,u_2,u_3\ge0$ and $\langle u,u\rangle=-1$. In other
words, $-u_1^2+u_2^2+u_3^2=-1$.

Let us write down the trace equation
\eqref{trace equation} in the basis $\mathcal B$. Define $\beta_{ij}\coloneq \beta_i-\beta_j$ and
$\alpha_{ij}\coloneq \alpha_i-\alpha_j.$ In particular, $\beta_{23}=0$. It
follows from \eqref{general formula for elliptic isometry} that
$$I_1=\left[\smallmatrix\alpha_1&0&0\\0&\alpha_2&0\\0&0&\alpha_3
\endsmallmatrix\right],\qquad
I_2=\left[\smallmatrix
u_1^2\beta_{13}+\beta_3&-u_2u_1\beta_{13}&-u_3u_1\beta_{13}\\
u_1u_2\beta_{13}&-u_2^2\beta_{13}+\beta_3&-u_3u_2\beta_{13}\\
u_1u_3\beta_{13}&-u_2u_3\beta_{13}&-u_3^2\beta_{13}+\beta_3
\endsmallmatrix\right].$$
The trace equation takes the form
$$\alpha_1(u_1^2\beta_{13}+\beta_3)+
\alpha_2(-u_2^2\beta_{13}+\beta_3)+
\alpha_3(-u_3^2\beta_{13}+\beta_3)=\sum_{i=1}^3\gamma_i$$
which is equivalent to
$$u_2^2\alpha_{12}+u_3^2\alpha_{13}=k,\quad
k\coloneq \frac1{\beta_{12}}\big(\sum\limits_{i=1}^3\gamma_i-\alpha_1\beta_1-
\beta_2(\alpha_2+\alpha_3)\big).$$

\begin{lemma}\label{lemma 1 for I2 special}
The determinant of\/
$M\coloneq \left[\smallmatrix\real\alpha_{21}&\real\alpha_{31}\\
\imag\alpha_{21}&\imag\alpha_{31}\endsmallmatrix\right]$
does not vanish. The trace equation is equivalent to the equations
$$u_2^2\det M=\imag(\alpha_{13}\overline k),\qquad
u_3^2\det M=\imag(\alpha_{21}\overline k).$$	
\end{lemma}

{\bf Proof.} The fact $\det M\ne0$ is proven exactly as in the
beginning of the proof of Lemma \ref{lemma 1 for I2 regular}. The trace equation is equivalent
to $u_2^2\real\alpha_{12}+u_3^2\real\alpha_{13}=\real k$ and
$u_2^2\imag\alpha_{12}+u_3^2\imag\alpha_{13}=\imag k$. Hence,
$$u_2^2\det M=\imag\alpha_{13}\real k-\real\alpha_{13}\imag
k=\imag(\alpha_{13}\overline k)$$
$$u_3^2\det M=\real\alpha_{12}\imag k-\imag\alpha_{12}\real
k=-\imag(\alpha_{12}\overline k).\eqno{_\blacksquare}$$

\medskip

By Lemma \ref{lemma 1 for I2 special}, the trace equation implies
$$\det M\imag(\alpha_{13}\overline k)\geqslant0,\qquad\det
M\imag(\alpha_{21}\overline k)\geqslant0.$$
Conversely, if the above inequalities hold, we obtain from Lemma \ref{lemma 1 for I2 special}
and from $-u_1^2+u_2^2+u_3^2=-1$ the negative point $u$ with
coordinates $u_1,u_2,u_3$. The corresponding isometry $I_2$ satisfies
the trace equation. Hence, the component of the space $\mathcal R$ of
conjugacy classes of representations $\rho:G\to\PU$ corresponding to the given conjugacy classes of
$I_1,I_2,I_3$ is either empty or a point. 

\smallskip

We have a similar result in
the case of rotation about a complex geodesic.

\subsubsection{Special case: $I_2$ is a rotation about a complex geodesic in $\HH_\CC^2$.}
Let $I_2$ be a rotation about the complex geodesic
$\mathbb P(v^\perp)$, where $v$ an eigenvector of $I_2$ for $\beta_2$. We
fix an orthogonal basis of eigenvectors of $I_1$ (the eigenvalues are
$\alpha_1,\alpha_2,\alpha_3$). In this basis, we write
$v=\left[\smallmatrix v_1\\v_2\\v_3\endsmallmatrix\right]$
and assume that $v_1,v_2,v_3\ge0$ and that $-v_1^2+v_2^2+v_3^2=1$. The
determinant of
$M\coloneq \left[\smallmatrix\real\alpha_{21}&\real\alpha_{31}\\
\imag\alpha_{21}&\imag\alpha_{31}\endsmallmatrix\right]$
does not vanish (see Lemma \ref{lemma 1 for I2 regular}) and the trace equation \eqref{trace equation} is equivalent
to the equations
$$v_2^2\det M=\imag(\alpha_{13}\overline k),\qquad
v_3^2\det M=\imag(\alpha_{21}\overline k),$$
where
$$k\coloneq \frac1{\beta_{12}}\big(\sum\limits_{i=1}^3\gamma_i-
\alpha_1\beta_2-\beta_1(\alpha_2+\alpha_3)\big).$$
The trace equation and the equation $-v_1^2+v_2^2+v_3^2=1$ imply
$$\det M\imag(\alpha_{13}\overline k)\ge0,\qquad\det
M\imag(\alpha_{21}\overline k)\ge0,\qquad\frac1{\det M}
\big(\imag(\alpha_{13}\overline k)+\imag(\alpha_{21}\overline
k)\big)\ge1.$$
Conversely, if the above inequalities hold, we obtain from the trace
equation and from the equation $-v_1^2+v_2^2+v_3^2=1$ the positive
point $v$ with coordinates $v_1,v_2,v_3$.

\newpage

\section{Discreteness: fundamental quadrangle of bisectors}
\label{section discreteness}
\subsection{Quadrangle of bisectors}\label{subsection quadrangle of bisectors}

\begin{wrapfigure}{r}{5. cm}
	\vspace*{-0.8cm}
	\centering
	\includegraphics[scale=.7]{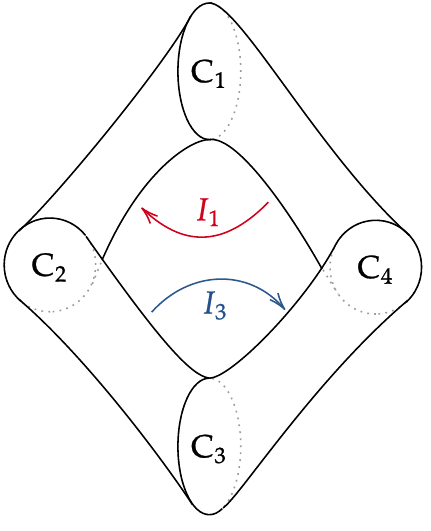}
	\vspace*{-0.8cm}
	\caption{Quadrangle $\mathcal Q$}
	\label{quadrangle of bisectors}
\end{wrapfigure}

Following \cite{discbundles}, we introduce the
{\it quadrangle\/} of bisectors associated to representations $\rho:G(n_1,n_2,n_3)\to\PU$ discussed in Section \ref{section generic representations}. We denote $G(n_1,n_2,n_3)$ by $G$. The quadrangles will depend on the parameters outlined in Section \ref{section generic representations}. In the case of generic representations where $\rho(g_1),\rho(g_2),\rho(g_3)$ are regular, these parameters are $s,t$ (see Remark \ref{remark alg}). We expect quadrangles of bisectors to bound fundamental polyhedra for discrete actions of $G$ on $\HH_\CC^2$ and the quotient $\HH_\CC^2/G$ to be a disc orbibundle over an orbifold (in our case, a sphere with three cone points). Passing to a finite index subgroup of $G$, one arrives at a complex hyperbolic disc bundle over a closed orientable surface (this comes from the fact that a finitely generated Fuchsian group always has a finite index torsion-free subgroup). The representations will provide discrete representations when the parameters have certain values, obtained computationally and outlined in Section~\ref{section computational results}.

\smallskip

We remind here a few definitions from \cite{discbundles}.

In order to orient a bisector $B$ we only need to orient its real spine (since the fibers are complex, hence, naturally oriented). An oriented bisector $B$ divides ${\overline\HH_\CC^2}$ into two {\it half-spaces\/} (closed $4$-balls) $K^+$ and $K^-$, where $K^+$ stands for the half-space lying on the side of the normal vector to $B$. See Figure \ref{hypersurface}.
\begin{figure}[H]
	\centering
        \vspace*{-0.3cm}
	\includegraphics[scale=.7]{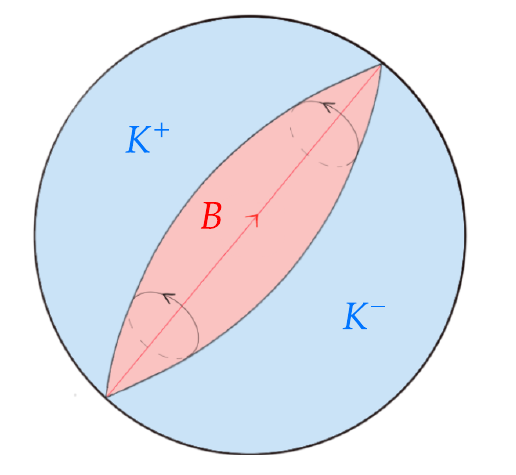}
	\caption{Orienting the real spine fixes a unique orientation of the bisector since the slices are naturally oriented. The half-space $K^+$ is the one on the side of the normal vector.}
	\label{hypersurface}
\end{figure}

Let $B_1=B_1[C_1,C_2]$ and $B_2=B_2[C_1,C_3]$ be two oriented segments of bisectors with a common slice $C_1$ such that the corresponding full bisectors are transversal along that slice. The {\it sector\/} from $B_1$ to $B_2$ is defined to be either $K_1^+\cap K_2^-$ (when the oriented angle from $B[C_1,C_2]$ to $B[C_1,C_3]$ at a point $c\in C_1$ is smaller than $\pi$) or
$K_1^+\cup K_2^-$ (when the oriented angle from $B[C_1,C_2]$ to $B[C_1,C_3]$ at a point $c\in C_1$ is greater than $\pi$).
Note that, while such oriented angle does depend on the point $c$, it cannot equal $\pi$ due to transversality. See Figure \ref{sector}.

\begin{figure}[H]
	\centering
        \vspace*{-.8cm}
	\includegraphics[scale=.6]{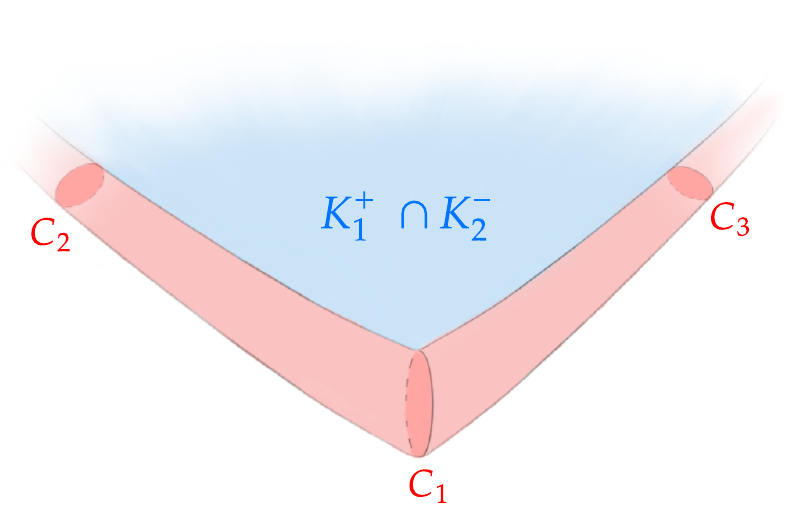}
	\caption{Sector between given by $B_1$ and $B_2$ when the angle between them is smaller than $\pi$.}
	\label{sector}
\end{figure}

Given pairwise ultraparallel complex geodesics $C_1,C_2,C_3$, the oriented triangle of bisectors $\Delta(C_1,C_2,C_3)$ is simply the union $B[C_1,C_2]\cup B[C_2,C_3]\cup B[C_3,C_1]$ of oriented segments of bisectors. Each such segment is a side of the oriented triangle and each of the complex geodesics $C_1,C_2,C_3$ is a vertex of the triangle. The triangle is {\it transversal\/} if the full bisectors containing its sides intersect transversally along the common slices.

Given three ultraparallel complex geodesics $C_1,C_2,C_3$ there are two possible orientations for a triangle of bisectors with vertices $C_1,C_2,C_3$. Assuming that such a triangle is transversal, its {\it counterclockwise orientation} is the one providing an oriented angle from $B[C_3,C_1]$ to $B[C_1,C_2]$ in~$(0,\pi)$. By \cite[Lemma 2.14]{discbundles}, this implies that oriented angles from $B[C_1,C_2]$ to $B[C_2,C_3]$ and from $B[C_2,C_3]$ to $B[C_3,C_1]$ lie in $(0,\pi)$ as well; moreover, in this case, each side of the triangle is contained in the sector determined by the other two.

Let $\Delta(C_1,C_2,C_4)$ and $\Delta(C_3,C_4,C_2)$ be counterclockwise oriented transversal triangles of bisectors sharing a common side (see Figure \ref{quadrangleadjacent}). We say these triangles are {\it transversally adjacent\/} if the sector at $C_1$ contains a point of $C_3$ and the full bisectors containing the segments $B[C_1,C_2]$ and $B[C_2,C_3]$ (respectively, $B[C_3,C_4]$ and $B[C_4,C_1]$) are transversal at $C_2$ (respectively, at $C_4$). In particular (see \cite[Lemma 2.14]{discbundles}), this implies that $\Delta(C_3,C_4,C_2)$ is contained in the sector at $C_1$; furthermore, $\Delta(C_3,C_4,C_2)$ and $\Delta(C_1,C_2,C_4)$ lie in opposite sides of the full bisector containing $B[C_2,C_4]$.
\begin{figure}[H]
	\centering
	\includegraphics[scale=.6]{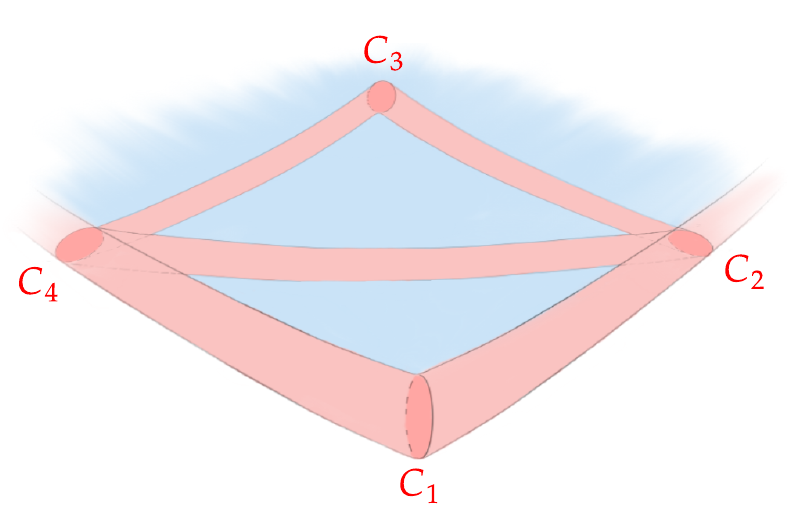}
	\caption{Transversally adjacent triangles.}
	\label{quadrangleadjacent}
\end{figure}
\medskip

Let $I_1,I_2,I_3$ be elliptic isometries as in Section \ref{section generic representations}. Let $c_j,p_j\in\mathbb P_\CC(V)$ denote pairwise orthogonal distinct fixed points of $I_j$ with $c_j\in\HH_\CC^2$. In particular, the $p_j$'s are positive. We also define the points $c_4 \coloneq  I_1^{-1}c_2$ and $p_4 \coloneq  I_1^{-1}p_2$ and the complex geodesics
\begin{equation*}
C_1\coloneq \mathbb P(p_1)^\perp\cap\overline{\HH}_\CC^2, \qquad
C_2\coloneq \mathbb P(p_2)^\perp\cap\overline{\HH}_\CC^2,
\qquad
C_3\coloneq \mathbb P(p_3)^\perp\cap\overline{\HH}_\CC^2, \qquad
C_4\coloneq \mathbb P(p_4)^\perp\cap\overline{\HH}_\CC^2.
\end{equation*}

Using the relation $I_3I_2I_1 = 1$ we obtain
$p_4=I_1^{-1}p_2=I_3p_2$. Therefore $C_4 = I_1^{-1}C_2 = I_3 C_2$.
Note that, by Remark \ref{simplecgproperties}, if $C_1$ and $C_2$ are ultraparallel, $C_1\!\!\parallel\!\!C_2$, then $C_1\!\!\parallel\!\!C_4$ since $\ta(p_1,p_2)=\ta(I_1^{-1}p_1,I_1^{-1}p_2) =\ta(p_1,p_4)>1$. Similarly, $C_3\!\!\parallel\!\!C_2$ implies $C_3\!\!\parallel\!\!C_4$. So, if $C_1\!\!\parallel\!\!C_2$, $C_3\!\!\parallel\!\!C_2$, and\/ $C_2\!\!\parallel\!\!C_4$, we get the oriented triangles of bisectors $\Delta(C_1,C_2,C_4)$ and $\Delta(C_3,C_4,C_2)$.

\subsection{Quadrangle conditions}\label{subsection quadrangle conditions} A representation $\rho:G(n_1,n_2,n_3)\to\PU$, $g_j\mapsto I_j$, satisfies the\/ {\rm quadrangle} conditions if

\begin{itemize}
	\item[(Q1)] $C_1\!\!\parallel\!\!C_2$, $C_3\!\!\parallel\!\!C_2$, and\/ $C_2\!\!\parallel\!\!C_4$.
	\item[(Q2)] The triangles $\Delta(C_1,C_2,C_4),\Delta(C_3,C_4,C_2)$ are
	transversal and counterclockwise-oriented.
	\item[(Q3)] The triangles $\Delta(C_1,C_2,C_4),\Delta(C_3,C_4,C_2)$ are
	transversally adjacent;
	\item[(Q4)]  The oriented angle from $\B[C_1,C_4]$ to $\B[C_1,C_2]$ at $c_1$
	equals $\displaystyle\frac{2\pi}{n_1}$; the oriented angle from
	$\B[C_3,C_2]$ to $\B[C_3,C_4]$ at $c_3$ equals
	$\displaystyle\frac{2\pi}{n_3}$; the sum of the oriented angle from
	$\B[C_2,C_1]$ to $\B[C_2,C_3]$ at $c_2$ with the oriented angle from
	$\B[C_4,C_3]$ to $\B[C_4,C_1]$ at $c_4$ equals
	$\displaystyle\frac{2\pi}{n_2}$.
\end{itemize}

A representation satisfying the quadrangle conditions gives rise to the {\it quadrangle of bisectors}
$$\mathcal Q: = B[C_1,C_2]\cup B[C_2,C_3]\cup B[C_3,C_4]\cup B[C_4,C_1].$$
The quadrangle $\mathcal Q$ bounds polyhedron $Q$ which is on the side of the normal vectors of the oriented segments of bisectors. Indeed, by \cite[Lemma 2.13]{discbundles}, there are no intersections between those segments of bisectors besides the common slices.
\subsection{The quadrangle conditions explicitly}\label{quadrangle revisited}
As in Section \ref{section generic representations}, let $\alpha_i$, $\beta_i$, $\gamma_i^{-1}$, $i=1,2,3$, stand respectively for the eigenvalues of $I_1$, $I_2$, $I_3$ such that
\begin{gather*}
I_1 (c_1) = \alpha_1 c_1,\, I_2 (c_2) = \beta_1 c_2, \,I_3 (c_3) = \gamma_1^{-1} c_3,\\
I_1 (p_1) = \alpha_2 p_1, \,I_2 (p_2) = \beta_2 p_2, \,I_3 (p_3) = \gamma_2^{-1} p_3.
\end{gather*}

Since $\det I_k = 1$, we must have $\alpha_1 \alpha_2 \alpha_3 = \beta_1 \beta_2 \beta_3 = \gamma_1 \gamma_2 \gamma_3 = 1$.

 Let us formulate the quadrangle conditions from Section \ref{subsection quadrangle conditions} in terms of algebraic formulas that are used for practical computations.

Quadrangle condition (Q1) asks that the complex geodesics $C_1,C_2,C_3,C_4$ are pairwise ultraparallel:

\begin{equation}\label{Q1}\tag{Q1}
\ta(p_1,p_2) = \ta(p_1,p_4)>1, \quad \ta(p_2,p_3)>1,\quad \textrm{and}\quad \ta(p_3,p_2) = \ta(p_3,p_4)>1
\end{equation}
(see Remark \ref{simplecgproperties}).

Assuming (Q1) we can define the bisectors segments $\B[C_1,C_2]$, $\B[C_2,C_3]$, $\B[C_2,C_3]$, and $\B[C_3,C_4]$.

Condition (Q2) says that the triangles of bisectors $\Delta(C_1,C_2,C_4)$ and $\Delta(C_2,C_3,C_4)$ are transversal and counterclockwise oriented. Accordingly with \cite[Criterion 2.27]{discbundles}, this is equivalent to the inequalities
\begin{equation}\label{Q2}\tag{Q2}
\begin{aligned}
&\epsilon_0^2 t^2 + s^2 + t^2 < 1+2t^2s\epsilon_0, \quad \epsilon_0^2 s^2 + 2t^2  < 1+2t^2s\epsilon_0,  \quad \epsilon_1<0,
\\ &\epsilon_0'^2 t'^2 + s^2 + t'^2 < 1+2t'^2s\epsilon_0', \quad \epsilon_0'^2 s^2 + 2t'^2  < 1+2t'^2s\epsilon_0',  \quad \epsilon_1'<0,
\end{aligned}
\end{equation}
where
\begin{gather*}
t = \sqrt{\ta(p_1,p_2)}, \quad t'=\sqrt{\ta(p_3,p_2)}, \quad s=\sqrt{\ta(p_2,p_4)},\\
\epsilon_0+\epsilon_1 i : = \frac{\sigma}{|\sigma|},\quad \textrm{where}\quad \sigma: = \langle p_1,p_2 \rangle \langle p_2,p_4 \rangle \langle p_4,p_1 \rangle,\\
\epsilon_0'+\epsilon_1' i : = \frac{\sigma'}{|\sigma'|},\quad \textrm{where}\quad \sigma': = \langle p_2,p_3 \rangle \langle p_3,p_4 \rangle \langle p_4,p_2 \rangle.
\end{gather*}

The quadrangle condition $\textrm{(Q3)}$ asserts that the triangles $\Delta(C_1,C_2,C_4)$ and $\Delta(C_3,C_4,C_2)$ are transversally adjacent. It is guaranteed by the conditions $\textrm{(Q3.1),\,(Q3.2),\,(Q3.3)}$ below. Condition (Q3.1) concerns the transversality of the bisectors $B[C_1,C_2]$ and $B[C_2,C_3]$,
\begin{equation}\label{q31}\tag{Q3.1}
\left\vert \real\left( \frac{\langle p_3,p_1 \rangle \langle p_2,p_2 \rangle }{\langle p_3,p_2 \rangle \langle p_2,p_1 \rangle}\right) - 1 \right\vert < \sqrt{1 - \frac{1}{\ta(p_2,p_3)}}\sqrt{1 - \frac{1}{\ta(p_2,p_1)}}.
\end{equation}
(see \cite[Criterion 3.3]{discbundles}). Similarly, (Q3.2) states the transversality of the bisectors $B[C_1,C_4]$ and $B[C_4,C_3]$,
\begin{equation}\label{q32}\tag{Q3.2} \left\vert \real\left( \frac{\langle p_1,p_3 \rangle \langle p_4,p_4 \rangle }{\langle p_1,p_4 \rangle \langle p_4,p_3 \rangle}\right) - 1 \right\vert < \sqrt{1 - \frac{1}{\ta(p_4,p_1)}}\sqrt{1 - \frac{1}{\ta(p_4,p_3)}}.
\end{equation}
Finally, (Q3.3) implies that $c_3$ belongs to the interior of the sector at $C_1$ of the triangle $\Delta(C_1,C_2,C_4)$:
\begin{equation}\label{q33}\tag{Q3.3}
\imag \left( \frac{\langle p_1,c_3 \rangle \langle c_3,p_2 \rangle}{\langle p_1,p_2 \rangle }\right)> 0 \quad \textrm{and} \quad \imag \left( \frac{\langle p_4,c_3 \rangle \langle c_3,p_1 \rangle}{\langle p_4,p_1 \rangle }\right) > 0
\end{equation}
(see \cite[Lemma 3.5]{discbundles}).

Consider the polyhedron $Q$ bounded in $\HH_\CC^2$ by the quadrangle
\[\mathcal Q \coloneq B[C_1,C_2]\cup B[C_2,C_3]\cup B[C_3,C_4]\cup B[C_4,C_1].\]

It follows from \cite[Lemma 3.4]{discbundles} that condition $(\textrm{Q4})$ translates, in terms of the eigenvalues $\alpha_i,\beta_i,\gamma_i$, into
\begin{equation}\label{q4}\tag{Q4}
\alpha_2/\alpha_1=\exp(-2\pi i/n_1),\quad
\beta_2/\beta_1=\exp(-2\pi i/n_2),\quad
(\gamma_2/\gamma_1)^{-1}=\exp(-2\pi i/n_3)
\end{equation}

\subsection{Discreteness}

Let $\rho:G(n_1,n_2,n_3)\to\PU$, $g_j\mapsto I_j$ be a representation satisfying the quadrangle conditions and let $\mathcal Q$ be the quadrangle of $\rho$ described in the previous Section.

Applying \cite[Theorem 3.2]{poincareth} we will show that $Q\cap\HH_\CC^2$ is a fundamental region for the action of the group $K\coloneq \langle I_1,I_3 \rangle$, the image of $\rho$. Remember that $I_1,I_3$ satisfy the relations $I_1^{n_1}=I_3^{n_3}=(I_3^{-1}I_1^{-1})^{n_2}=1$ in $\PU$. The main idea is to prove that, given a point $x$ in the polyhedron $Q$, there are corresponding copies of the polyhedron $Q$ that tessellate a (small) ball centered at $x$. When $x$ belongs to the interior of $Q$, the fact is immediate; when it lies in the interior of a side, it follows from the fact that the elliptic isometries $I_1$ and $I_3$ send the interior of $Q$ to its exterior. Finally, when $x$ is in a vertex, it is enough to understand the case $x=c_j$. Here, the tessellation follows from an infinitesimal conditional: the local tessellation of the complex geodesic normal to the vertex at $c_j$ (see Figure \ref{localtessellation}). For more details, see \cite{poincareth}. This leads to a tessellation of a neighborhood of $Q$ and, by \cite[Lemma 2.10]{discbundles}, such a tessellation provides a tessellation of a metric neighborhood of $Q$. By \cite[Theorem 3.2]{poincareth}, $K$ is discrete.

\begin{figure}[H]
	\centering
	\begin{minipage}[b]{.5\textwidth}
		\centering
		\includegraphics[scale = .6]{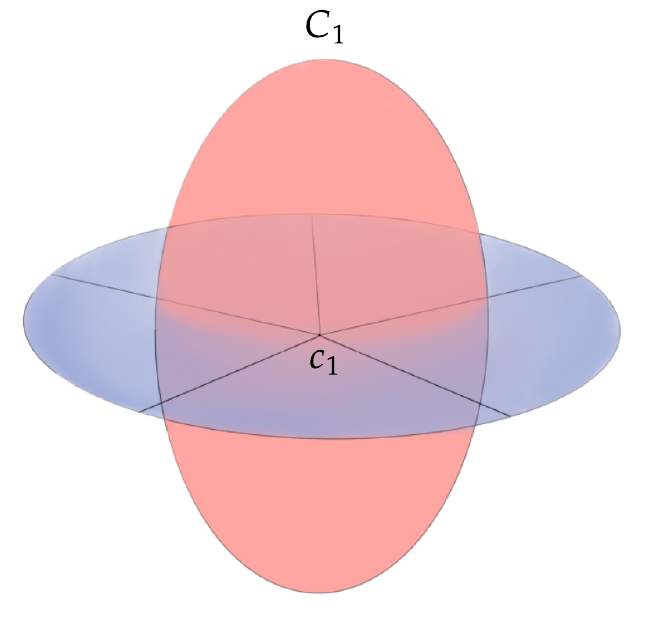}
		\caption*{(a)}
	\end{minipage}%
	\begin{minipage}[b]{.5\textwidth}
		\centering
		\includegraphics[scale = .74]{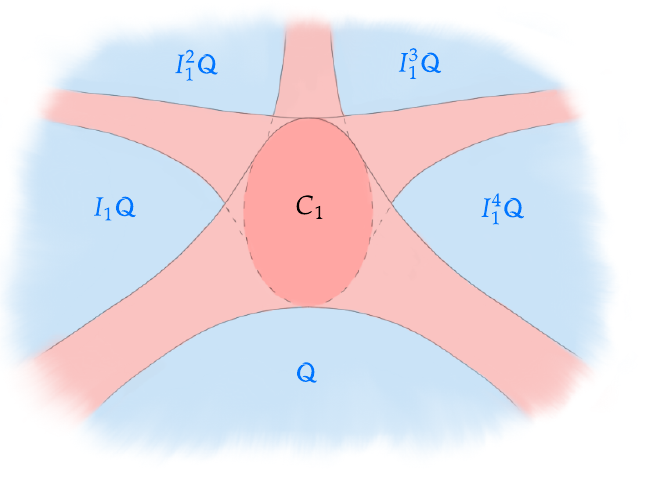}
		\caption*{(b)}
	\end{minipage}%

	\caption{\textbf{(a)} Tessellation of the complex geodesic normal to $C_1$ at $c_1$, and \textbf{(b)} tessellation around the vertex $C_1$. In both cases, $I_1$ has order $5$.}
 \label{localtessellation}
\end{figure}

\begin{thm} \label{teo: tessalation}
The group\/ $K$ is discrete and\/ $Q$ is a fundamental domain for its action on\/ $\HH_\CC^2$.
\end{thm}

\begin{proof}
By \cite[Lemma 2.10]{discbundles}, we only need to verify Conditions (i) and (ii) of \cite[Theorem 3.2]{poincareth}. Since $I_1$ maps $\B[C_1,C_4]$ onto $\B[C_1,C_2]$ and $I_3$ maps $\B[C_3,C_2]$ onto $\B[C_3,C_4]$, the Condition (i) of \cite[Theorem 3.2]{poincareth} follows from the definition of counterclockwise-oriented transversal triangles. There are three (geometrical) cycles of vertices. The cycle of $C_1$ have total angle $2\pi$ at $c_1\in C_1$ by \cite[Lemma 3.4]{discbundles}. The same concerns the cycle of $C_3$ at $c_3\in C_3$.

The geometric cycle of $C_2$ has length $2n_2$ due to the relation $I_2^{n_2}=1$. Let us verify that the total angle at $c_2\in C_2$ is $2\pi$. Note that $I_3^{-1}$ sends $\B[C_3,C_4]$ onto $\B[C_3,C_2]$ and sends $\B[C_4,C_1]$ onto $\B[C_2,I_2C_1]$ (indeed, $I_3^{-1}C_1=I_3^{-1}I_1^{-1}C_1=I_2C_1$). Therefore, by the definition of counterclockwise-orientation, the sum of the interior angle from $\B[C_2,C_3]$ to $\B[C_2,C_1]$ at $c_2$ with the interior angle from $\B[C_4,C_1]$ to $\B[C_4,C_3]$ at~$c_4$ equals the angle from $\B[C_2,C_3]$ to $\B[C_2,I_2C_3]$ at $c_2$. By \cite[Lemma 3.4]{discbundles}, this angle equals $\Arg(\beta_1^{-1}\beta_2)=-2\pi/n_2$.
\end{proof}

\newpage

\section{Computational results}\label{section computational results}
We now present examples of discrete faithful representations $\rho: G(n_1,n_2,n_3) \to \PU$.
Their construction is delineated in Section \ref{section discreteness} and all found representations define complex hyperbolic disc orbibundles by design. We denote $\rho(g_j)$ by $I_j$.

As previously said, we consider positive eigenvectors $p_1$, $p_2$, $p_3$ with eigenvalues $\alpha_2,\beta_2,\gamma_2^{-1}$ for $I_1$, $I_2$, $I_3$, respectively, and the complex geodesics $$C_1\coloneq \PP(p_1^\perp)\cap \HH_\CC^2,\quad C_2\coloneq \PP(p_2^\perp)\cap \HH_\CC^2,\quad C_3\coloneq \PP(p_3^\perp)\cap \HH_\CC^2,\quad C_4\coloneq \PP(p_4^\perp)\cap \HH_\CC^2,$$
where $p_4 \coloneq  I_3 p_2 = I_1^{-1} p_2$. Assuming the tessellation conditions outlined in Section \ref{section discreteness}, the quadrangle of bisectors (see Figure \ref{quadrangle of bisectors}) $$\mathcal Q: = B[C_1,C_2]\cup B[C_2,C_3]\cup B[C_3,C_4]\cup B[C_4,C_1]$$
bounds a $4$-ball which tessellates $\HH_\CC^2$ under action of $\rho$. The foliation of the bisector segments by complex geodesics (see Figure \ref{bisector1}) extends to a disc foliation of the $4$-ball bounded by $\mathcal Q$, giving rise to the disc orbibundle structure (see Section \ref{section: Orbifold bundles and Euler number}). 

It is also important to point out that the conditions imposed on the representation for the tessellation are given by strict inequality (see Section \ref{quadrangle revisited}). Thus, given a representation defining a complex hyperbolic disc orbibundle in the character variety, all other representations in a small neighborhood also define complex hyperbolic disc orbibundles.

As described in Remark \ref{remark disc orbibundle}, complex hyperbolic disc orbibundles have three invariants in play: the Euler characteristic
$$\chi \coloneq  -1 + \frac{1}{n_1}+\frac{1}{n_2}+\frac{1}{n_3}$$
of the sphere with $3$ conic points $\SP^2(n_1,n_2,n_3)$,
the Euler number $e$ of the disc orbibundle, and the Toledo invariant $\tau$, following techniques developed in \cite{bot}. For more details, see Section \ref{subsection euler number} and Section \ref{section Toledo invariant}. An example is detailed in Section \ref{section:Explicit example with trivial Euler number}, showing how to analyze examples of complex hyperbolic disc orbibundles explicitly. 

Let us first consider generic representation $G(n_1,n_2,n_3) \to \PU$ where the isometries $I_1$, $I_2$, $I_3$ are regular elliptic with eigenvalues $\alpha_i$'s, $\beta_i$'s, and $\gamma_i^{-1}$'s, respectively, following the Remark~\ref{remark alg}. 
In our computational investigation, we set the following eigenvalues:
{\footnotesize	
\begin{alignat*}{3}
&\alpha_1 = \exp\left(\frac{2 (n_1-k_1) \pi i}{3 n_1}\right), \quad
&&\alpha_2 = \exp\left(\frac{2 (n_1 -k_1 - 3) \pi i }{3 n_1}\right), \quad
&&\alpha_3 = \exp\left( \frac{2 (2 k_1 + n_1 +3) \pi i}{3 n_1}\right),\\ 
&\beta_1 = \exp\left(\frac{-2k_2 \pi i}{3 n_2}\right),\quad
&&\beta_2 = \exp\left(\frac{2 (-k_2 - 3) \pi i}{3 n_2}\right), \quad
&&\beta_3 = \exp\left( \frac{2 (2 k_2+3) \pi i)}{3 n_2}\right),\\
&\gamma_1^{-1} = \exp\left(\frac{2 (2 n_3 d - k_3) \pi i}{3 n_3}\right), \quad
&&\gamma_2^{-1} = \exp\left(\frac{2 (2 n_3 d - k_3 - 3) \pi i}{3 n_3}\right), \quad
&&\gamma_3^{-1} = \exp\left( \frac{2 (2 n_3 d + 2 k_3 + 3) \pi i)}{3 n_3}\right), 
\end{alignat*}}%
where $n_1,n_2,n_3,k_1,k_2,k_3,d$ are integers. 

\begin{rmk} These parameters were chosen so $\alpha_2/\alpha_1 = \exp(-2\pi i/n_1)$, $\beta_2/\beta_1 = \exp(-2\pi i/n_2)$ and $(\gamma_2/\gamma_1)^{-1} = \exp(-2\pi i/n_3)$. This imposition follows from condition $(\textrm{Q4})$ of the quadrangle conditions. Due to the properties of the exponential, we can impose $0\leq k_i\leq n_i-1$ and $d=0,1,2$.

Since we are interested here in cases when the three isometries are regular elliptic, meaning the eigenvalues are pairwise distinct, we exclude here the cases where $k_i = n_i-2$ and $k_i = n_i-1$.
\end{rmk}

\begin{rmk} Following Proposition \ref{prop Euler number}, the Euler number formula is given by
$$e=f-\frac{k_1+1}{n_1}-\frac{k_2+1}{n_2}-\frac{k_3+1}{n_3}$$
for all examples. The parameter $f$ is defined in Section \ref{subsection meridional curve} and is known to value $1$ for all investigated examples. The Toledo invariant computation is outlined in the proof of Proposition~\ref{toledomod2}, where we prove that $\tau\equiv\displaystyle\frac{\Arg(\alpha_1\beta_1\gamma_1^{-1})}{\pi}\mod2$. Empirically, we obtain that for all examples $3\tau = 2(e+\chi)$. This formula holds true for all fundamental domains constructed from quadrangles of bisectors as observed in \cite[Theorem 8]{bot2}. 
\end{rmk}

We have found $3464$ lists $(n_1,n_2,n_3,k_1,k_2,k_3)$, with $3 \leq n_i \leq 15$ and $ 0\leq k_i\leq n_i-3$ and $d=1$, corresponding to generic discrete faithful representations with $I_1$, $I_2$, $I_3$ regular elliptic. 
Thus, we have $3464$ representations $G(n_1,n_2,n_3) \to \PU$ that lead to disc orbibundles over spheres with three conic points. These are examples of complex hyperbolic disc orbibundles $\HH_\CC^2/G \to \HH_\RR^2/G$ as discussed in Remark \ref{remark disc orbibundle}. Additionally, they are topologically distinct because they have different relative Euler numbers $e/\chi$. Of these $3464$ examples, $853$ distinct triples $(n_1,n_2,n_3)$ occur, thus our examples take place in $853$ different character varieties. Passing to finite index, each of these disc orbibundles over an orbifold gives rise to a disc bundle over a surface with the same relative Euler number $e/\chi$ and the same relative Toledo invariant $\tau/\chi$ (for details about these invariants see \cite{bot}). See Figure \ref{euler frequency non rigid}.

By Proposition \ref{prop main result}, each of these orbibundles belongs to a $2$-dimensional family of pairwise non-isometric (nevertheless, diffeomorphic) orbibundles, because the conditions defining these orbibundles define open regions in the $\PU$-character varieties. Clearly, examples in the same family share the same discrete invariants. The corresponding relative Euler numbers vary in the region $\{-1\}\cup (-0.63,0.49)$ and all examples satisfy the relation $3\tau = 2(e+\chi)$, which is a necessary condition for the existence of a holomorphic section (see \cite[Corollary 43]{bot}). 

\begin{figure}[H]
	\centering
	\includegraphics[scale =.8]{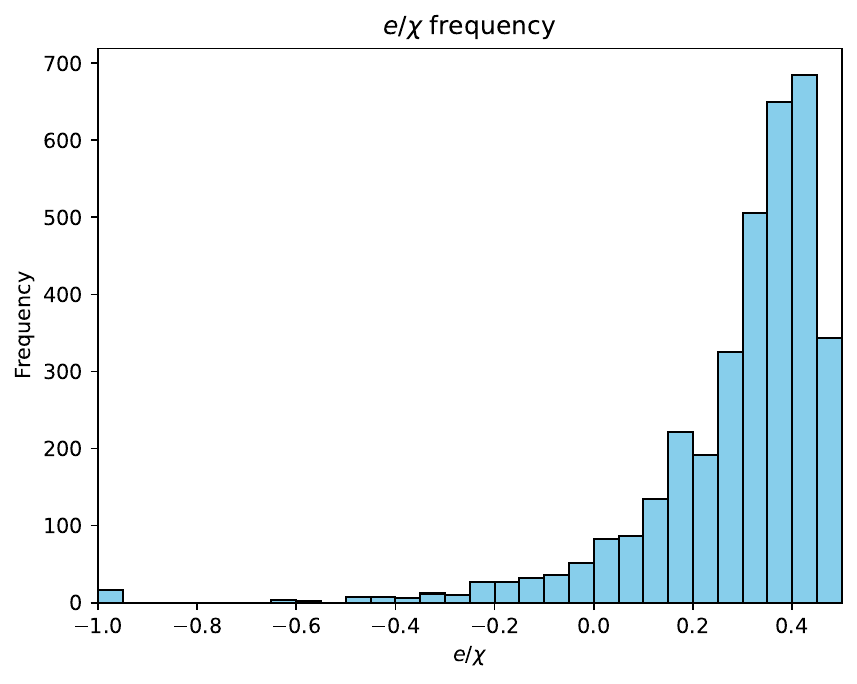}
	\caption{Histogram for the variable $e/\chi$. A total of $3464$ non-rigid disc orbibundles with different relative Euler numbers $e/\chi$ are found in $853$ character varieties.}
	\label{euler frequency non rigid}
\end{figure}

A couple of examples deserve to be highlighted: the cotangent bundle ($e/\chi = -1$) and the trivial bundle ($e/\chi = 0$). This seems to be the first instance of a complex hyperbolic structure on the cotangent bundle of a compact Riemann surface. As for the trivial bundle, an example has been constructed in \cite{agu} thus solving a long-standing conjecture \cite[Open Question 8.1]{eli}, \cite[p. 583]{gol2}, and \cite[p. 14]{sch}. Our construction is quite different from the one in \cite{agu}; while the latter produces, at the orbibundle level, a single rigid example, the former leads to several two-dimensional families of such trivial orbibundles. Finally, we also find non-rigid discrete representations corresponding to disc orbibundles whose relative Toledo invariant vanishes and which are not $\mathbb R$-Fuchsian\footnote{A representation of a orbifold group in $\PU$ with a stable real plane is called $\mathbb R$-Fuchsian.} because, for such examples, $\tr [I_1,I_2] \not \in \RR$ (it is well-known that $\mathbb R$-Fuchsian representations have vanishing Toledo invariant; in this regard, see also \cite{Gus}).
It is interesting to note that we have found $2$ examples with $e/\chi = -0.5$, thus corresponding to a square root of the cotangent bundle.

The regions of the $\PU$-character variety $\mathcal R(n_1,n_2,n_3)$ that we study are those parametrized in the coordinates $(s,t) \in \RR_{\geq 0}\times\RR_{\geq0}$, as described in Remark \ref{remark alg}. 
Once fixed the parameters $\alpha_i$'s, $\beta_i$'s, $\gamma_i^{-1}$'s we determine, when possible, $I_1,I_2,I_3$ via the procedure outlined in the Remark \ref{remark alg}: we start by fixing the isometry $I_1$ as a diagonal matrix with eigenvalues $\alpha_1,\alpha_2,\alpha_3$, then we compute the isometry $I_2$ entries as function of $(s,t)$, where its negative eigenvector is $u\coloneq (\sqrt{1+s+t},\sqrt{s},\sqrt{t})$ and the rest of the information follows by imposing that $I_2$ has eigenvalues $\beta_i$'s and $I_2 I_1$ has eigenvalues $\gamma_i$'s. By imposing that the eigenvalues for $I_2$ are $\beta_1,\beta_2,\beta_3$ we can write $I_2$ explicitly as function of its negative eigenvector $u$ associated to $\beta_1$ and a positive eigenvector $v$, associated to $\beta_2$, using the expression in Equation \eqref{matrices I1, I2}. This second eigenvector is determined by imposing that $I_2I_1$ have trace $\gamma_1+\gamma_2+\gamma_3$, which implies, by a result due to Goldman, that $I_2I_1$ has eigenvalues $\gamma_1,\gamma_2,\gamma_3$, assuming these three numbers are pairwise distinct.
For details, see {\it Regular case: $I_2$ is regular elliptic} in the Section \ref{section proof of main result}.

With that in mind, we list all examples found with $e/\chi =-1,\,-0.5,\,0$.

{\footnotesize
\begin{center}
\begin{table}[H]
    \renewcommand*{\arraystretch}{0.9}

\begin{minipage}{.51\linewidth}
  \begin{tabular}{|c|c|c|c|c|c|c|c|c|}
    \hline
    $n_1$ & $n_2$ & $n_3$ & $k_1$ & $k_2$ & $k_3$ & $s$ & $t$ & $e/\chi$ \\
    \hline
    3 & 4 & 12 & 0 & 0 & 4 & 1.0 & 0.4 & 0 \\
    4 & 3 & 12 & 0 & 0 & 4 & 0.7 & 0.5 & 0 \\
    5 & 3 & 15 & 0 & 0 & 6 & 1.2 & 1.0 & 0 \\
    5 & 3 & 15 & 1 & 0 & 3 & 1.2 & 0.4 & 0 \\
    5 & 5 & 5 & 0 & 2 & 0 & 2.4 & 0.9 & 0 \\
    6 & 4 & 12 & 0 & 0 & 6 & 3.8 & 2.7 & 0 \\
    8 & 4 & 8 & 4 & 0 & 0 & 3.5 & 0.7 & 0 \\
    9 & 3 & 9 & 0 & 0 & 4 & 3.5 & 2.6 & 0 \\
    9 & 3 & 9 & 1 & 0 & 3 & 2.9 & 1.0 & 0 \\
    9 & 3 & 9 & 3 & 0 & 1 & 2.6 & 0.4 & 0 \\
    9 & 3 & 9 & 4 & 0 & 0 & 2.4 & 0.2 & 0 \\
    12 & 3 & 4 & 4 & 0 & 0 & 2.9 & 0.3 & 0 \\
    12 & 3 & 12 & 6 & 0 & 0 & 3.4 & 0.3 & 0 \\
    12 & 4 & 3 & 4 & 0 & 0 & 3.7 & 0.6 & 0 \\
    15 & 3 & 5 & 6 & 0 & 0 & 4.4 & 0.3 & 0 \\
    3 & 3 & 9 & 0 & 0 & 1 & 0.6 & 0.4 & -0.5 \\
    9 & 3 & 3 & 1 & 0 & 0 & 2.5 & 1.2 & -0.5 \\
    \hline
  \end{tabular}
\end{minipage}%
\begin{minipage}{.51\linewidth}
  \begin{tabular}{|c|c|c|c|c|c|c|c|c|}
    \hline
     $n_1$ & $n_2$ & $n_3$ & $k_1$ & $k_2$ & $k_3$ & $s$ & $t$ & $e/\chi$ \\
    \hline
    3 & 3 & 4 & 0 & 0 & 0 & 0.3 & 0.2 & -1 \\
    3 & 3 & 5 & 0 & 0 & 0 & 0.5 & 0.3 & -1 \\
    3 & 3 & 6 & 0 & 0 & 0 & 0.6 & 0.4 & -1 \\
    3 & 3 & 7 & 0 & 0 & 0 & 0.7 & 0.5 & -1 \\
    3 & 4 & 3 & 0 & 0 & 0 & 0.4 & 0.4 & -1 \\
    4 & 3 & 3 & 0 & 0 & 0 & 0.3 & 0.4 & -1 \\
    4 & 3 & 4 & 0 & 0 & 0 & 0.7 & 0.7 & -1 \\
    4 & 3 & 5 & 0 & 0 & 0 & 1.1 & 1.1 & -1 \\
    4 & 3 & 6 & 0 & 0 & 0 & 1.4 & 1.4 & -1 \\
    5 & 3 & 3 & 0 & 0 & 0 & 0.7 & 0.9 & -1 \\
    5 & 3 & 4 & 0 & 0 & 0 & 1.5 & 1.6 & -1 \\
    5 & 3 & 5 & 0 & 0 & 0 & 2.3 & 2.3 & -1 \\
    6 & 3 & 3 & 0 & 0 & 0 & 1.4 & 1.6 & -1 \\
    6 & 3 & 4 & 0 & 0 & 0 & 2.8 & 2.8 & -1 \\
    7 & 3 & 3 & 0 & 0 & 0 & 2.2 & 2.4 & -1 \\
    8 & 3 & 3 & 0 & 0 & 0 & 3.4 & 3.6 & -1 \\
      &   &   &   &   &   &     &     &      \\
    \hline
  \end{tabular}
\end{minipage}
\caption{Non-rigid examples with $e/\chi$ equal to $0,-0.5,-1$.}
\label{table e/x =0, -1, 0.5 non-rigid}
\end{table}
\end{center}
}%
\vspace*{-1cm}

The methodology followed to find examples is the following: for each $(n_1,n_2,n_3,k_1,k_2,k_3,d)$, with $3\leq n_i\leq 15$, $0\leq k_i \leq n_i-3$ and $d=1$, we constructed the representation following the formulas in Section \ref{section proof of main result} and then tested the conditions for discreteness, described in Section \ref{section discreteness} and more in-depth in Section \ref{quadrangle revisited} for $s,t  \in  (0.1 \ZZ) \cap [0,5]$. We found $39431$ examples. Among them, we observed $853$ distinct $(n_1,n_2,n_3,k_1,k_2,k_3,d)$ with different Euler numbers $e/\chi$ (the relative Euler number does not depend on the parameters $s,t$), thus corresponding to distinct topological objects.

\begin{rmk}
    Explicit computations can be found in the Section \ref{section:Explicit example with trivial Euler number} for the example $$(n_1,n_2,n_3,k_1,k_2,k_3,s,t)=(5, 5, 5, 0, 2, 0, 2.4, 0.9),$$ from the Table \ref{table e/x =0, -1, 0.5 non-rigid} above. Such representation is discrete forming a disc orbibundle over $\SP^2(5,5,5)$. Additionally, we compute its  Euler number to be $e = 0$ and the Toledo invariant to be $\tau = \frac23 \chi$.
\end{rmk}

In Figure \ref{fig: R334}, we illustrate a prototypical connected component of $\mathcal R(3,3,4)$ in the coordinates $(s,t)$. The precise parameters for this component are $n_1=n_2=3,n_3=4,k_1=k_2=k_3=0$, as shown in Table~\ref{table e/x =0, -1, 0.5 non-rigid}.

\begin{figure}[H]\label{figure1}
	\centering
	\includegraphics[scale = .45]{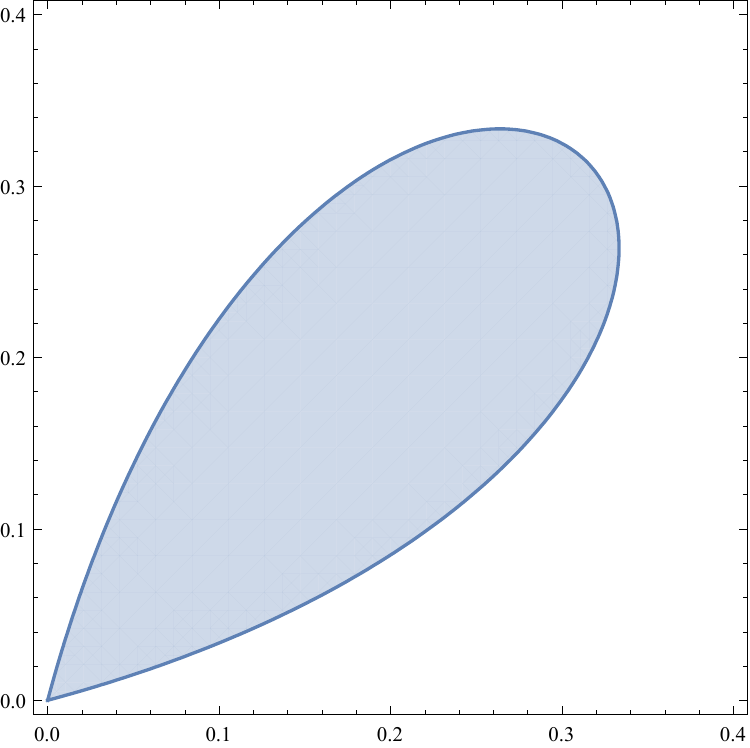}
	\caption{$\mathcal R(3,3,4)$ with regular $I_1,I_2,I_3$}
    \label{fig: R334}
\end{figure}

\noindent
The shaded region in Figure \ref{fig: R334 discbundles} corresponds to the representations modulo $\PU$ where the inequalities outlined in Section \ref{quadrangle revisited} hold. Each point in the shaded region defines a complex hyperbolic disc orbibundle, and two distinct points correspond to a pair of diffeomorphic but non-isometric disc orbibundles.

\begin{figure}[H]
	\centering
	\includegraphics[scale = .45]{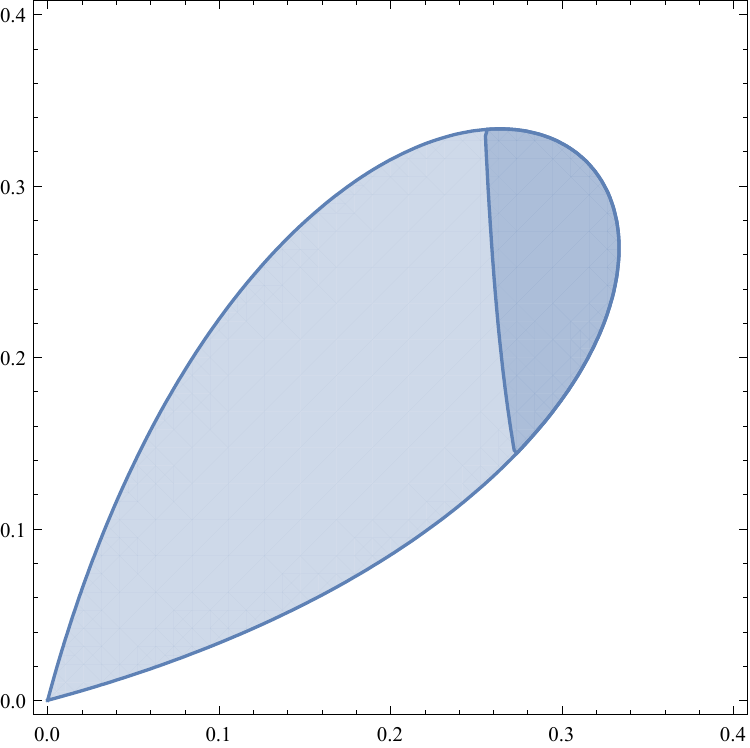}
	\caption{Some disc orbibundles over $\SP^2(3,3,4)$}
     \label{fig: R334 discbundles}
\end{figure}

In principle, it could be that all generic representations $G(n_1,n_2,n_3) \to \PU$ with regular isometries $I_1,I_2,I_3$ were discrete; for a point in the shaded region, discreteness is guaranteed because a particular fundamental domain is shown to exist (see Section \ref{section discreteness}), but nothing is preventing the points outside such region to also correspond to discrete representations. However, this is not the case. Indeed, consider the function
$G(s,t) = f(\tr[I_1,I_2])$, where
\[f(z) =|z|^4 - 8\real(z^3) + 18 |z|^2 -27\]
is Goldman's discriminant (see \cite[p.~204, Theorem 6.2.4]{goldmanbook}). The region in $\mathcal R(3,3,4)$ described by $G(s,t)<0$ is the shaded area in the Figure \ref{334goldman}. Note that $G(0.01,0.01) \simeq -12.282$ is negative and $G(0.05,0.05) = 26.392$ is positive.

\begin{figure}[H]
	\centering
	\includegraphics[scale = .45]{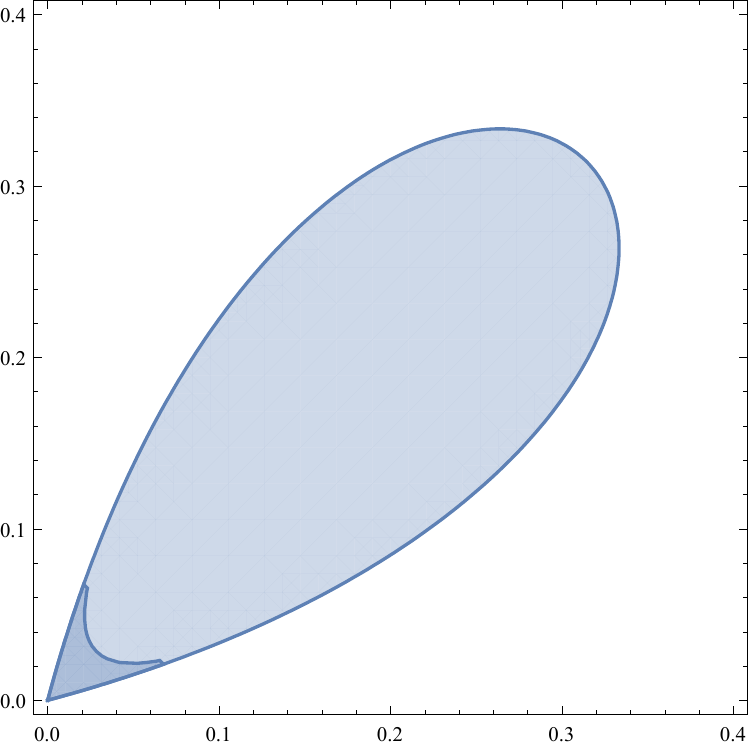}
	\caption{$\mathcal R(3,3,4)$ and $\G(s,t) < 0$}
	\label{334goldman}
\end{figure}

\noindent
Not all examples with $G(s,t)<0$ can be discrete because, by \cite[Theorem 6.2.4]{goldmanbook}, $G(s,t)<0$ means that $[I_1,I_2]$ is regular elliptic. Since in $\mathcal R(3,3,4)$ there are also points where $G(s,t)>0$, we obtain uncountable many distinct negative values for $G(s,t)$. If all representations in $\mathcal R(3,3,4)$ were discrete, the elliptic isometries in the group $\langle I_1,I_2,I_3 \rangle$ would be of finite order, since discrete groups of isometries have finite stabilizers, leading to a countable amount of possibles values for $G(s,t)<0$. By continuity, there is a non-discrete representation with $s=t$ and $t \in (0.01,0.05)$.

\medskip

Now we discuss the representations where $I_1,I_3$ are regular elliptic and $I_2$ is special elliptic with $3\leq n_1,n_3 \leq 30$, $0\leq k_1\leq n_1-3$, $0\leq k_3\leq n_3-3$, $2\leq n_2 \leq 30$, $0\leq k_2\leq n_2-1$ and $d=0,1,2$. When $I_2$ is a rotation about a point ($\beta_2=\beta_3$), we found $34240$ examples in $13345$ different character varieties, with $e/\chi\in[-1,0.5]$. See Figure \ref{euler frequency rigid}. The values $e/\chi =-1,\,-0.5,\,0,\,0.5$ occur here. On the other hand, when $I_2$ is a rotation around a complex geodesic ($\beta_1=\beta_3$), we found $67769$ examples in $17030$ different character varieties, with $e/\chi\in(0,0.5]$, including the right extreme (thus neither $e/\chi=-1$ nor $e/\chi=0$ were observed here, the case $e/\chi = 0.5$ occurs $50$ times, corresponding to square roots of tangent bundles). As in all the examples we found, the identity $3\tau = 2(e+ \chi)$ is satisfied. Note that $e/\chi =0.5$ is the maximal relative Euler number allowed by this formula (because $|\tau/\chi| \leq 1$ by Toledo Rigidity) and that this particular relative Euler number only happens for $\tau/\chi = 1$ (thus, the corresponding representation is $\CC$-Fuchsian).

\begin{figure}[H]
	\centering
	\begin{minipage}{.5\textwidth}
		\centering
		\includegraphics[scale =0.9]{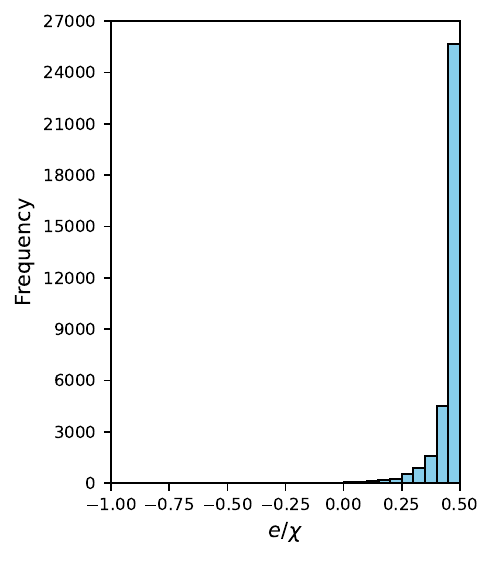}
        \vspace*{-.4cm}
		\caption*{\hspace{1cm}(a)}
	\end{minipage}%
	\begin{minipage}{.5\textwidth}
		\centering

		\includegraphics[scale =0.9]{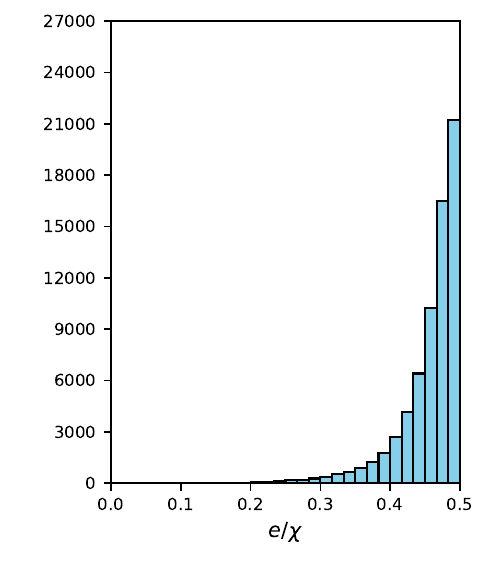}
        \vspace*{-.4cm}
		\caption*{\hspace{1cm}(b)}
	\end{minipage}
	\caption{\textbf{(a)} Histogram for the variable $e/\chi$ when $I_2$ is a rotation about a point. A total of $34240$ rigid disc orbibundles with different relative Euler numbers $e/\chi$ are found in $13345$ character varieties. \textbf{(b)} Histogram for the variable $e/\chi$ when $I_2$ is a rotation about a complex geodesic. A total of $67769$ rigid disc orbibundles with different relative Euler numbers $e/\chi$ are found in $17030$ character varieties.}
	\label{euler frequency rigid}
\end{figure}

We list all the rigid examples with $e/\chi=-1,\,-0.5,\,0,\,0.5$ on the Tables \ref{table e/x =0, -1, 0.5 non-rigid about a point} and \ref{table e/x =0, -1, 0.5 non-rigid about a complex geodesic}. On these tables, there is no parameter $s,t$ displayed because they are determined by the eigenvalues when $I_2$ is a rotation about a point ($\beta_2=\beta_3$) and $I_2$ does not depend them when $I_2$ is a rotation about a geodesic ($\beta_1=\beta_3$), as can be seen from the formula for the matrix $I_2$ at the Equation \eqref{matrices I1, I2}.

{\footnotesize
\begin{center}
\begin{table}[H]
 \setlength{\tabcolsep}{2pt}
    \renewcommand*{\arraystretch}{0.9}

\begin{minipage}{.246\linewidth}
  \resizebox{.95\textwidth}{!}{
  \begin{tabular}{|c|c|c|c|c|c|c|}
    \hline
    $n_1$ & $n_2$ & $n_3$ & $k_1$ & $k_2$ & $k_3$ & $e/\chi$ \\
    
    \hline
16 & 16 & 8 & 0 & 14 & 2 & 0.5 \\
18 & 6 & 5 & 0 & 4 & 1 & 0.5 \\
18 & 12 & 12 & 0 & 10 & 4 & 0.5 \\
18 & 15 & 30 & 0 & 13 & 12 & 0.5 \\
19 & 19 & 19 & 0 & 17 & 7 & 0.5 \\
19 & 19 & 19 & 16 & 17 & 10 & 0.5 \\
20 & 15 & 12 & 0 & 13 & 4 & 0.5 \\
20 & 20 & 30 & 0 & 18 & 12 & 0.5 \\
23 & 23 & 23 & 0 & 21 & 9 & 0.5 \\
24 & 4 & 24 & 0 & 2 & 12 & 0.5 \\
24 & 8 & 3 & 0 & 6 & 0 & 0.5 \\
24 & 8 & 5 & 0 & 6 & 1 & 0.5 \\
24 & 8 & 7 & 0 & 6 & 2 & 0.5 \\
24 & 8 & 9 & 0 & 6 & 3 & 0.5 \\
24 & 8 & 11 & 0 & 6 & 4 & 0.5 \\
24 & 8 & 13 & 0 & 6 & 5 & 0.5 \\
24 & 8 & 15 & 0 & 6 & 6 & 0.5 \\
24 & 8 & 17 & 0 & 6 & 7 & 0.5 \\
24 & 8 & 19 & 0 & 6 & 8 & 0.5 \\
24 & 8 & 21 & 0 & 6 & 9 & 0.5 \\
24 & 8 & 23 & 0 & 6 & 10 & 0.5 \\
24 & 8 & 25 & 0 & 6 & 11 & 0.5 \\
24 & 8 & 27 & 0 & 6 & 12 & 0.5 \\
24 & 8 & 29 & 0 & 6 & 13 & 0.5 \\
24 & 12 & 24 & 0 & 10 & 10 & 0.5 \\
24 & 24 & 12 & 0 & 22 & 4 & 0.5 \\
25 & 5 & 5 & 1 & 3 & 1 & 0.5 \\
    \hline
  \end{tabular}}
  
\end{minipage}
\begin{minipage}{.246\linewidth}
  \resizebox{.95\textwidth}{!}{
  \begin{tabular}{|c|c|c|c|c|c|c|}
    \hline
     $n_1$ & $n_2$ & $n_3$ & $k_1$ & $k_2$ & $k_3$ & $e/\chi$ \\
    \hline
25 & 25 & 25 & 20 & 23 & 15 & 0.5 \\
25 & 5 & 25 & 22 & 3 & 15 & 0.5 \\
25 & 25 & 25 & 22 & 23 & 13 & 0.5 \\
27 & 3 & 27 & 0 & 1 & 15 & 0.5 \\
27 & 9 & 3 & 0 & 7 & 0 & 0.5 \\
27 & 9 & 5 & 0 & 7 & 1 & 0.5 \\
27 & 9 & 7 & 0 & 7 & 2 & 0.5 \\
27 & 9 & 9 & 0 & 7 & 3 & 0.5 \\
27 & 9 & 11 & 0 & 7 & 4 & 0.5 \\
27 & 9 & 13 & 0 & 7 & 5 & 0.5 \\
27 & 9 & 15 & 0 & 7 & 6 & 0.5 \\
27 & 9 & 17 & 0 & 7 & 7 & 0.5 \\
27 & 9 & 19 & 0 & 7 & 8 & 0.5 \\
27 & 9 & 21 & 0 & 7 & 9 & 0.5 \\
27 & 9 & 23 & 0 & 7 & 10 & 0.5 \\
27 & 9 & 25 & 0 & 7 & 11 & 0.5 \\
27 & 9 & 27 & 0 & 7 & 12 & 0.5 \\
27 & 9 & 29 & 0 & 7 & 13 & 0.5 \\
27 & 3 & 27 & 24 & 1 & 18 & 0.5 \\
27 & 6 & 18 & 24 & 4 & 10 & 0.5 \\
27 & 9 & 9 & 24 & 7 & 4 & 0.5 \\
27 & 9 & 27 & 24 & 7 & 15 & 0.5 \\
27 & 18 & 30 & 24 & 16 & 16 & 0.5 \\
30 & 4 & 20 & 0 & 2 & 10 & 0.5 \\
30 & 10 & 3 & 0 & 8 & 0 & 0.5 \\
30 & 10 & 5 & 0 & 8 & 1 & 0.5 \\
30 & 10 & 7 & 0 & 8 & 2 & 0.5 \\
    \hline
  \end{tabular}}
\end{minipage}
\begin{minipage}{.246\linewidth}
  \resizebox{.95\textwidth}{!}{
  \begin{tabular}{|c|c|c|c|c|c|c|}
    \hline
     $n_1$ & $n_2$ & $n_3$ & $k_1$ & $k_2$ & $k_3$ & $e/\chi$ \\
    \hline
30 & 10 & 9 & 0 & 8 & 3 & 0.5 \\
30 & 10 & 13 & 0 & 8 & 5 & 0.5 \\
30 & 10 & 15 & 0 & 8 & 6 & 0.5 \\
30 & 10 & 17 & 0 & 8 & 7 & 0.5 \\
30 & 10 & 19 & 0 & 8 & 8 & 0.5 \\
30 & 10 & 21 & 0 & 8 & 9 & 0.5 \\
30 & 10 & 23 & 0 & 8 & 10 & 0.5 \\
30 & 10 & 27 & 0 & 8 & 12 & 0.5 \\
30 & 10 & 29 & 0 & 8 & 13 & 0.5 \\
5 & 2 & 10 & 0 & 0 & 2 & 0 \\
5 & 2 & 10 & 1 & 0 & 0 & 0 \\
7 & 2 & 14 & 1 & 0 & 2 & 0 \\
7 & 2 & 14 & 2 & 0 & 0 & 0 \\
8 & 2 & 8 & 0 & 0 & 2 & 0 \\
8 & 2 & 8 & 2 & 0 & 0 & 0 \\
10 & 2 & 5 & 0 & 0 & 1 & 0 \\
10 & 2 & 5 & 2 & 0 & 0 & 0 \\
12 & 2 & 12 & 4 & 0 & 0 & 0 \\
14 & 2 & 7 & 2 & 0 & 1 & 0 \\
14 & 2 & 7 & 4 & 0 & 0 & 0 \\
18 & 2 & 9 & 4 & 0 & 1 & 0 \\
18 & 2 & 9 & 6 & 0 & 0 & 0 \\
22 & 2 & 11 & 8 & 0 & 0 & 0 \\
9 & 2 & 6 & 1 & 0 & 0 & -0.5 \\
3 & 2 & 7 & 0 & 0 & 0 & -1 \\
3 & 2 & 8 & 0 & 0 & 0 & -1 \\
3 & 2 & 9 & 0 & 0 & 0 & -1 \\
    \hline
  \end{tabular}}
\end{minipage}
\begin{minipage}{.246\linewidth}
  \resizebox{.95\textwidth}{!}{
  \begin{tabular}{|c|c|c|c|c|c|c|}
    \hline
     $n_1$ & $n_2$ & $n_3$ & $k_1$ & $k_2$ & $k_3$ & $e/\chi$ \\
    \hline
3 & 2 & 10 & 0 & 0 & 0 & -1 \\
3 & 2 & 11 & 0 & 0 & 0 & -1 \\
3 & 2 & 12 & 0 & 0 & 0 & -1 \\
3 & 2 & 13 & 0 & 0 & 0 & -1 \\
3 & 2 & 14 & 0 & 0 & 0 & -1 \\
3 & 2 & 15 & 0 & 0 & 0 & -1 \\
3 & 2 & 16 & 0 & 0 & 0 & -1 \\
4 & 2 & 5 & 0 & 0 & 0 & -1 \\
4 & 2 & 6 & 0 & 0 & 0 & -1 \\
4 & 2 & 7 & 0 & 0 & 0 & -1 \\
4 & 2 & 8 & 0 & 0 & 0 & -1 \\
4 & 2 & 9 & 0 & 0 & 0 & -1 \\
5 & 2 & 4 & 0 & 0 & 0 & -1 \\
5 & 2 & 5 & 0 & 0 & 0 & -1 \\
5 & 2 & 6 & 0 & 0 & 0 & -1 \\
5 & 2 & 7 & 0 & 0 & 0 & -1 \\
6 & 2 & 4 & 0 & 0 & 0 & -1 \\
6 & 2 & 5 & 0 & 0 & 0 & -1 \\
7 & 2 & 3 & 0 & 0 & 0 & -1 \\
7 & 2 & 4 & 0 & 0 & 0 & -1 \\
7 & 2 & 5 & 0 & 0 & 0 & -1 \\
8 & 2 & 3 & 0 & 0 & 0 & -1 \\
8 & 2 & 4 & 0 & 0 & 0 & -1 \\
9 & 2 & 3 & 0 & 0 & 0 & -1 \\
  &   &   &   &   &   &  \\
  &   &   &   &   &   &  \\
  &   &   &   &   &   &  \\
    \hline
  \end{tabular}}
\end{minipage}
\caption{Rigid examples with $e/\chi$ equal to $0.5, 0,-0.5,-1$ when $I_2$ is a rotation about a point. There are 27 instances with $e/\chi = -1$ , $1$ with $e/\chi =- 0.5$, $14$ with $e/\chi=0$, and $63$ with $e/\chi = 0.5$.
}
\label{table e/x =0, -1, 0.5 non-rigid about a point}
\end{table}
\end{center}
}%

{\footnotesize
\begin{center}
\begin{table}[H]
 \setlength{\tabcolsep}{2pt}
     \renewcommand*{\arraystretch}{0.9}

\begin{minipage}{.246\linewidth}
  \resizebox{.95\textwidth}{!}{
  \begin{tabular}{|c|c|c|c|c|c|c|}
    \hline
    $n_1$ & $n_2$ & $n_3$ & $k_1$ & $k_2$ & $k_3$ & $e/\chi$ \\
    \hline
9 & 27 & 27 & 0 & 26 & 7 & 0.5 \\
9 & 27 & 27 & 6 & 26 & 16 & 0.5 \\
9 & 30 & 30 & 6 & 29 & 18 & 0.5 \\
11 & 22 & 22 & 8 & 21 & 12 & 0.5 \\
12 & 28 & 7 & 0 & 27 & 1 & 0.5 \\
14 & 14 & 7 & 0 & 13 & 1 & 0.5 \\
15 & 30 & 30 & 0 & 29 & 10 & 0.5 \\
15 & 25 & 25 & 12 & 24 & 13 & 0.5 \\
15 & 30 & 30 & 12 & 29 & 16 & 0.5 \\
18 & 12 & 4 & 0 & 11 & 0 & 0.5 \\
18 & 18 & 27 & 0 & 17 & 9 & 0.5 \\
20 & 10 & 4 & 0 & 9 & 0 & 0.5 \\
20 & 12 & 30 & 0 & 11 & 10 & 0.5 \\
    \hline
  \end{tabular}}
\end{minipage}
\begin{minipage}{.246\linewidth}
  \resizebox{.95\textwidth}{!}{
  \begin{tabular}{|c|c|c|c|c|c|c|}
    \hline
     $n_1$ & $n_2$ & $n_3$ & $k_1$ & $k_2$ & $k_3$ & $e/\chi$ \\
    \hline
21 & 30 & 30 & 16 & 29 & 18 & 0.5 \\
21 & 21 & 21 & 18 & 20 & 10 & 0.5 \\
21 & 28 & 28 & 18 & 27 & 14 & 0.5 \\
22 & 22 & 11 & 0 & 21 & 3 & 0.5 \\
24 & 8 & 4 & 0 & 7 & 0 & 0.5 \\
24 & 8 & 20 & 0 & 7 & 6 & 0.5 \\
24 & 24 & 18 & 0 & 23 & 6 & 0.5 \\
25 & 25 & 25 & 1 & 24 & 8 & 0.5 \\
25 & 30 & 30 & 1 & 29 & 10 & 0.5 \\
25 & 25 & 25 & 18 & 24 & 16 & 0.5 \\
25 & 25 & 25 & 21 & 24 & 13 & 0.5 \\
25 & 30 & 30 & 21 & 29 & 16 & 0.5 \\
25 & 25 & 25 & 22 & 24 & 12 & 0.5 \\
    \hline
  \end{tabular}}
\end{minipage}
\begin{minipage}{.246\linewidth}
  \resizebox{.95\textwidth}{!}{
  \begin{tabular}{|c|c|c|c|c|c|c|}
    \hline
     $n_1$ & $n_2$ & $n_3$ & $k_1$ & $k_2$ & $k_3$ & $e/\chi$ \\
    \hline
26 & 26 & 13 & 22 & 25 & 6 & 0.5 \\
27 & 18 & 18 & 0 & 17 & 6 & 0.5 \\
27 & 27 & 27 & 0 & 26 & 10 & 0.5 \\
27 & 27 & 27 & 1 & 26 & 9 & 0.5 \\
27 & 27 & 9 & 4 & 26 & 1 & 0.5 \\
27 & 27 & 27 & 18 & 26 & 19 & 0.5 \\
27 & 27 & 27 & 19 & 26 & 18 & 0.5 \\
27 & 27 & 27 & 21 & 26 & 16 & 0.5 \\
27 & 27 & 9 & 22 & 26 & 4 & 0.5 \\
27 & 27 & 27 & 22 & 26 & 15 & 0.5 \\
27 & 9 & 27 & 24 & 8 & 12 & 0.5 \\
27 & 27 & 27 & 24 & 26 & 13 & 0.5 \\
28 & 2 & 28 & 0 & 1 & 4 & 0.5 \\
    \hline
  \end{tabular}}
\end{minipage}
\begin{minipage}{.246\linewidth}
  \resizebox{.95\textwidth}{!}{
  \begin{tabular}{|c|c|c|c|c|c|c|}
    \hline
     $n_1$ & $n_2$ & $n_3$ & $k_1$ & $k_2$ & $k_3$ & $e/\chi$ \\
    \hline
28 & 12 & 21 & 0 & 11 & 7 & 0.5 \\
28 & 14 & 28 & 0 & 13 & 10 & 0.5 \\
29 & 29 & 29 & 1 & 28 & 10 & 0.5 \\
29 & 29 & 29 & 21 & 28 & 19 & 0.5 \\
29 & 29 & 29 & 22 & 28 & 18 & 0.5 \\
29 & 29 & 29 & 23 & 28 & 17 & 0.5 \\
29 & 29 & 29 & 24 & 28 & 16 & 0.5 \\
29 & 29 & 29 & 26 & 28 & 14 & 0.5 \\
30 & 15 & 18 & 0 & 14 & 6 & 0.5 \\
30 & 30 & 5 & 22 & 29 & 2 & 0.5 \\
30 & 30 & 25 & 22 & 29 & 16 & 0.5 \\
  &   &   &   &   &   &  \\
  &   &   &   &   &   &  \\
    \hline
  \end{tabular}}
\end{minipage}
\caption{We have found $50$ rigid examples with $e/\chi=0.5$ when $I_2$ is a rotation about a complex geodesic. No example with $e/\chi < 0.16$ was found where $I_2$ is a rotation about a complex geodesic.
}
\label{table e/x =0, -1, 0.5 non-rigid about a complex geodesic}
\end{table}
\end{center}
}%

\begin{rmk}When we drop the transversalities conditions $(\mathrm{Q2})$ and $(\mathrm{Q3})$ stated in the Section \ref{subsection quadrangle conditions} for the quadrangles, the formula defining $e$ still makes sense, since it only depends on information coming from the eigenvalues of $I_1,I_2,I_3$ and on the integer $f$ defined in Section \ref{subsection holonomy of the quadrangle}. Curiously, the formula $3\tau = 2(e+\chi)$ still holds in the majority of the cases where the transversalities conditions were dropped. In the few cases where it fails, $3/2\tau - \chi$ differs from $e$ by an integer. This suggests that the formula for $f$ needs a correction with respect to the topology of the ``quadrangle'' corresponding to such degenerate cases. We believe that this corrected Euler number $e_{\text{cor}}\coloneq 3/2\tau - \chi$ have a geometrical meaning that is related to the object obtained by gluing the sides of the ``quadrangle'' respectively to the relations defined by $I_1,I_2,I_3$. Thus, in some sense, for all points in the character variety, it seems that there exists a geometric object that behaves as a bundle over $\SP^2(n_1,n_2,n_3)$ with Euler number $e_{\text{cor}}$.
\end{rmk}

\newpage
\section{Orbifold bundles and Euler number} \label{section: Orbifold bundles and Euler number}

\subsection{Deformation lemma}
Given a complex geodesic $C$ with a chosen $c\in C$, we identify $C$ with the unit open disc in $\CC$ as follows. Let $p$ be the point orthogonal to $c$ in the complex projective line extending $C$. Take representatives such that $-\langle c,c\rangle=\langle p,p\rangle=1$. Then every point in $C$ has the form $c+\gamma p$, $|\gamma|\leq1$. For obvious reasons, we call $c$ the center of $C$.

Consider the action of $\SP^1$ on the circle $\partial_\infty C$ by rotations centered at $c$. More precisely, given a unit complex number $\theta\in\SP^1$, we define
\[\gamma[c+\theta p] = [c+\gamma\theta p].\]
In particular, we have an $\SP^1$-action on the vertices $C_i$ of the quadrangle $\mathcal Q$, where each $C_i$ has center $c_i$.

\begin{lemma}\label{deformationlemma} Consider an orientation on $V$ and let $q_i$ be a vector such that $c_i,p_i,q_i$ is a positively oriented orthonormal basis of $V$. There is a family of curves $c_i(\delta),p_i(\delta),q_i(\delta)$, with $\delta\in [0,1]$, such that
	\begin{itemize}
	\item $c_i(0) = c_i$, $p_i(0)= p_i$ and $q_i(0) = q_i$;
	\item for each $\delta$ the vectors $c_i(\delta),p_i(\delta),q_i(\delta)$ form a positively oriented orthonormal basis of $V$;
	\item for all $i,j$, with $i \neq j$, we have $\ta(p_i(\delta),p_j(\delta))>1$;
	\item If $C_i(\delta)\coloneq \PP(p_i(\delta)^\perp)$ and we consider the triangles of bisectors $\triangle_1(\delta)$, with vertices $C_1(\delta)$, $C_2(\delta)$ and $C_4(\delta)$, and $\triangle_2(\delta)$, with vertices $C_2(\delta)$, $C_3(\delta)$ and $C_4(\delta)$, then for each $\delta$ the triangles $\triangle_1(\delta)$ and $\triangle_2(\delta)$ are transversal and counterclockwise oriented.
	\item $q_1(1) = q_2(1) = q_3(1) = q_4(1)$.
	\end{itemize}

The last item means that if $\mathcal Q(\delta)$ is the quadrangle with vertices $C_i(\delta)$, then the bisectors forming the boundary of $\mathcal Q(1)$ all have the same focus (recall that the focus of a bisector is the polar point of its complex spine, see Section \ref{section preliminaries}).
\end{lemma}

\textbf{Proof:} Consider $s,t,t', \epsilon_0, \epsilon_0'$ as consider in the inequalities \eqref{Q2}. It is know from \cite[Lemma A.31]{discbundles} that the parameters $s,t, \epsilon_0$, with $t,s>1$ and $0<\epsilon_0<1$, determine up to isometry the transversal counterclockwise oriented triangle of bisectors $\triangle_1\coloneq \triangle[C_1,C_2,C_4]$.

Suppose $t>s$. We will show that choosing a convenient $\epsilon_0>0$ we can reduce $t$ until $t=s$. The inequalities which determine that $\triangle_1$ is transversal are the
\[\epsilon_0^2 s^2 + 2t^2 < 1+2\epsilon_0 st^2 \leq 2t^2 + s^2.\]

Consider the quadratic polynomial $f(x)=x^2 s^2 - 2x st^2 + 2t^2 - 1$. The roots of $f(x)=0$ are
\[x = \frac{t^2 \pm (t^2-1)}{s}\]
and, therefore, $f(x)<0$ when $1/s<x<(2t^2 -1)/s$. Notice $x_0 =(2s^2 -1)/s$ is between $1/s$ and $(2t^2 -1)/s$. We can reduce $\epsilon_0$ until $1/s<\epsilon_0 < x_0$. Now, since the inequality $\epsilon_0^2 s^2 + 2t^2 < 1+2\epsilon_0 st^2$ is equivalent
\[t^2 >\frac{s\epsilon_0+1}{2},\]
and, by our choice of $\epsilon_0$,  \[s^2 >\frac{s\epsilon_0+1}{2},\]
we can reduce $t$ until $t=s$.

Note that the inequality $1+2\epsilon_0 st^2 \leq s^2 + 2t^2$ is kept during the above procedure.

Applying the same reasoning to $t'$ we may suppose $1<t,t'\leq s$.

Now, we will show that we can deform $t$ and $\epsilon_0$ until $t=s$ and $1+2\epsilon_0 s^3 =3s^2$ always keeping $1<t\leq s$,  $0<\epsilon_0<1$ and $\epsilon_0^2 t^2+s^2+t^2 < 1+2\epsilon_0 st^2$.

Indeed, if $t<s$, then increase $t$ until one of the two following possibilities happens:
\[t=s \quad \textrm{or} \quad 1+2\epsilon_0st^2 = 2t^2+s^2.\]

If $t=s$, then we have $1+2\epsilon_0s^3 \leq 3s^2$, which is equivalent to
\[\epsilon_0 \leq \frac{3s^2-1}{2s^3}.\]

Now, the function $g(x): = \frac{3x^2-1}{2x^3}$ is strictly decreasing for $x>1$ and, therefore, $g(x) < g(1)=1$ for $x>1$. Therefore, we can increase $\epsilon_0$ until
\[\epsilon_0 = \frac{3s^2-1}{2s^3},\]
or equivalently $1+2\epsilon_0s^3 = 3s^2$.

If $1<t\leq s$ and $1+2\epsilon_0st^2 = 2t^2+s^2$, with $0<\epsilon_0<1$, then we have the inequality
\[ 2t^2+s^2<2st^2+1,\]
or equivalently
\[t^2 > \frac{s^2-1}{2(s-1)} = \frac{s+1}{2}.\]
Therefore, we can increase $t$ until $t=s$ and have $\epsilon_0\in (0,1)$ satisfying $1+2\epsilon_0s^3 = 3s^2$.

So, we can deform $t$ and $\epsilon_0$, always keeping $t>1$ and $0<\epsilon_0<1$ during the process, and in the end we obtain $t=s$ and $1+2\epsilon_0s^3 = 3s^2$.

By the same reasoning we can deform $t'$ and $\epsilon_0'$ such that we always have $t'>1$ and $0<\epsilon_0'<1$ during the deformation and in the end we obtain $t'=s$ and $1+2\epsilon_0's^3 = 3s^2$.

So we reduced the problem to the case where $t = t'= s$ and $\epsilon_0=\epsilon_0'$. Geometrically we deformed the quadrangle $\mathcal Q$ inside $\HH_\CC^2$ always keeping the vertices $C_2$ and $C_4$ fixed and moving $C_1$ and $C_3$ around such that the two triangles of bisectors are kept transversal and counterclockwise oriented. In the case we are now all sides of the quadrangle have the same length.

Now, the quadrangle $\mathcal Q$ depends only on the parameters $s$ and $\epsilon_0$, which means it depends only on the triangle $\triangle_1 $. Let $q$ be the focus of the bisector $\B[C_2,C_4]$. Using that the space of transversal and counterclockwise oriented triangles of bisectors is path-connected \cite[Lemma 2.28]{discbundles} we can deform $\triangle_1$ until $q_1=q_2=q_3 =q$. The same deformation will be done to the triangle
$\triangle_2 =\triangle[C_2,C_3,C_4]$ simultaneously using the same parameters of $\triangle_1$. Therefore, in the end of the deformation we have the desired quadrangle.
\qed
\vspace{0.5cm}

Now we apply lemma \ref{deformationlemma} to the quadrangle $\mathcal Q$. Deform the vertices $C_1,C_2,C_3,C_4$ until the focuses of the four bisectors coincide in one point $q\in\mathrm E(V)$. Let $C_1',C_2',C_3',C_4'$ stand for the vertices at the end of the deformation. We can assume that the deformation is such that the centers $c_1,c_2,c_3,c_4$ belong to $\PP(q^\perp)\cap\HH_\CC^2$ at the end and are the vertices $c_1',c_2',c_3',c_4'$ of a convex quadrilateral $P$. Also, we can suppose that the angles at $c_1',c_2',c_3',c_4'$ are respectively $2\pi/n_1, \pi/n_2, 2\pi/n_3, \pi/n_2$, that is, this quadrilateral constitutes a fundamental polygon for the turnover group $\langle R_1,R_2,R_3\rangle$ action on the hyperbolic plane $\PP(q^\perp)\cap\HH_\CC^2$; here, $R_1,R_2,R_3$ are rotations in $\PP(q^\perp)\cap\HH_\CC^2$ with respective centers $c_1',c_2',c_3'$ and angles $-2\pi/n_1,-2\pi/n_2,-2\pi/n_3$ and satisfying the relation $R_3R_2R_1=1$.

We have a polyhedron $Q'$ bounded in $\HH_\CC^2$ by the quadrangle \[\mathcal Q '\coloneq \B[C_1',C_2'] \cup \B[C_2',C_3'] \cup \B[C_3',C_4'] \cup \B[C_4',C_1'].\]
The deformation gives rise to a diffeomorphism \[F:Q'\to Q\] such that the restriction $F|_{\mathcal Q'}:\mathcal Q'\to\mathcal Q$ maps slices to slices isometrically. Furthermore, we can assume that the geodesic curves $G[c_i',c_{i+1}']$ are mapped by this diffeomorphism to curves $\mu_i$ with end points $c_i$ and $c_{i+1}$ such that $I_1\mu_4=\mu_1$ and $I_3\mu_2=\mu_3$. The curve $\mu_1 \cup \mu_2 \cup \mu_3 \cup \mu_4$ intersects each slice of $\mathcal Q$ in one point, which we will take as a center. Given these centers, we introduce an $\SP^1$-action on each slice of $\mathcal Q$ such that the map $F$ restricted to $\mathcal Q'$ is $\SP^1$-equivariant.

Note that there are two ways of mapping $B[C_1',C_2']$ to $B[C_4,C_1]$. The first one is by the map
$$ [x+\gamma q] \mapsto I_1^{-1}F(x+\gamma q)
$$
and the second one is
$$ [x+\gamma q] \mapsto F(R_1^{-1}x+(\alpha_3/\alpha_1)\gamma q),
$$
and since the diffeomorphism $F$ maps slices to slices isometrically we have this two maps coincide in $C_1'$. Nevertheless, we want these two maps to be equal near $C'_1$ in order to calculate the Euler number of the disc orbibundle over $\SP^2(n_1,n_2,n_3)$ to be constructed, and the lemma bellow tell us that this is possible. The proof of this lemma is based on the idea of ``twisting the tube". More visually, we have the diffeomorphism $G:B[C_1',C_2'] \to B[C_1',C_2']$ given by $[x+\gamma q] \mapsto F^{-1}I_1F(R^{-1}x+(\alpha_3/\alpha_1)\gamma q)$ and it gives us the behavior described in the Figure \ref{twistedtube}.

\begin{figure}[H]
	\centering
	\includegraphics[scale = .7]{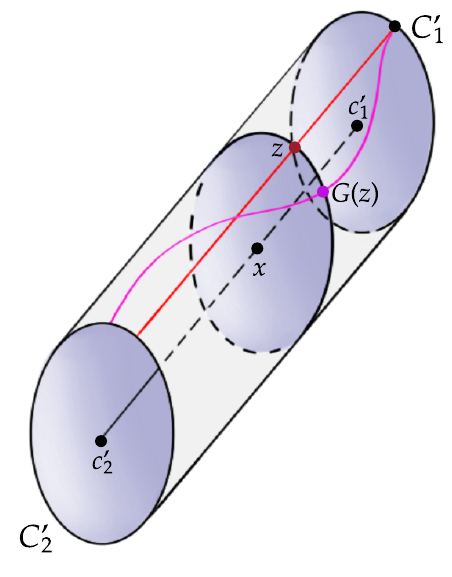}
	\captionof{figure}{Consider the red curve intersecting each slice of the bisector $B[C'_1,C'_2]$ once. If $z$ is in this red curve, then it can be writen as $z=x+\gamma q$, where $x$ is a representative of the center of the disc containing $z$ and satisfying $\langle x,x \rangle =-1$, and $G(z)$ will be a rotation depending on the center $x$. Therefore, for $x$ near $c_1'$ we have $G(x+\gamma q) = \big[x+\exp(i\theta(x))\gamma q\big]$, with $\theta(c_1')=0$. }
	\label{twistedtube}
\end{figure}

We want the red and the pink curve to coincide near $C'_1$.

\begin{lemma} \label{lema difeomorfismo} We can modify the diffeomorphism $F:\mathcal{Q'} \to \mathcal Q$ in such a way it still maps slices to slices isometrically and
\[F(x +\gamma q) = I_1F(R_1^{-1}x + (\alpha_3/\alpha_1)^{-1}\gamma  q) \quad \textrm{for} \quad  x \in N_1 \cap \G[c_1',c_2'],\quad \langle x,x \rangle =-1,\]
for some neighborhood $N_1$ of $c_1'$ in $\PP(q^\perp)$.
\end{lemma}

\textbf{Proof:} On the vertices we have
\[\gamma F(c_i'+\theta q) = F(c_i'+\gamma\theta q),\]
and therefore the desired identity holds on $c_1'$.

By continuity, for a small neighborhood $V$ of $c_1'$ on the geodesic $\G[c_1',c_2']$ we have a smooth function $\theta:V \to \RR$ satisfying \[F^{-1}I_1F\big(R_1^{-1}x+ (\alpha_3/\alpha_1)^{-1}\gamma q\big) = \big[x+\exp(i\theta(x))\gamma q\big] \quad \textrm{for} \quad  x \in V \cap \G[c_1',c_2'],\quad \langle x,x \rangle =-1.\]
In particular, we may suppose $\theta(c_1') = 0$.

There is $\widetilde\theta$ in $\G[c_1,c_2]$ such that $\theta(x) =\widetilde\theta(x)$ in a small compact neighborhood $N_1 \subset V$ of $c_1'$ such that $\mathrm{supp} (\widetilde\theta)\subset V$. Furthermore, we can extend $\widetilde\theta$ to all the quadrilateral $P$ in such a way that $\widetilde\theta$ is zero over the geodesics $\G[c_2',c_3']$, $\G[c_3',c_4']$, and $\G[c_4',c_1']$. Therefore, we can consider $\widehat{F}(x+\gamma q) = F\Big(x+\exp\big(i\widetilde\theta(x)\big)\gamma q\Big)$, with $\langle x,x \rangle = -1$. With this new map we have
\[\widehat F(x +\gamma q) = I_1\widehat F\big(R_1^{-1}x + (\alpha_3/\alpha_1)^{-1}\gamma  q\big) \quad \textrm{for} \quad  x \in N_1 \cap \G[c_1',c_2'],\quad \langle x,x \rangle =-1,\]
where we are using $\widetilde\theta(R^{-1}x) =0$ because $R_1^{-1}x \in G[c_4',c_1']$.
\qed

\medskip

We may also suppose the same kind of property for the other vertices: For $i=2,3,4$ we have a small neighborhood $N_i$ of the point $c_i'$ in $\PP(q^\perp)$ such that
\begin{align*}
&F(x + \gamma q) = I_3^{-1} F\big(R_3x+(\beta_3/\beta_1)\gamma q\big) &\textrm{for}&\quad  x \in N_2 \cap \G[c_2',c_3'],\quad \langle x,x \rangle =-1,
\\
&F(x+\gamma q) = I_3 F\big(R_3^{-1}x +  (\gamma_3/\gamma_1)\gamma q\big) &\textrm{for}&\quad  x \in N_3 \cap\G[c_3',c_4'],\quad \langle x,x \rangle =-1,
\\
&  F(x+\gamma q)=I_1^{-1}F\big(R_1x+ \gamma q\big)&\textrm{for} &\quad  x\in N_4 \cap \G[c_4',c_1'],\quad \langle x,x \rangle =-1.
\end{align*}

\subsection{Constructing complex hyperbolic disc orbigoodles over \texorpdfstring{$\SP^2(n_1,n_2,n_3)$}{S2(n1,n2,n3)}}
\label{subsection orbibundle}

An {\bf $n$-orbifold $B$} is a space locally modeled by quotients of the form $\DD^n/\Gamma$, where $\Gamma$ is a finite subgroup of $\mathrm{O}(n)$. All orbifolds considered in this paper are locally oriented, which means that we are only considering trivializations with $\Gamma \subset \mathrm{SO}(n)$. More technically by space we mean diffeological space (for details see \cite{bot}). A diffeomorphism $\phi:\DD^n/\Gamma \to D$, where $D$ stands for an open subset of $B$, is called {\bf orbifold chart}. Furtheremore, if $\phi([0]) = p$ we say that the orbifold chart is centered at $p$. We say that $p$ is a regular point if the finite group $\Gamma$ corresponding to a chart centered at $p$ is trivial, that is, the orbifold is locally Euclidian around $p$; the point is called singular otherwise and the order of the singular point is the cardinality of the group $\gamma$. Since we are interested in orbibundles over $2$-orbifolds, the groups $\Gamma's$ are generated by $\exp(2\pi i/n)$, where we think of $\DD^2$ as the unit open ball on the complex plane.

\begin{defi}(see \cite[3.1.  Orbibundles]{bot})\label{definition orbibundle} Consider a smooth map between orbifolds $\zeta: L \to B$. We say $\zeta$ is a {\bf disc orbibundle}, if for every point $p \in B$ there is an orbifold chart $\phi:\DD^n/\Gamma \to D$ centered at $p$ satisfying the following properties:
    \begin{itemize}
        \item there is a smooth action of $\Gamma$ on $\DD^n \times \DD^2$ of the form $h(x,f) = (hx,a(h,x)f)$, where\break $a: \Gamma \times \DD^n \to \mathrm{Diff}(\DD^2)$ is smooth and $\mathrm{Diff}(\DD^2)$ stands for the group of diffeomorphisms of~$\DD^2$;
        \item there is a diffeomorphism $\Phi:(\DD^n \times \DD^2)/\Gamma \to \zeta^{-1}(D)$ such that the diagram

		\begin{equation*}
		\begin{tikzcd}
		(\DD^n \times \DD^2)/\Gamma \arrow[rr, "\Phi"] \arrow[d,"\mathrm{pr}_1"'] &  & \zeta^{-1}(D) \arrow[d, "\zeta"] \\
		\DD^n/\Gamma \arrow[rr, "\phi"]                                &  & D
		\end{tikzcd}
		\end{equation*}
	    commutes, where $\mathrm{pr}_1([x,f])=[x]$.
	   \end{itemize}

\end{defi}

A {\bf disc orbigoodle} (see \cite[Definition  23]{bot}) is a special case of disc orbibundle. Consider a simply-connected manifold $\HH$ on which acts a group $G$ properly discontinuously. If we have an action of $G$ on $\HH \times \DD^2$ by diffeomorphisms of the form $g(p,v) =(gp, a(g,p)v) $ then the quotient $(\HH \times \DD^2)/G \to \HH/G$ is a disc orbibundle. Such orbibundles are called disc orbigoodles. All disc orbibundles of this paper are disc orbigoodles where $\HH$ is the hyperbolic plane.

\smallskip

A natural \/$\SP^1$-action is defined on $Q'$ (the polyhedron $Q'$ is defined right after Lemma \ref{deformationlemma}) because
\[Q'=\bigcup_{x \in P}\big(\mathrm L[q,x]\cap \HH_\CC^2\big),\]
where $\mathrm L[q,x]$ is the complex projective line connecting $q$ and $x$,
and each disc $\overline{\HH_\CC^2} \cap \mathrm L[q,x] $ has the point $x$ as center. The action we define is simply given by rotation around $x$,
\[\gamma[x+\theta q] = [x+\gamma \theta q],\]
where $\langle x,x \rangle =-1$, $\gamma \in \SP^1$, and $|\theta|\leq 1$. Therefore, we can define an $\SP^1$-action on $Q$ using the diffeomorphism $F$. Since $F$ is an isometry at the level of the discs foliating the quadrangles $\mathcal Q',\mathcal Q$, $I_1\mu_4=\mu_1$ and $I_3\mu_2=\mu_3$ (remind that the curves $\mu_i$'s are image under $F$ of the curves defining the boundary of the quadrilateral $P$), we conclude
$$\gamma I_1F(x)=I_1 \gamma F(x) \quad \text{for}\quad x\in B[C_1',C_4'],
$$
$$\gamma I_3F(x)=I_3\gamma F(x) \quad \text{for}\quad x\in B[C_2',C_3'],
$$
$\gamma \in \SP^1$. In particular, since $F(c_i')=c_i$ for each vertex $C_i=F(C_i')$, we obtain that $F(\gamma x)$ is the rotation of $F(x)$ with respect to the center $c_i$ of $C_i$ and angle given by the unitary complex number~$\gamma$.

Note that the image of the quadrilateral $P$ under $F$ in addition to the action of $G$ on $\HH_\CC^2$ provides an embedded disc $D$ transversal to all discs foliating $\HH_\CC^2$ and stable under action of $\SP^1$. Hence, the quotient $L: = \HH_\CC^2/G\to D/G$ is a disc orbigoodle and by construction $D/G = \SP^2(n_1,n_2,n_3)$.

Furthermore, from $\partial_\infty Q$ we can build the $\SP^1$-orbibundle $\SP^1(L) \to D/G$, from which we will deduce the formula for the Euler number of the disc bundle $ L \to D/G$. Let $\pi: \partial_\infty Q \to \SP^1(L)$ be the quotient map.
It is interesting to note that the action on $\partial L$ is not necessarily principal, i.e., there are points $x \in \partial_\infty Q$ such that the map $\SP^1 \ni \gamma \mapsto \pi(\gamma x) \in \SP^1(L)$ is non-injective. More precisely, the action fails to be principal on the circles $\pi(\partial C_j)$'s.

Take a small ball $V_i$ of radius $r$ and center $c_i'$ on $P$ for $i=1,2,3,4$. Let's see what happens nearby these non-principal circles. Without loss of generality, we will work with $i=1$. We have the open set \[U \coloneq  F\left[\bigcup_{x \in V_1} \mathrm{L}(x,q) \cap \partial \HH_\CC^2 \right]\]
of $\partial_\infty Q$ and the open set $W \coloneq  \pi(U)$ in $\SP^1(L)$. Let $p'_1$ be the orthogonal point $c'_1$ on the projective line $\PP(q^\perp)$ such that $\langle p'_1,p'_1 \rangle =1$, $c'_2 \in \RR c'_1+\RR p'_1$ and the geodesic curve $t \mapsto [\cosh(t)c'_1+\sinh(t)p'_1]$ reaches $c'_2$ for some $t>0$, that is, this curve goes from $c'_1$ to $c'_2$.

Consider the map $\Lambda: \SP^1 \times S \to W$ given by
\[\Lambda (\gamma,z) =\pi \circ F \left[\frac{c_1' + z p_1'}{\sqrt{1-|z|^2}} + \gamma q\right],\]
where $S$ is the intersection of $\overline{\DD_\epsilon^2} \subset \CC$, the disc of center $0$ and radius $\epsilon$ such that $\cosh(r) = 1/\sqrt{1-\epsilon^2}$, and the sector given by the inequality $0\leq \arg(z)\leq 2\pi/n_1$. The sides of $S$ can be glued because, if $z$ is real, $\Lambda(\gamma,z) = \Lambda(\gamma,\xi z)$, where $\xi = \exp(2\pi/n_1)$. Therefore, we have the smooth map
$$\Lambda:\SP^1 \times \big(\overline{\DD_\epsilon^2}/\langle \xi \rangle\big) \to W,
$$
and using the natural projection $\overline{\DD_\epsilon^2} \to \overline{\DD_\epsilon^2}/\langle \xi \rangle$, we have the smooth map
$$\widetilde\Lambda:\SP^1 \times \overline{\DD_\epsilon^2} \to W.
$$

Remember the eigenvalues of $I_1$ are $\alpha_1$, $\alpha_2$ and $\alpha_3$. Let $e^{2\pi i l_1/n_1} = \alpha_3/\alpha_1$ and $e^{-2\pi i /n_1} = \alpha_2/\alpha_1$, with $0\leq l_1 < n_1$.

Taking the diffeomorphism $\eta(\gamma,z) \coloneq  (\xi^{l_1}\gamma ,\xi^{-1} z)$ on $\SP^1 \times \overline{\DD_\epsilon^2}$, we have the equivariant diffeomorphism
\[\widehat\Lambda:\big(\SP^1 \times \overline{\DD_\epsilon^2}\big)/\langle \eta \rangle \to W\]
as a consequence of lemma \ref{lema difeomorfismo}. So $W$ is a solid torus (see \cite[Lemma 20]{bot}) with an $\SP^1$ action which is principal except for the circle $\big(\SP^1 \times 0\big)/\langle \eta \rangle$. Hence we have a trivialization of the $\SP^1$-orbibundle around the fiber $\pi(\partial_\infty C_1)$.

If we write $\beta_3/\beta_1 = e^{2\pi i l_2/n_2}$ and $\gamma_3^{-1}/\gamma_1^{-1} = e^{2\pi i l_3/n_3}$, with $1\leq l_1 <n_i$, we obtain the same kind of trivialization of the $\SP^1$-orbibundle as described above for $\pi(\partial_\infty C_2)$ and $\pi(\partial_\infty C_3)$.

\subsection{An integer contribution to the Euler number}\label{subsection meridional curve}
We now tackle the problem of calculating the Euler number of the constructed orbibundles. First, we need to introduce a particular curve $d$ for the quadrangle~$\mathcal Q$.

\begin{figure}[H]
	\centering
	\includegraphics[scale=.8]{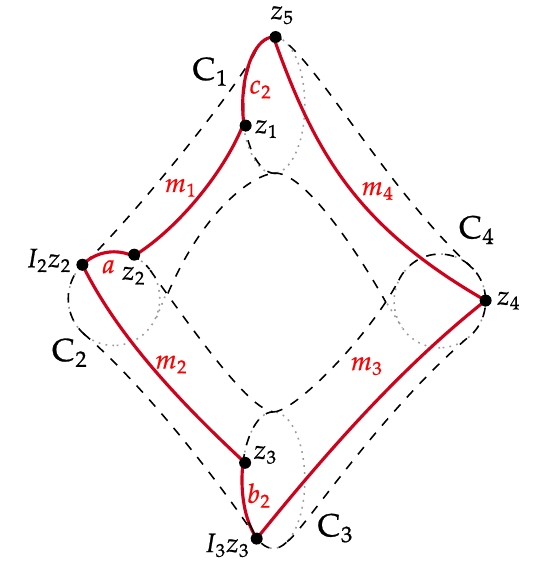}
	\caption{Meridional curve.}
	\label{meridionalcurvefigure}
\end{figure}

Take a point $z_1$ on $\partial_\infty C_1$ and define the following curves:

\smallskip

$\bullet$  the meridional curve $m_1\subset\partial_\infty\B[C_1,C_2]$ that
begins at $z_1\in\partial_\infty C_1$ and ends at
$z_2\in\partial_\infty C_2$;

$\bullet$ the naturally oriented simple arc
$a\subset\partial_\infty C_2$ that begins at $z_2$ and ends at
$I_2 z_2$;

$\bullet$ the meridional curve $m_2\subset\partial_\infty\B[C_2,C_3]$ that
begins at $I_2 z_2$ and ends at $z_3\in\partial_\infty C_3$;

$\bullet$ the naturally oriented simple arc $b_2\subset\partial_\infty S$
that begins at $z_3$ and ends at $I_3z_3$;

$\bullet$ the meridional curve $m_3\subset\partial_\infty\B[C_3,C_4]$ that
begins at $I_3z_3$ and ends at $z_4\in\partial_\infty C_4$;

$\bullet$ the meridional curve $m_4\subset\partial_\infty\B[C_4,C_1]$ that
begins at $z_4$ and ends at $z_5\in\partial_\infty C_1$;

$\bullet$ the naturally oriented simple arc $c_2\subset\partial_\infty C_1$
that begins at $z_5$ and ends at $z_1$.

\smallskip

Note that $z_4 = I_3 I_2 z_2$, because $I_3 m_2 = m_3$. Therefore, $z_4 = I_1^{-1}z_2$ and consequently $z_5 = I_1^{-1}z_1$.

Let
\begin{equation} \label{curved}d \coloneq  m_1 \cup a \cup m_2 \cup b_2 \cup m_3 \cup m_4 \cup c_2
\end{equation}
and let $s$ stand for a generator of $H_1(\partial_\infty Q,\ZZ)$. Then there exists $f\in \ZZ$ such that $d=fs$ in $H_1(\partial_\infty Q,\ZZ)$. This integer $f$ is an important component of the Euler number of the orbibundles we will encounter in Section \ref{subsection euler number}. It will be expressed in a more computational friendly manner in Section \ref{subsection holonomy of the quadrangle}.

\subsection{Euler Number of the constructed disc bundles}\label{subsection euler number}
\begin{defi} \label{Def euler number}
If $M \to B$ is an $\SP^1$-orbibundle over a oriented compact connected $2$-orbifold with singular points $x_1, \ldots,x_n$ then the Euler number is calculated as follows: Take a regular point $x_0$ and for each $i=0,\cdots,n$ consider a small smooth closed disc $D_i$ centered at $x_i$ trivializing the $\SP^1$-orbibundle $M \to B$. The $\SP^1$-orbibundle restricted over the surface with boundary $B'=B\setminus \sqcup_i D_i$ is trivial, since $\SP^1$-bundles over graphs are trivial and $B'$ is homotopically equivalent to a graph. Consider a section $\sigma$ for $M|_B' \to B'$ and a fiber $s$ over a regular point, oriented accordingly to action of $\SP^1$ on $M$. The Euler number of the $\SP^1$-orbibundle $M \to B$ is defined by the identity
$$\sigma|_{\partial B'} = - e(M)s$$
in $H_1(M,\QQ)$ (See \cite[Definition 16]{bot}).
\end{defi}

We now prove the following proposition:

\begin{prop}\label{prop Euler number} Assume that we have constructed a disc orbibundle using the quadrangle of bisectors satisfying the quadrangle conditions in accordance with the Sections \ref{section discreteness} and \ref{subsection orbibundle}.
Let $l_1,l_2,l_3$ be least non-negative integers such that $$\alpha_3/\alpha_1 = e^{2\pi i l_1/n_1}, \qquad \beta_3/\beta_1 = e^{2\pi i l_2/n_2},\qquad \gamma_3^{-1}/\gamma_1^{-1} = e^{2\pi i l_3/n_3}.$$

The Euler number of the disc orbibundle $M \to B$ is
\[e(M) = f-\frac{l_1}{n_1}- \frac{l_2}{n_2} - \frac{l_3}{n_3},\]
where $f$ is the parameter defined in Section \ref{subsection meridional curve}.
\end{prop}

{\bf Proof.} Following \cite[3.2.  Euler number of $\SP^1$-orbibundles over $2$-orbifolds]{bot} the Euler number of the disc orbibundle $L \to D/G$ described in the Section \ref{subsection orbibundle} is the Euler number of the $\SP^1$-orbibundle $\SP^1(L) \to D/G$.

Now we apply the Definition \ref{Def euler number} of Euler number to the particular bundle $\SP^1(L) \to D/G$. Let us also denote $\SP^1(L)$ by $M$ and $D/G$ by $B$. Remember that $B$ is the quotient of the hyperbolic plane by the turnover group. Here we think of $B$ as the quotient of $P$ by the gluing relations described by the turnover group (the quadrilateral $P$ is the fundamental domain for the turnover group as described in Section \ref{turnover definition}). Hence we denote the point under the fiber $\pi(\partial C_i)$ by $[c_i']$. The points $[c_i']$ are the only singular points of $B$.

\begin{figure}[H]
	\centering
	\begin{minipage}{.5\textwidth}
		\centering
		\includegraphics[scale = .59]{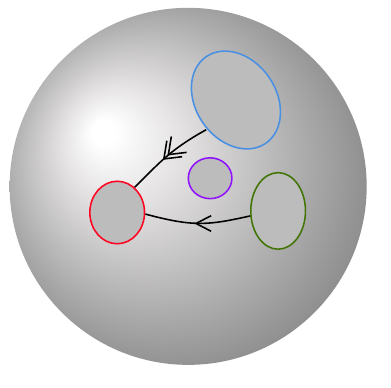}
		\caption*{(a)}
	\end{minipage}%
	\begin{minipage}{.5\textwidth}
		\centering
		\includegraphics[scale =1]{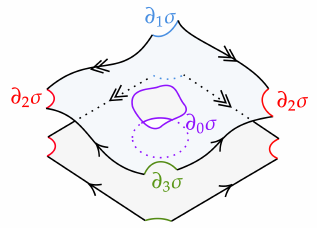}
		\caption*{(b)}
	\end{minipage}
	\caption{\textbf{(a)} Surface $B'$, and \textbf{(b)} Section $\sigma:B' \to M'$}
 \label{section}
\end{figure}
Removing small open discs $D_1$, $D_2$ and $D_3$ on $B$ around the three singular points $[c_1'],[c_3'],[c_2']=[c_4']$ and one small disc $D_0$ around a regular point $[x_0]$ in $B$, with $x_0 \in \mathring P$, we have the surface with boundary $B'\coloneq  B \setminus \sqcup_i D_i$. The $3$-manifold $M'= \zeta^{-1}(B')$ is a principal $\SP^1$-bundle over $B'$. Notice that $M \setminus M'$ is made of four solid tori $W_0,W_1, W_2, W_3$, where $W_i = \zeta^{-1} D_i$.

For any section $\sigma: D' \to M'$, let's denote $\sigma|_{\partial D_i}$ by $\partial_i \sigma$. See Figure \ref{section}. Remember the curve $d$ defined in Section \ref{subsection meridional curve}. Shrinking $F^{-1}(d)$  inside the torus $\partial_\infty Q'$ we can build a section $\sigma: D' \to M'$ satisfying the identities (See Figure \ref{meridionalcurvefigure}) $$\partial_0\sigma=\pi(d),\quad\partial_1\sigma  =-\pi(c_2),\quad\partial_2\sigma=-\pi(b_2),\quad\partial_3\sigma=-\pi(a)$$ in $H_1(M, \QQ)$.

The identity $n_i\partial_i\sigma=-l_i\omega_i$ in $H_1(M,\QQ)$ holds for $i=1,2,3$, where $\omega_i$ is the orbit of a point in $\partial W_i$. Furthermore, $\omega_0 = \omega_1 = \omega_2 = \omega_3 = s$ in $H_1(M,\QQ)$.

Let us prove the identity  $n_i \partial_i \sigma = -l_i\omega_i$ for $i=1$.

Consider a generator $s'$ of the fundamental group of $\pi(\partial C_1)$.
\begin{figure}[H]
	\centering
	\begin{minipage}{.3\textwidth}
		\centering
		\includegraphics[scale = 1]{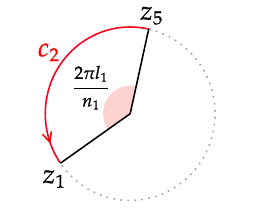}
		\caption*{(a)}
	\end{minipage}
	\begin{minipage}{.3\textwidth}
		\centering
		\includegraphics[scale = 1]{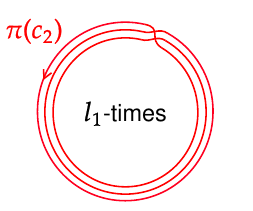}
		\caption*{(b)}
	\end{minipage}
	\begin{minipage}{.3\textwidth}
		\vspace{-1.3cm}
		\centering
		\includegraphics[scale = 1]{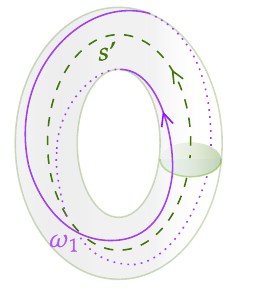}
		\caption*{(c)}
	\end{minipage}
\caption{\textbf{(a)} Curve $c_2$ in $\partial C_1$, \textbf{(b)} loop $\pi(c_2)$ in $\pi(\partial C_1)$, and \textbf{(c)} loop $\omega_1$ on the solid torus $W_1$.}
\label{eulernumber}
\end{figure}

We can think of $\omega_1$ as $\SP^1 \to M$ given by
$$\gamma \mapsto \pi \circ F \left[\frac{(c_1'+ z p_1')}{\sqrt{1-|z|^2}}+\gamma q\right]
$$
for a fixed $z$.

Notice $\omega_1 = n_1s'$, because $\omega_1$ is homotopic to the curve $\gamma \mapsto [c_1' + \gamma q]$ in $M$, which is a curve that goes $n_1$ times around the circle $\pi(\partial C_1)$, and $\partial_1\sigma=-l_1s'$, because $\partial_1\sigma=-\pi(c_2)$ in $M$ and $\pi(c_2)$ goes $l_1$ times around the circle $\pi(\partial C_1)$, since in $\partial C_1$ the curve $c_2$ is constructed as the curve going from $z_5$ to $z_1$ following the natural orientation of the circle and $I_1 z_1 = z_5$.
See Figure \ref{eulernumber}.

Therefore, we have
\[n_1\partial_1\sigma=-l_1\omega_1\quad\text{in}\quad H_1(M,\QQ).\]

In the case of $i=0$, we have $\partial_0\sigma=\pi(d)$ and, therefore, we have $\partial_0\sigma=f\omega_0$ in $H_1(M',\ZZ)$, because $d=fs$.

Note $\partial D_i$ is oriented in opposite direction of $\partial B'$. Therefore, in $H_1(M,\mathbb Q)$ we can write
\[\partial \sigma = \sum_{i=0}^3 - \partial_i \sigma  = \left(-f+\frac{l_1}{n_1}+\frac{l_2}{n_2}+\frac{l_3}{n_3}\right)s\]
and, therefore,
\[e(M) = f-\frac{l_1}{n_1}- \frac{l_2}{n_2} - \frac{l_3}{n_3}. \eqno{_\blacksquare}\]

\subsection{Holonomy of triangle of bisectors: computing the parameter $f$}\label{subsection holonomy of the quadrangle}

In Section \ref{subsection meridional curve} we define the curve $d$, shown in Figure \ref{meridionalcurvefigure}, and the integer $f$, necessary to calculate the Euler number.
In order to express this integer explicitly we use the concept of holomony of a transversal triangle of bisectors.

Given a counterclockwise oriented transversal triangle of bisectors $\Delta(L_1,L_2,L_3)$, let $M_1,M_2,M_3$ be the middle slices (see Section \ref{subsection bisectors}) of the segments of bisectors $\B[L_1,L_2]$, $\B[L_2,L_3]$, $\B[L_3,L_1]$. The product $I$ of the reflections in the middle slices $M_1$, $M_2$, $M_3$ (in that order) is called the {\it holonomy\/} of the triangle $\Delta(L_1,L_2,L_3)$ \cite[Section 2.5.1]{discbundles}. Note that $I$ stabilizes $L_1$.

The triangle $\Delta(L_1,L_2,L_3)$ is respectively called {\it elliptic, parabolic,} or {\it hyperbolic\/} when the holonomy $I$ restricted to $L_1$ is an elliptic, parabolic, or hyperbolic isometry of the Poincar\'e disc $L_1$. The holonomy of a counterclockwise oriented transversal triangle cannot be trivial, that is, $I$ restricted to $L_1$ is never the identical isometry; moreover, parabolic triangles are always $L$-parabolic, that is, the holonomy restricted to $L_1$ moves its non-fixed points in the clockwise sense  \cite[Theorem 2.24]{discbundles}. In the case of a hyperbolic triangle, the action of $I$ on $L_1$ divides $\partial_\infty L_1$ into the $L$ and $R$-parts: the $L$-part (respectively, the $R$-part) consists of those points that are moved by $I$ in the clockwise sense (respectively, counterclockwise sense). In the elliptic and parabolic cases, all (non-fixed) points belong to the $L$-part.

A simple closed curve in the torus $\partial_\infty\Delta(L_1,L_2,L_3)$ is called a {\it trivialing curve} of the triangle if it generates the fundamental group of the solid torus $\Delta(L_1,L_2,L_3)$ and is contractible in the ideal boundary of the polyhedron bounded by $\Delta(L_1,L_2,L_3)$ (see Section \ref{section discreteness}).

\smallskip

As introduced in Section \ref{subsection meridional curve}, let
\begin{equation}\label{meridional curve}
d\coloneq m_1\cup a\cup m_2\cup b_2\cup m_3\cup m_4\cup c_2
\end{equation}
be the oriented closed curve in the boundary of the solid torus $\partial_\infty Q$, where $Q$ stands for the polyhedron of the quadrangle $\mathcal Q$. Remind that the group $H_1(\partial_\infty Q,\ZZ)$ is generated by $[s]$, where $[s]$ stands for the naturally oriented boundary of $C_1$. Hence, $[d]=f[s]$ for some $f\in\ZZ$. In order to express $f$ in terms of the holonomies of the
triangles $\Delta(C_1,C_2,C_4)$ and $\Delta(C_3,C_4,C_2)$, we introduce more
points and curves:

\medskip

$\bullet$ the meridional curve $m_2'\subset\partial_\infty\B[C_2,C_3]$
that begins at $z_2$ and ends at $z_3'\in\partial_\infty C_3$;

$\bullet$ the meridional curve $m\subset\partial_\infty\B[C_2,C_4]$ that
begins at $z_2$ and ends at $z_4'\in\partial_\infty C_4$;

$\bullet$ the meridional curve $m_3'\subset\partial_\infty\B[C_4,C_3]$ that
begins at $z_4'$ and ends at $z_3''\in\partial_\infty C_3$;

$\bullet$ the naturally oriented arc $b\subset\partial_\infty C_3$ that
begins at $z_3'$ and ends at $z_3''$;

$\bullet$ the meridional curve $m_4'\subset\partial_\infty\B[C_4,C_1]$ that
begins at $z_4'$ and ends at $z_5'\in\partial_\infty C_1$;

$\bullet$ the naturally oriented arc $c\in\partial_\infty C_1$ that
begins at $z_5'$ and ends at $z_1$.

\begin{figure}[H]
	\centering
	\includegraphics[scale=.75]{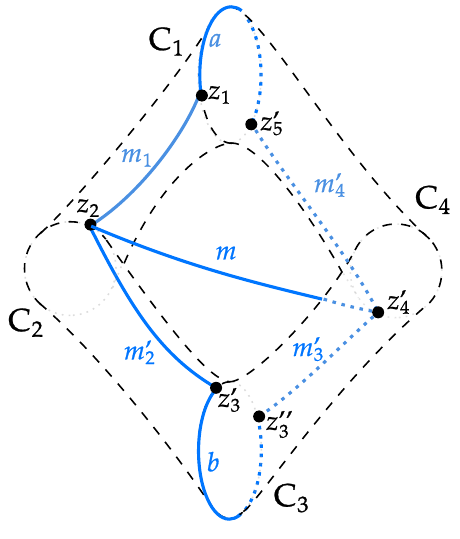}
	\caption{Auxiliary curves}
    \label{fig: aux curves}
\end{figure}

Denote by $I$ the holonomy of the triangle $\Delta(C_1,C_2,C_4)$ and by $J$
the holonomy of the triangle $\Delta(C_3,C_4,C_2)$. By the definition of
holonomy of a triangle, we have $z_3''=J^{-1}z_3'$ and $z_5'=Iz_1$.

Let us assume that $z_1$ belongs to the $L$-part of $\Delta(C_1,C_2,C_4)$ and that $z_3'$ belongs to the $L$-part of $\Delta(C_3,C_4,C_2)$ (this is harmless because all the triangles that appear in the constructed orbibundles are elliptic). In this case, by \cite[Theorem 2.24]{discbundles}, the closed oriented curve $m_1\cup m\cup m_4'\cup c$ is a trivializing curve of $\Delta(C_1,C_2,C_4)$. Similarly, ${m_3'}^{-1}\cup m^{-1}\cup m_2'\cup b$ is a trivializing curve of the triangle $\Delta(C_3,C_4,C_2)$. (We denote by $x^{-1}$ the (not necessarily closed) curve $x$ taken with the opposite orientation.) By \cite[Remark 2.21]{discbundles}, $m_1\cup m_2'\cup b\cup{m_3'}^{-1}\cup m_4'\cup c$ is a trivializing curve of the quadrangle $\mathcal Q$, that is, it generates the fundamental group of $\mathcal Q$ and is contractible in $\partial_\infty Q$. In terms of $1$-chains modulo boundaries, this means that
\begin{equation}\label{holonomy identity 1}
[m_1]+[m_2']+[b]-[m_3']+[m_4']+[c]=0.
\end{equation}

Finally, we introduce the following arcs:

\medskip

$\bullet$ the naturally oriented simple arc $b_1\subset\partial_\infty C_3$
that begins at $z_3'$ and ends at $z_3$;

$\bullet$ the naturally oriented simple arc $b_3\subset\partial_\infty C_3$
that begins at $I_3z_3$ and ends at $z_3''$;

$\bullet$ the naturally oriented simple arc $c_1\subset\partial_\infty C_1$
that begins at $z_5$ and ends at $z_5'$.

\medskip

\begin{figure}[H]
	\centering
	\includegraphics[scale=.8]{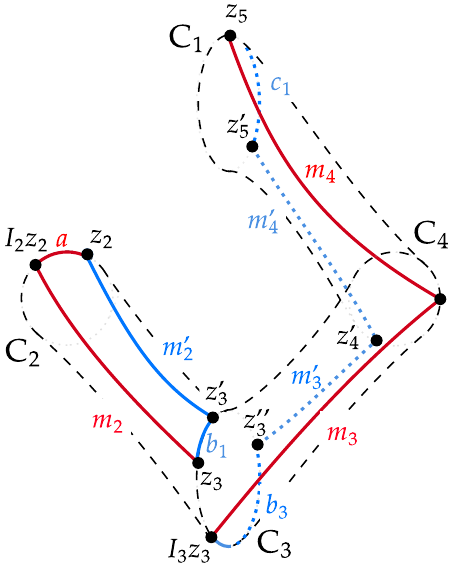}
	\caption{Cylinders $\partial_\infty\B[C_2,C_3]$ and $\partial_\infty\B[C_3,C_4]\cup\partial_\infty\B[C_4,C_1]$.}
 \label{thecylinders}
\end{figure}

By looking at the cylinder $\partial_\infty\B[C_2,C_3]$, represented in Figure \ref{thecylinders}, it is easy to see
that
\begin{equation}\label{holonomy identity 2}
[a]+[m_2]-[b_1]-[m_2']=0.
\end{equation}

Similarly, by considering the cylinder
$\partial_\infty\B[C_3,C_4]\cup\partial_\infty\B[C_4,C_1]$, one obtains that
\begin{equation}\label{holonomy identity 3}[m_3]+[m_4]+[c_1]-[m_4']+[m_3']-[b_3]=0.\end{equation}

It follows from equations \eqref{meridional curve}, \eqref{holonomy identity 2}, and \eqref{holonomy identity 3} that
\begin{multline}\label{holonomy identity 4}
[d]=[m_1]+[a]+[m_2]+[b_2]+[m_3]+[m_4]+[c_2]=\\
[m_1]+[m_2']+[b_1]+[b_2]+[b_3]-[m_3']+[m_4']-[c_1]+[c_2].
\end{multline}

We introduce the following notation. Let $C$ be an oriented circle and
let $t_1,t_2,t_3\in C$. We define $o(t_1,t_2,t_3)=1$ if $t_1,t_2,t_3$ are
pairwise distinct and not in cyclic order. Otherwise, we put
$o(t_1,t_2,t_3)=0$.

\medskip

\begin{lemma}\label{holonomy lemma} Let\/ $C$ be an oriented circle and let\/
	$t_1,t_2,t_3,t_4\in C$ be such that\/ $t_3\ne t_1\ne t_4$. Following the
	orientation of\/ $C$, we define four simple arcs\/ {\rm(}some of them may
	consist of a single point\/{\rm):} $a_i\subset C$ joining\/ $t_i$ and\/
	$t_{i+1}$ for\/ $i=1,2,3$ and\/ $a\subset C$ joining\/ $t_1$ and\/ $t_4$.
	Then we have
	$$[a_1]+[a_2]+[a_3]-[a]=o(t_1,t_2,t_3)[C]+o(t_3,t_4,t_1)[C]$$
	in\/ $C_1(C,\ZZ)/\partial C_0(C,\ZZ)$.
	{\rm(}Of course, we take\/ $[C]$ as a generator of\/
	$H_1(C,\ZZ)$.{\rm)}
\end{lemma}

\medskip

{\bf Proof.} Define the following oriented simple arcs: $a_4$ joining
$t_4$ and $t_1$; $m_1$ joining $t_1$ and $t_3$; and $m_2$ joining $t_3$
and $t_1$. It follows from $t_1\ne t_4$ that $[a]+[a_4]=[C]$.
Analogously, $t_1\ne t_3$ implies $[m_1]+[m_2]=[C]$. By drawing the
corresponding arcs in $C$, it is easy to see that
$[a_1]+[a_2]-[m_1]=o(t_1,t_2,t_3)[C]$ and
$[a_3]+[a_4]-[m_2]=o(t_3,t_4,t_1)[C]$. So,
$$[a_1]+[a_2]+[a_3]-[a]=[a_1]+[a_2]+[a_3]+[a_4]-[C]=$$
$$=o(t_1,t_2,t_3)[C]+[m_1]+o(t_3,t_4,t_1)[C]+[m_2]-[C]=
o(t_1,t_2,t_3)[C]+o(t_3,t_4,t_1)[C].\eqno{_\blacksquare}$$

Applying Lemma \ref{holonomy lemma} to the naturally oriented circle
$\partial_\infty C_3$ and the points
$z_3',z_3,I_3z_3,J^{-1}z_3'\in\partial_\infty C_3$ (note that
$z_3'\ne J^{-1}z_3'$ always hold and one can assume that $z_3'\ne I_3z_3$)
we obtain
\begin{equation} \label{holonomy identity 5}
[b_1]+[b_2]+[b_3]=[b]+o(z_3',z_3,I_3z_3)[s]+o(I_3z_3,J^{-1}z_3',z_3')[s].
\end{equation}
In the naturally oriented circle $\partial_\infty C_1$ we have
\begin{equation}\label{holonomy identity 6}
[c_1]+[c]=[c_2]+o(I_1^{-1}z_1,Iz_1,z_1)[s]
\end{equation}
since $[\partial_\infty C_1]=[s]$. Therefore, it follows from \eqref{holonomy identity 4}, \eqref{holonomy identity 5}, and
\eqref{holonomy identity 6} that
$$[d]=[m_1]+[m_2']+[b]-[m_3']+[m_4']+[c]+o(z_3',z_3,I_3z_3)[s]+
o(I_3z_3,J^{-1}z_3',z_3')[s]-o(I_1^{-1}z_1,Iz_1,z_1)[s].$$
Hence, by \eqref{holonomy identity 1},
$$[d]=o(z_3',z_3,I_3z_3)[s]+o(I_3z_3,J^{-1}z_3',z_3')[s]-
o(I_1^{-1}z_1,Iz_1,z_1)[s],$$
that is,
$$f=o(z_3',z_3,I_3z_3)+o(z_3',I_3z_3,J^{-1}z_3')-o(z_1,I_1^{-1}z_1,Iz_1).$$

\newpage
\section{Toledo invariant}
\label{section Toledo invariant}
\begin{defi} \label{def toledo}
Let $\rho:G\to\PU$ be a $\PU$-representation of the turnover group $G$ and let $\phi:\HH_\RR^2 \to \HH_\CC^2$ be a $G$-equivariant map. The {\it Toledo invariant\/} of $\rho$ is defined by the formula \[\tau(\rho) = 4\cdot\frac{1}{2\pi} \int_P \phi^\ast\omega,\]
where $P$ is a fundamental domain in $\HH_\RR^2$ for the action of $G$ (see Section \ref{turnover definition} and Figure \ref{fundamental domain and orbifold}). 
\end{defi}
The number $\tau$ does not depend on the choice of the $G$-equivariant map $\phi$. For details about the Toledo invariant in the context of orbifolds, see \cite[Definition 35]{bot} and \cite{krebs}). The factor $4$ in our formula for the Toledo invariant comes from the fact that our metric is four times the usual one.

Let $\mathcal Q$ be the quadrangle associated to the representation $\rho$. We assume that it satisfies the quadrangle conditions in Section \ref{subsection quadrangle conditions}. In order to calculate the Toledo invariant of $\rho$, we introduce in $\mathcal Q$ several curves as illustrated in Figure \ref{toledo curve}. First, we define the oriented meridional curves
$$m_1\coloneq [c_1',c_2]\subset\B[C_1,C_2],\quad
m_2\coloneq [c_2,c_3']\subset\B[C_2,C_3],\quad
m_3^{-1}\coloneq I_3 m_2=[c_4,I_3c_3']\subset\B[C_3,C_4],$$
$$m_4^{-1}\coloneq I_1^{-1}m_1=[I_1^{-1}c_1',c_4]\subset\B[C_4,C_1],$$
with $c_1'\in C_1$ and $c_3'\in C_3$ (note that $c_4 = I_1^{-1}c_2 = I_3 c_2$). We also introduce the oriented geodesics
$$h_1\coloneq \G[c_1,c_1']\subset C_1,\quad
h_2\coloneq \G[c_3',c_3]\subset C_3,\quad
h_3^{-1}\coloneq I_3h_2=\G[I_3c_3',c_3]\subset C_3,$$
$$h_4^{-1}\coloneq I_1^{-1}h_1=\G[c_1,I_1^{-1}c_1']\subset C_1$$
\begin{figure}[H]
	\centering
	\includegraphics[scale=.6]{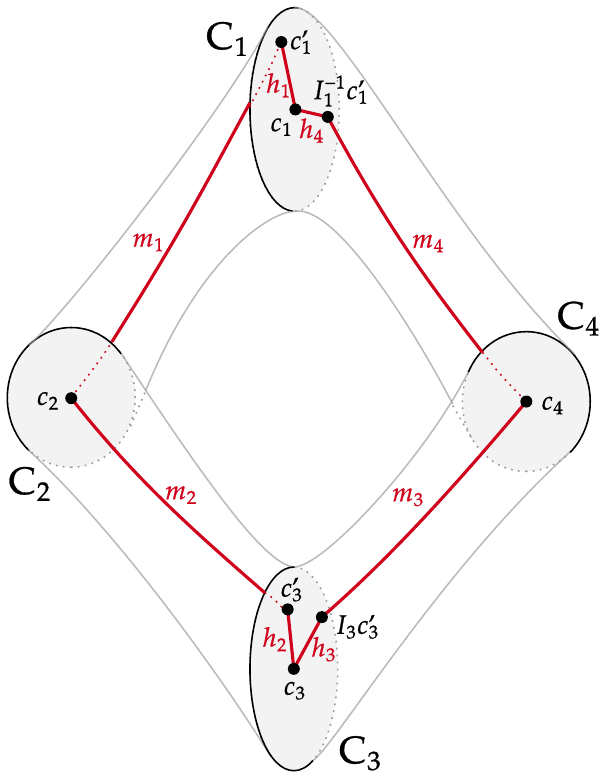}
	\caption{Curve $c$.}
	\label{toledo curve}
\end{figure}
\noindent
thus obtaining the closed oriented curve
\begin{equation}\label{toledo's curve}\zeta\coloneq h_1\cup m_1\cup m_2\cup h_2\cup h_3\cup m_3\cup m_4\cup h_4.
\end{equation}

Following the notation in Section \ref{quadrangle revisited}, let $\alpha_1,\beta_1,\gamma_1^{-1}$ be the eigenvalues of $I_1,I_2,I_3$ corresponding to negative eigenvectors. The proof below is similar to that of \cite[Proposition 2.7]{discbundles}. The strategy of the proof is the following. By Stokes theorem, the Toledo invariant of $\rho$ can be obtained by integrating a K\"ahler potential along $\zeta$ because the quadrangle conditions allow one to build a $G$-equivariant map $\HH_\RR^2\to\HH_\CC^2$ sending $\partial P$ to $\zeta$ (note that $\zeta$ is the boundary of a smooth disc inside the real $4$-ball $Q$). A potential for the K\"ahler is obtained by choosing a basepoint $c\in\HH_\CC^2$ as in formula \eqref{potential}. The boundary of $\zeta$ is made of meridional curves and geodesics. Since each of these curves is contained in a real plane, it follows from formula \eqref{potential} that the integral of a K\"ahler potential along the curve vanishes when we choose the basepoint $c$ in the curve (say, we can take $c$ as the starting point of the curve). So, the contributions to the Toledo invariant come from the changes of basepoints which are explicitly given in \eqref{basepointchange}.

\begin{prop}\label{toledomod2}
Let\/ $\rho:G\to\PU$, $g_j\mapsto I_j$, be a representation satisfying the quadrangle conditions \ref{subsection quadrangle conditions}. Then\/ $\tau\equiv\displaystyle\frac{\Arg(\alpha_1\beta_1\gamma_1^{-1})}{\pi}\mod2$, where\/ $\tau$ stands for the Toledo invariant of\/ $\rho$.
\end{prop}

{\bf Proof.} Note that $\alpha_1\beta_1\gamma_1^{-1}$ is well-defined for $\rho$ because we assume the equality $I_3I_2I_1=1$ in $\SU$. We take the quadrangle of bisectors $\mathcal Q$ of $\rho$ described in Section~\ref{subsection quadrangle conditions}, the closed oriented curve $\zeta\subset\mathcal Q$ defined in \eqref{toledo's curve}, and the geodesic polygon $P\subset \HH_\RR^2$ defined in Section \ref{turnover definition}. Let $D\subset\HH_\CC^2$ be a disc with $\partial D=\zeta$ and let $\varphi:\HH_\RR^2\to\HH_\CC^2$ be a $\rho$-equivariant map such that $\varphi P=D$, $\varphi v_j=c_j$, and
$$\varphi e_1=h_1\cup m_1,\quad\varphi e_2=m_2\cup h_2,\quad\varphi e_3=h_3\cup m_3,\quad\varphi e_4=m_4\cup h_4$$
(see Figure \ref{fundamental domain and orbifold}). Then $\tau=\frac4{2\pi}\int_P\varphi^*\omega$, that is,
$$\tau=\frac2\pi\int_D\omega=\frac2\pi\int_{\partial D}P_{c_2}=\frac2\pi\sum_{j=1}^4\left(\int_{m_j}P_{c_2}+\int_{h_j}P_{c_2}\right),$$
where $P_{c_2}$ is a K\"ahler primitive with basepoint $c_2\in\HH_\CC^2$. The choice of the basepoint implies
$\int_{m_1}P_{c_2}=\int_{m_2}P_{c_2}=0$. The remaining integrals can be evaluated with the aid of the formula relating primitives based on distinct points:
$$J_1\coloneq \int_{h_2}P_{c_2}=\int_{h_2}(P_{c_2}-P_{c_3'})=\int_{h_2}\dd f_{c_2,c_3'}=$$
$$=\frac12\Arg\left(\frac{\langle c_2,c_3\rangle\langle c_3,c_3'\rangle}{\langle c_2,c_3'\rangle}\right)-\frac12\Arg\left(\frac{\langle c_2,c_3'\rangle\langle c_3',c_3'\rangle}{\langle c_2,c_3'\rangle}\right)=\frac12\Arg\left(\frac{\langle c_2,c_3\rangle\langle c_3,c_3'\rangle}{\langle c_2,c_3'\rangle}\right)-\frac\pi2;$$
$$J_2\coloneq \int_{h_3}P_{c_2}=\int_{h_3}(P_{c_2}-P_{c_3})=\int_{h_3}\dd f_{c_2,c_3}=\frac12\Arg\frac{\langle c_2,I_3c_3'\rangle\langle I_3c_3',c_3\rangle}{\langle c_2,c_3\rangle}-\frac12\Arg\frac{\langle c_2,c_3\rangle\langle c_3,c_3\rangle}{\langle c_2,c_3\rangle}=$$
$$=\frac12\Arg\left(\frac{\langle c_2,I_3c_3'\rangle\langle I_3c_3',c_3\rangle}{\langle c_2,c_3\rangle}\right)-\frac\pi2=\frac12\Arg\left(\gamma_1^{-1}\frac{\langle c_2,I_3c_3'\rangle\langle c_3',c_3\rangle}{\langle c_2,c_3\rangle}\right)-\frac\pi2.$$
Similarly, one obtains
$$J_3\coloneq \int_{m_3}P_{c_2}=\frac12\Arg\left(\beta_1\frac{\langle c_2,I_1^{-1}c_2\rangle\langle c_2,c_3'\rangle}{\langle c_2,I_3c_3'\rangle}\right)-\frac\pi2;\quad J_4\coloneq \int_{m_4}P_{c_2}=\frac12\Arg\left(\frac{\langle c_2,I_1^{-1}c_1'\rangle\langle c_1',c_2\rangle}{\langle c_2,I_1^{-1}c_2\rangle}\right)-\frac\pi2;$$ $$J_5\coloneq \int_{h_4}P_{c_2}=\frac12\Arg\left(\alpha_1\frac{\langle c_2,c_1\rangle\langle c_1,c_1'\rangle}{\langle c_2,I_1^{-1}c_1'\rangle}\right)-\frac\pi2;\quad J_6\coloneq \int_{h_1}P_{c_2}=\frac12\Arg\left(\frac{\langle c_2,c_1'\rangle\langle c_1',c_1\rangle}{\langle c_2,c_1\rangle}\right)-\frac\pi2.$$
We have $\tau=\displaystyle\frac2\pi\sum_jJ_j$. Calculating mod $2$, we multiply the arguments of every $\Arg$ function participating in the previous sum thus obtaining the result.\hfill$_\blacksquare$

\medskip

The Toledo invariant is in fact an invariant of representations $G\to\PU$; it does not depend on the quadrangle conditions. However, we only consider here representations satisfying such condition since discreteness is our main concern (see Section \ref{section discreteness}).

\begin{rmk}
There is another proof of Proposition \ref{toledomod2}.

Fix a point $x_0 \in \HH_\CC^2$. Consider the Dupont cocycle $\phi:\PU \times \PU\to \RR$ given by $$\phi(g,h) = \int_{[x_0,gx_0,ghx_0]} \frac{1}{\pi}\omega.$$ This cocycle defines an element in the Borel cohomology group $\mathrm H^2(\PU,\ZZ)$. See \cite[Lemma 5.2]{bui}, noting that we use the complex hyperbolic plane with holomorphic curvature~$-4$.

The map $f:\PU \to \RR/\ZZ$ given by $f(g) = \frac{\Arg(\langle g x_0, x_0 \rangle)}{2\pi}$
satisfies $\phi(g,h) = f(g)+f(h)-f(gh)$. Therefore, for an elliptic isometry $g$, the rotation number of $g$ is $\mathrm{rot}(g) = \frac{\Arg(\alpha)}{2\pi}$, where $\alpha$ is the eigenvalue corresponding to a negative eigenvector of $g$. To see this, we lift $f:\langle g \rangle \to \RR/\ZZ$ to $F:\langle g \rangle \to \RR$ satisfying $F(\mathrm{id})=0$ and compute $\mathrm{rot}(g) = \lim\limits_{n\to \infty} \frac{F(g^n)}{n}$.

Removing the three conic points of the sphere $\HH_\CC^1/G$, we obtain a $3$-punctured sphere $\Sigma_{0,3}$, which we can think of as a hyperbolic surface with finite volume. We denote by $r_1,r_2,r_3$ the parabolic elements in $\pi_1(\Sigma_{0,3})$  corresponding to loops around the punctures in the counterclockwise direction. Additionally, $\pi_1(\Sigma_{0,3})$ is a free group generated by $r_1,r_2$. The inclusion $\Sigma_{0,3} \hookrightarrow \HH_\CC^2/G$ induces the homomorphisms $\pi_1(\Sigma_{0,3}) \to G \to \PU$ given by $r_j \mapsto g_j^{-1} \mapsto I_j^{-1}$.
Consider a $G$-equivariant map $s:\HH_\CC^1 \to \HH_\CC^2$. We have
$$\frac{\tau}2 = \int_{\HH_\CC^1/G} \frac{1}{\pi}s^\ast\omega = \int_{\Sigma_{0,3}} \frac{1}{\pi}s^\ast\omega.$$

By \cite[Lemma 8.2]{biw}, we have
$$\int_{\Sigma_{0,3}} \frac{1}{\pi}s^\ast\omega = -\mathrm{rot}_\phi(I_1^{-1})- \mathrm{rot}_\phi(I_2^{-1})- \mathrm{rot}_\phi(I_3^{-1}) \mod 1,$$
$$\int_{\Sigma_{0,3}} \frac{1}{\pi}s^\ast\omega = \frac{\Arg(\alpha_1)}{2\pi}+ \frac{\Arg(\beta_1)}{2\pi}+\frac{\Arg(\gamma_1^{-1})}{2\pi} \mod 1,$$
$$\tau = \frac{\Arg(\alpha_1\beta_1 \gamma_1^{-1})}{\pi} \mod 2.$$

\end{rmk}

\newpage

\section{Explicit example with trivial Euler number}
\label{section:Explicit example with trivial Euler number}
\noindent
 We work with the standard model for the complex hyperbolic plane: consider the Hermitian form $$\langle z,w \rangle = - z_1\overline{w_1}+z_2\overline{w_2}+z_3\overline{w_3}$$ 
 on $\CC^3$ and the complex hyperbolic plane $\HH_\CC^2$, formed by points $[z_1,z_2,z_3]$ of the complex projective plane $\PP_\CC^2$ satisfying $-|z_1|^2+|z_2|^2+|z_3|^2<0$.

The parameters we set are 
$$n_1=n_2=n_3=5, \quad k_1=k_3=0, \quad k_2=2, \quad d=1,\quad s=2.4,\quad t=0.9,$$ 
taken from the Table \ref{table e/x =0, -1, 0.5 non-rigid}. They correspond to an example of a non-rigid representation whose associated disc orbibundle has trivial Euler number, as we will show.

The eigenvalues, as set in the beginning of Section \ref{section computational results}, are
{\footnotesize	
\begin{alignat*}{3}
&\alpha_1 = \exp\left(\frac{2 (n_1-k_1) \pi i}{3 n_1}\right), \quad
&&\alpha_2 = \exp\left(\frac{2 (n_1 -k_1 - 3) \pi i }{3 n_1}\right), \quad
&&\alpha_3 = \exp\left( \frac{2 (2 k_1 + n_1 +3) \pi i}{3 n_1}\right),\\ 
&\beta_1 = \exp\left(\frac{-2k_2 \pi i}{3 n_2}\right),\quad
&&\beta_2 = \exp\left(\frac{2 (-k_2 - 3) \pi i}{3 n_2}\right), \quad
&&\beta_3 = \exp\left( \frac{2 (2 k_2+3) \pi i)}{3 n_2}\right),\\
&\gamma_1^{-1} = \exp\left(\frac{2 (2 n_3 d - k_3) \pi i}{3 n_3}\right), \quad
&&\gamma_2^{-1} = \exp\left(\frac{2 (2 n_3 d - k_3 - 3) \pi i}{3 n_3}\right), \quad
&&\gamma_3^{-1} = \exp\left( \frac{2 (2 n_3 d + 2 k_3 + 3) \pi i)}{3 n_3}\right).
\end{alignat*}}%
In accordance with Remark \ref{remark alg}, define the following terms, which we use to compute $I_2$, $$\alpha_{ij} \coloneq  \alpha_i - \alpha_j, \quad \beta_{ij} \coloneq  \beta_i - \beta_j, \quad \gamma_{ij} \coloneq  \gamma_i - \gamma_j,$$
$$k\coloneq \frac1{\beta_{23}}\left(\gamma_1+\gamma_2+\gamma_3-\alpha_1(\beta_1+
\beta_2-\beta_3)-\beta_3(\alpha_2+\alpha_3)\right),$$
$$z\coloneq \displaystyle\frac{\beta_{13}}{\beta_{23}}\big(\alpha_{21}s+
\alpha_{31}t\big)+k,$$
$$M\coloneq \left[\smallmatrix\real\alpha_{21}&\real\alpha_{31}\\
\imag\alpha_{21}&\imag\alpha_{31}\endsmallmatrix\right].$$

The matrices $I_1, I_2$ are computed by the formulas
\begin{equation}
I_1=\left[\smallmatrix\alpha_1&0&0\\0&\alpha_2&0\\0&0&\alpha_3
\endsmallmatrix\right],\qquad
I_2=\left[\smallmatrix-v_1^2\beta_{23}+u_1^2\beta_{13}+\beta_3&v_1
\overline v_2\beta_{23}-u_2u_1\beta_{13}&v_1\overline
v_3\beta_{23}-u_3u_1\beta_{13}\\-v_1v_2\beta_{23}+u_1u_2\beta_{13}&
|v_2|^2\beta_{23}-u_2^2\beta_{13}+\beta_3&v_2\overline
v_3\beta_{23}-u_3u_2\beta_{13}\\-v_1v_3\beta_{23}+u_1u_3\beta_{13}&
\overline v_2v_3\beta_{23}-u_2u_3\beta_{13}&|v_3|^2\beta_{23}-u_3^2
\beta_{13}+\beta_3\endsmallmatrix\right],
\end{equation}
where we define $$u_1 \coloneq  \sqrt{1+s+t}, \quad u_2\coloneq \sqrt{s}, \quad u_3\coloneq \sqrt{t},$$
and $v_1,v_2,v_3$ is given by
$$v_2=\frac1{2v_1\sqrt{s(1+s+t)}}\big(-t|v_3|^2+(1+s+t)v_1^2+
s|v_2|^2+
i\sqrt\Delta\big),$$
$$v_3=\frac1{2v_1\sqrt{t(1+s+t)}}\big(-s|v_2|^2+(1+s+t)v_1^2+
t|v_3|^2-
i\sqrt\Delta\big).$$

In the above formula we use $$\Delta\coloneq 4v_1^2|v_2|^2s(1+s+t)-\big(-t|v_3|^2+(1+s+t)v_1^2+
s|v_2|^2\big)^2,$$
$$|v_2|^2\coloneq \imag(\alpha_{31}\overline z)/\det M,$$
$$|v_3|^2\coloneq -\imag(\alpha_{21}\overline z)/\det M,$$
$$v_1 \coloneq  \sqrt{1 + |v_2|^2+|v_3|^2}.$$

Note that in order to compute $v_2,v_3$ we must have the formulas defining $\Delta$, $|v_2|^2$, and $|v_3|^2$ non-negative. Once we have $I_1,I_2$, we define $I_3 \coloneq  I_1^{-1}I_2^{-1}$.

Using the described framework, we obtain the matrices
\begin{align*}
&I_1 =
\left[
\begin{smallmatrix}
   -0.5 + 0.8660254038i &\quad\,\,  0 &\quad\,\, 0 \\
    0 &\quad\,\,  0.6691306064 + 0.7431448255i &\quad\,\,  0 \\
    0 &\quad\,\,  0 &\quad\,\,  -0.9781476007 - 0.2079116908i \\
\end{smallmatrix}
\right],\\
&I_2=
\left[
\begin{smallmatrix}
     4.806086358 - 0.9638895607i &\,\, -4.070625549 - 0.2082686381i &\,\, -2.395307321 + 0.8226079388i\\
    3.683729972 + 0.0360116289i &\,\, -3.385353305 - 0.9825139687i &\,\, -1.430999775 + 0.3121734931i\\
    3.027105153 - 0.5413134207i &\,\, -2.276681096 - 0.06434808952i &\,\, -2.229750048 + 0.545144991i\\
\end{smallmatrix}
\right],\\
&I_3 =\left[\begin{smallmatrix}
    -1.568290333 - 4.644137659i &\,\, 1.873051971 + 3.172197922i &\,\, 1.044761403 + 2.892206673i \\
    2.569006381 - 3.164423233i &\,\, -1.535093338 + 3.173237958i &\,\, -1.475577052 + 1.734961052i \\
    -2.513993916 - 0.3066195865i &\,\, 1.464633515 + 0.007830170545i &\,\, 2.294366676 + 0.06964116243i \\
\end{smallmatrix}\right].
\end{align*}

They satisfy $I_3 I_2 I_1 = 1,\, I_1^5 =e^{-2\pi i/3},\, I_2^5 = e^{2\pi i/3},\, I_3^5 = e^{2\pi i/3}$ in $\SU$.

The positive eigenvectors $p_1,p_2,p_3$ for $I_1,I_2,I_3$ with respect to $\alpha_2,\beta_2,\gamma_2$ are
$$
p_1 =\left[
\begin{smallmatrix}
    0 \\
    1 \\
    0
\end{smallmatrix}\right],\quad
p_2 =\left[
\begin{smallmatrix}
   1.648291445 \\
    1.794969768 + 0.1092825754i \\
    0.6716832499 - 0.1784576984i \\
\end{smallmatrix}\right],  \quad
p_3 =\left[\begin{smallmatrix}
        1.662566196 + 0.7047909878i \\
    1.026420435 + 1.623659667i \\
    0.7556765638\\
\end{smallmatrix}\right],
$$
and $$p_4 \coloneq  I_1^{-1} p_2  = \left[\begin{smallmatrix}
    -0.8241457224 - 1.427462264i \\
    1.28228199 - 1.260798179i \\
    -0.6199019175 + 0.3142087697i \\
\end{smallmatrix}\right].
$$

Condition \eqref{Q1} ensuring that the complex geodesics $C_i=\PP(p^\perp) \cap \HH_\CC^2$ are pairwise ultraparallel is satisfied:

$$\ta(p_1,p_2)  \simeq 3.234, \quad \ta(p_2,p_4) \simeq 9.304, \quad \ta(p_2,p_3) \simeq 3.197.$$
Thus, we can connect the complex geodesics with segments of bisectors to construct the quadrangle of bisectors 
$$\mathcal Q = B[C_1,C_2] \cup B[C_2,C_3] \cup B[C_3,C_4] \cup B[C_4,C_1],$$
which will bound our fundamental domain.

Following the notation in \eqref{Q2} we have that the transversality and orientation conditions are met

$$1+2t^2s\epsilon_0- (\epsilon_0^2 t^2 + s^2 + t^2) \simeq 0.0251, \quad   1+2t^2s\epsilon_0- (\epsilon_0^2 s^2 + 2t^2) \simeq 3.476,  \quad \epsilon_1 \simeq -0.754,$$
$$1+2t'^2s\epsilon_0' - (\epsilon_0'^2 t'^2 + s^2 + t'^2) \simeq 0.303, \quad  1+2t'^2s\epsilon_0' -(\epsilon_0'^2 s^2 + 2t'^2) \simeq 3.575,  \quad \epsilon_1'\simeq-0.732$$
where $$t = \sqrt{\ta(p_1,p_2)}, \quad t'=\sqrt{\ta(p_3,p_2)}, \quad s=\sqrt{\ta(p_2,p_4)},$$ $$\epsilon_0+\epsilon_1i = \frac{\langle p_1,p_2 \rangle\langle p_2,p_4 \rangle\langle p_4,p_1 \rangle}{|\langle p_1,p_2 \rangle\langle p_2,p_4 \rangle\langle p_4,p_1 \rangle|}, \quad \epsilon_0'+\epsilon_1'i = \frac{\langle p_3,p_4 \rangle\langle p_4,p_2 \rangle\langle p_2,p_3 \rangle}{|\langle p_3,p_4 \rangle\langle p_4,p_2 \rangle\langle p_2,p_3 \rangle|}.$$

The first three equations of the six considered above mean that the triangle of bisectors $\Delta(C_1,C_2,C_4)$ is transversal and counterclockwise-oriented. The last three ones guarantee that the triangle of bisectors  $\Delta(C_3,C_4,C_2)$ is transversal and counterclockwise-oriented.

Now we must ensure that these two triangles are coupled suitably. More precisely, the conditions below ensure they are transversely adjacent. 

The conditions \eqref{q31}, \eqref{q32}, \eqref{q33} are satisfied via direct computation.
$$\left\vert \real\left( \frac{\langle p_3,p_1 \rangle \langle p_2,p_2 \rangle }{\langle p_3,p_2 \rangle \langle p_2,p_1 \rangle}\right) - 1 \right\vert -\sqrt{1 - \frac{1}{\ta(p_2,p_3)}}\sqrt{1 - \frac{1}{\ta(p_2,p_1)}} \simeq -0.129,$$
$$ \left\vert \real\left( \frac{\langle p_1,p_3 \rangle \langle p_4,p_4 \rangle }{\langle p_1,p_4 \rangle \langle p_4,p_3 \rangle}\right) - 1 \right\vert - \sqrt{1 - \frac{1}{\ta(p_4,p_1)}}\sqrt{1 - \frac{1}{\ta(p_4,p_3)}} \simeq -0.045,$$
$$\imag \left( \frac{\langle p_1,c_3 \rangle \langle c_3,p_2 \rangle}{\langle p_1,p_2 \rangle }\right) \simeq  0.713,$$
$$\imag \left( \frac{\langle p_4,c_3 \rangle \langle c_3,p_1 \rangle}{\langle p_4,p_1 \rangle }\right) \simeq 0.740.$$

The first of the four equations above ensures that $B[C_1,C_2]$ and $B[C_2,C_3]$ are transversal. The second equation does the same for $B[C_1,C_4]$ and $B[C_4,C_3]$. The last two equations guarantees that $c_3$, the negative eigenvector of $I_3$ corresponding to the eigenvalue $\gamma_1^{-1}$, is inside the sector $C_1$ defined by the triangle $\Delta(C_1,C_2,C_4)$. The explicit entries for $c_3$ are in Equation \eqref{eq: the value of c3}.

The above conditions imply that the quadrangle $\mathcal Q$ bounds a fundamental domain candidate for the action of $G(n_1,n_2,n_3)$.

The condition \eqref{q4} is satisfied by our choice of eigenvalues: 
$$\frac{\alpha_2}{\alpha_1} = \exp\left(\frac{-2\pi i}{n_1}\right), \quad\frac{\beta_2}{\beta_1} = \exp\left(\frac{-2\pi i}{n_2}\right), \quad\frac{\gamma_2^{-1}}{\gamma_1^{-1}} = \exp\left(\frac{-2\pi i}{n_3}\right).$$
This condition guarantees the tessellation around each vertex $C_i$ under the action of the group $G(n_1,n_2,n_3)$. From the Theorem \ref{teo: tessalation} we have a tessellation and, as a consequence, a complex hyperbolic orbifold. The structure of disc orbibundle is induced from the natural foliation by complex geodesics on the quadrangle, as it is discussed in Section \ref{section: Orbifold bundles and Euler number}.

Since all discreteness conditions are met, we have a disc orbibundle over $\SP^2(5,5,5)$.

Now we compute its Euler number via the formula outlined at the end of the Section \ref{subsection euler number} and to do so, we must compute the parameter $f$ following the steps in the Section \ref{subsection holonomy of the quadrangle}.

The triangle of bisector $\Delta(C_1,C_2,C_4)$ have holonomy $I$ with trace $|\tr(I)|=1.9244$ and the triangle of bisector $\Delta(C_3,C_4,C_2)$ have holonomy $J$ with trace $|\tr(J)| = 1.9159$. Therefore, both holonomies are elliptic. Since both triangles are counterclockwise oriented, they are L-elliptic, meaning that the holonomies rotate clockwise, and, as consequence, we can use any $z_1$ in our computation of the Euler number. The computation of these holonomies can be done using the formulas outlined in \cite[Lemma A.33]{discbundles}:
$$|\tr(I)| =\sqrt{2 (1+\epsilon_0)}\left(1-\frac{1+2 t_{12} t_{24} t_{41} \epsilon_0-t_{12}^2-t_{24}^2-t_{41}^2}{(t_{12}+1)(t_{24}+1)(t_{41}+1)}\right),$$
$$|\tr(J)| =\sqrt{2 (1+\epsilon_0')}\left(1-\frac{1+2 t_{34} t_{42} t_{23} \epsilon_0'-t_{34}^2-t_{42}^2-t_{23}^2}{(t_{34}+1)(t_{42}+1)(t_{23}+1)}\right),$$
where $t_{jk} =\sqrt{\ta(p_j,p_k)}$.

We choose
$z_1 = \left[\begin{matrix}1 & 0 & 1 \end{matrix}\right]^{\textrm{T}}$.
To construct the points $z_2,z_3,z_3',z_3'',z_4',z_5'$ described in the Section \ref{subsection holonomy of the quadrangle}, we need to compute reflections in the middle slice of each bisector segment defining the quadrangle. Middle slice for $B[C_1,C_2],B[C_2,C_3],B[C_3,C_4],B[C_4,C_1]$ have polar points $\mathfrak{m}_1,\mathfrak{m}_2,\mathfrak{m}_3,\mathfrak{m}_4$ given by
$$\mathfrak{m}_i = \frac{1}{\sqrt{2+2 \langle p_i,p_{i+1} \rangle}} \left( p_i + \frac{\langle p_i, p_{i+1}\rangle}{|\langle p_i, p_{i+1}\rangle|} p_{i+1} \right),$$
where we assume that $\langle p_i,p_i\rangle =1$. We have
\begin{alignat*}{2}
&\mathfrak{m}_1 = \left[\begin{smallmatrix}
    0.6967426286\\
    1.180669147 + 0.07188230542 i \\
    0.2839245175 - 0.07543513392 i
\end{smallmatrix}\right],\quad
&&\mathfrak{m}_2 = \left[\begin{smallmatrix}
    0.6209147951 + 0.9915030695 i \\
    0.2981752232 + 1.436781411 i \\
    0.3611646254 + 0.2914178221 i
\end{smallmatrix}\right],
\\&\mathfrak{m}_3 =\left[ \begin{smallmatrix}
0.3390781806 - 0.9381337155i \\
1.346757824 - 0.4086352516i \\
-0.05189725816 - 0.1078709913i
\end{smallmatrix} \right],\quad
&&\mathfrak{m}_4 = \left[\begin{smallmatrix}
    0.1746385463 - 0.6745010517 i \\
    1.182855316 \\
    -0.2799659635 - 0.08900927095 i
\end{smallmatrix}\right].
\end{alignat*}

We define $\mathfrak{m}_5$ to be the polar for the middle slice of the bisector $B[C_2,C_4]$:
$$\mathfrak{m}_5 = \frac{1}{\sqrt{2+2 \langle p_2,p_4 \rangle}} \left( p_2 + \frac{\langle p_2, p_4\rangle}{|\langle p_2, p_4\rangle|} p_4 \right).$$
$$
\mathfrak{m}_5 = \left[ \begin{smallmatrix}
0.2813718628 - 0.4967217499i \\
1.073868075 - 0.4119239074 i \\
0.02003505593 + 0.05125691089 i
\end{smallmatrix} \right].
$$
The reflection $R_i$ on the middle slice $\textbf{middle}_i$ is given by
$$x \mapsto -x + 2\langle x, \mathfrak{m}_i \rangle \mathfrak{m}_i$$

The points $z_2,z_3,z_3',z_3'', z_4',z_5'$, as described in Section \ref{subsection holonomy of the quadrangle} (see Figures \ref{meridionalcurvefigure} and \ref{fig: aux curves}), are given by
$$z_2=R_1z_1,\quad z_3=R_2 I_2z_2,\quad z_3'=R_2 z_2,\quad z_4'=R_5 z_2, \quad z_3''=R_3 z_4',\quad z_5'=R_4 z_4'.$$
\begin{alignat*}{2}
&z_2=\left[ \begin{smallmatrix}
-1.575255952 + 0.105117747i \\
-0.9856481166 + 0.1187792353i \\
-1.223037447 + 0.105117747i
\end{smallmatrix} \right],\quad
&&z_3=\left[ \begin{smallmatrix}
0.7841986264 + 0.2491608075i \\
0.5437649113 + 0.5867020173i \\
-0.04148256194 + 0.1882239487i
\end{smallmatrix} \right],\\
&z_3'=\left[ \begin{smallmatrix}
1.563008521 + 0.8399546493i \\
0.5593561185 + 0.9874435857i \\
1.340772987 + 0.2508448595i
\end{smallmatrix} \right],\quad
&&z_3''=\left[ \begin{smallmatrix}
-1.400737655 + 0.1542581335i \\
-0.8820775665 - 0.3309459089i \\
-1.028711762 + 0.2000690728i
\end{smallmatrix} \right],\\
&z_4'=\left[ \begin{smallmatrix}
1.755822676 + 0.8252418547i \\
0.07455119635 + 1.559221055i \\
1.142473276 - 0.1482026409i
\end{smallmatrix} \right],\quad
&&z_5'=\left[ \begin{smallmatrix}
-0.8556992719 - 0.6375475231i \\
0 \\
-1.042787938 - 0.2264539772i
\end{smallmatrix} \right].
\end{alignat*}

\begin{rmk}\label{cartan invariant} Given three pairwise distinct points $\xi_1,\xi_2,\xi_3 \in \HH_\CC^2 \cup \partial \HH_\CC^2$, we consider the number
$$\frac{\langle \xi_1,\xi_2 \rangle\langle \xi_2,\xi_3 \rangle\langle \xi_3,\xi_1 \rangle}{|\langle \xi_1,\xi_2 \rangle\langle \xi_2,\xi_3 \rangle\langle \xi_3,\xi_1 \rangle|}.$$
Note that it is unchanged under the change of representatives for the given points. Additionally, it never vanishes and its real part is the determinant of the Gram matrix
$(\langle \xi_i, \xi_j\rangle )$, which is negative when $\xi_1, \xi_2,\xi_3$ are not in the same complex geodesic and $0$ otherwise (see \cite[7.1 Cartan's angular invariant]{goldmanbook} for details). 

If $\xi_1,\xi_2,\xi_3$ are in the same complex geodesic, then, up to choosing representatives for $\xi_1,\xi_2,\xi_3$, we can assume that $\langle \xi_1, \xi_2 \rangle=1$ and $\xi_3 = (\xi_2-\xi_1)+ \omega (\xi_1+\xi_2)$ with $|\omega|=1$. In this representation, we obtain $\langle \xi_1,\xi_2 \rangle\langle \xi_2,\xi_3 \rangle\langle \xi_3,\xi_1 \rangle=-2i\imag(\omega) $. Therefore, if $\xi_1,\xi_2,\xi_3$ follow a cyclic order, then $\imag \langle \xi_1,\xi_2 \rangle\langle \xi_2,\xi_3 \rangle\langle \xi_3,\xi_1 \rangle$ is negative, otherwise, it is positive.

Thus, we can write
$$o(\xi_1,\xi_2,\xi_3)=\begin{cases} 
      1, \text{ if } \imag \langle \xi_1,\xi_2 \rangle\langle \xi_2,\xi_3 \rangle\langle \xi_3,\xi_1 \rangle>0,\\
      0, \text{ if } \imag \langle \xi_1,\xi_2 \rangle\langle \xi_2,\xi_3 \rangle\langle \xi_3,\xi_1 \rangle<0.
   \end{cases}$$
\end{rmk}
\medskip

Since
\begin{align*}
&\langle z_3',z_3 \rangle\langle z_3,I_3z_3 \rangle\langle I_3z_3,z_3' \rangle \simeq -0.387 i,\\
&\langle z_3',I_3 z_3 \rangle\langle I_3 z_3,z_3'' \rangle\langle z_3'',z_3' \rangle \simeq 0.163 i,\\
&\langle z_1,I_1^{-1} z_1 \rangle\langle I_1^{-1} z_1,z_5' \rangle\langle z_5',z_1 \rangle \simeq -0.457i,
\end{align*}
we have
$$o(z_3',z_3,I_3z_3)=0, \quad o(z_3',z_3,I_3z_3)=1, \quad o(z_1,I_1^{-1}z_1,z_5') =0,$$
and, as a consequence,
$$f=o(z_3',z_3,I_3z_3)+o(z_3',I_3z_3,z_3'')-o(z_1,I_1^{-1}z_1,z_5')=1.$$

Therefore, following Proposition \ref{prop Euler number}, we obtain the Euler number
$$e=f-\frac{k_1+1}{n_1}-\frac{k_2+1}{n_2}-\frac{k_3+1}{n_3} =1-\frac{1+3+1}5=0,$$
because
$$\frac{\alpha_3}{\alpha_1} = \exp\left(\frac{(k_1+1)\pi i}{n_1}\right), \quad\frac{\beta_3}{\beta_1} =  \exp\left(\frac{(k_2+1)\pi i}{n_2}\right),\quad \frac{\gamma_3^{-1}}{\gamma_1^{-1}} =  \exp\left(\frac{(k_3+1)\pi i}{n_3}\right).$$

The Toledo invariant can be computed from Proposition \ref{toledomod2}. Note that 
$$\alpha_1 \beta_1 \gamma_1^{-1} = \exp\left(-\frac{4\pi i}{15}\right),$$
$$\tau \equiv \frac1{\pi}\arg(\alpha_1 \beta_1 \gamma_1^{-1}) \equiv -\frac{4}{15} \mod 2.$$

On the other hand, $\chi =  - 2/5$. The Toledo rigidity states that $|\tau| \leq |\chi|$ and, as consequence, we must have $\tau = -4/15$.

Observe as well that $3\tau = 2\chi+2 e = - 4/5$. Alternatively, we can use the formulas in the proof of Proposition \ref{toledomod2} to compute $\tau$ directly, without using Toledo rigidity. Let us do so.

The negative eigenvectors $c_1,c_2,c_3$ for $I_1,I_2,I_3$ associated to the eigenvalues $\alpha_1,\beta_1,\gamma_1$ are

\begin{equation}\label{eq: the value of c3}
c_1=\left[ \begin{smallmatrix}
1 \\
0 \\
0
\end{smallmatrix} \right],\quad
c_2=\left[ \begin{smallmatrix}
2.073644135 \\
1.549193338 \\
0.9486832981
\end{smallmatrix} \right],\quad
c_3=\left[ \begin{smallmatrix}
2.222238961 + 1.252946098i \\
0.7345344135 + 1.665521923i \\
1.481457337
\end{smallmatrix} \right]
\end{equation}

Following the proof of Proposition \ref{toledomod2}, we must consider the points
$c_1' = R_1 c_2$, $c_3' = R_2 c_2$. 

$$c_1' = \left[\begin{smallmatrix}
-1.162802309 - 0.05545449422i \\
0\\
-0.583516738 - 0.1212131456i
\end{smallmatrix} \right],\quad c_3' = \left[\begin{smallmatrix}
-1.788443315 - 1.5119981 i \\
-0.5547823714 - 1.654060537 i\\
-1.03745005 - 0.6038753169 i
\end{smallmatrix} \right].$$

\begin{rmk} Consider two distinct points $\xi_1,\xi_2 \in \HH_\CC^2$. As noted in the Remark \ref{cartan invariant}, for $x \in \HH_\CC^2$ we have $$\frac{\langle \xi_1,x\rangle\langle x,\xi_2 \rangle\langle \xi_2,\xi_1 \rangle}{|\langle \xi_1,x\rangle\langle x,\xi_2 \rangle\langle \xi_2,\xi_1 \rangle|} \not \in \RR_{\geq 0}.$$

Now, consider the branch $\Arg:\CC \setminus \RR_{\geq 0} \to (0,2\pi)$ of the argument function. We conclude that the map
$$x \mapsto \Arg\left (\frac{\langle \xi_1,x\rangle\langle x,\xi_2 \rangle}{\langle \xi_1,\xi_2 \rangle}\right)$$
is well-defined and it is the function used to compute the Toledo invariant in the Proposition \ref{toledomod2} (now without thinking of it as a multivalued function).
\end{rmk}

Thus, from Proposition \ref{toledomod2} we have $\tau = \frac2{\pi}(J_1+J_2+J_3+J_4+J_5+J_6)$, which becomes
   \begin{align*}
    \tau=&\frac1 \pi \Arg\left(\frac{\langle c_2,c_3\rangle\langle c_3,c_3'\rangle}{\langle c_2,c_3'\rangle}\right)+\frac1 \pi\Arg\left(\gamma_1^{-1}\frac{\langle c_2,I_3c_3'\rangle\langle c_3',c_3\rangle}{\langle c_2,c_3\rangle}\right)+\frac1 \pi\Arg\left(\beta_1\frac{\langle c_2,I_1^{-1}c_2\rangle\langle c_2,c_3'\rangle}{\langle c_2,I_3c_3'\rangle}\right)+
    \\&+ \frac1 \pi\Arg\left(\frac{\langle c_2,I_1^{-1}c_1'\rangle\langle c_1',c_2\rangle}{\langle c_2,I_1^{-1}c_2\rangle}\right)+\frac1 \pi\Arg\left(\alpha_1\frac{\langle c_2,c_1\rangle\langle c_1,c_1'\rangle}{\langle c_2,I_1^{-1}c_1'\rangle}\right)+\frac1 \pi\Arg\left(\frac{\langle c_2,c_1'\rangle\langle c_1',c_1\rangle}{\langle c_2,c_1\rangle}\right) - 3 .
    \end{align*}

    And we obtain $\tau =-0.266666666\ldots = -4/15$.

    Therefore, we can compute the Toledo in two ways. It's also possible to prove that the identity $3\tau = 2e+2\chi$ holds for the described type of construction using quadrangles of bisectors (see the preprint \cite{bot2}). With that, we can also derive $e=0$ from the Toledo invariant.

\end{document}